\documentclass[11pt]{amsart}
\usepackage{graphicx,amssymb,amsmath, latexsym,cite, amscd,stmaryrd}
\usepackage{amsthm,amsrefs}  
\usepackage{enumerate,lscape}      
\usepackage{float,xcolor}  
\usepackage{fancyhdr, lastpage}     
\usepackage{libertine}             
 \usepackage{url}        
\usepackage{verbatim} 
\usepackage{algorithm}    
\usepackage[noend]{algpseudocode} 
\usepackage{longtable} 
 
\algblock[TryCatchFinally]{try}{endtry}
\algcblock[TryCatchFinally]{TryCatchFinally}{finally}{endtry}  
\algcblockdefx[TryCatchFinally]{TryCatchFinally}{catch}{endtry}
	[1]{\textbf{catch} #1} 
	{\textbf{end try}}

\usepackage[all,cmtip]{xy}

\newcommand{\soc}{{\rm soc}}

\usepackage{algorithm}
\usepackage[noend]{algpseudocode}

\algblock[TryCatchFinally]{try}{endtry}
\algcblock[TryCatchFinally]{TryCatchFinally}{finally}{endtry}
\algcblockdefx[TryCatchFinally]{TryCatchFinally}{catch}{endtry}
	[1]{\textbf{exception:} #1}
	{\textbf{end try}}

\newtheoremstyle{mythm}                   
{6pt}
{6pt}
{\it}
{}
{\bf}
{.}
{.5em}
{}

\newtheoremstyle{mydef}                   
{6pt}
{6pt}
{}
{}
{\bf}
{.}
{.5em}
{}

\newtheoremstyle{myrem}                   
{6pt}
{6pt}
{}
{}
{\bf}
{.}
{.5em}
{}

\theoremstyle{mythm}      
\newtheorem{theorem}{Theorem}[section]
\newtheorem{proposition}[theorem]{Proposition}
\newtheorem{lemma}[theorem]{Lemma}
\newtheorem{corollary}[theorem]{Corollary}

\theoremstyle{mydef}      
\newtheorem{definition}[theorem]{Definition} 
\newtheorem{example}[theorem]{Example}

\theoremstyle{myrem}
\newtheorem{remark}[theorem]{Remark}
\newtheorem{notation}[theorem]{Notation}
\numberwithin{equation}{section}

 \usepackage[a4paper,right=2.65cm, left=2.65cm, top=3.5cm, bottom=3.5cm, marginpar=1.6cm]{geometry}  

\newcounter{ithmcount}
\newenvironment{iprf}{\begin{list}{{\rm
	\alph{ithmcount})}}{\usecounter{ithmcount}\labelwidth-5pt
      \leftmargin0pt \topsep3pt \itemsep1pt \parsep2pt}}{\qedhere\end{list}}

\newenvironment{ithm}{\begin{list}{{\rm \alph{ithmcount})}}{\usecounter{ithmcount}\labelwidth18pt
      \leftmargin18pt \topsep3pt \itemsep1pt \parsep2pt}}{\end{list}}

\allowdisplaybreaks

\reversemarginpar

\newcommand{\Aut}{{\rm Aut}}
\newcommand{\Hom}{{\rm Hom}}
\newcommand{\GL}{{\rm GL}}
\newcommand{\GF}{{\rm GF}}

\newcommand{\diag}{{\rm diag}}
\newcommand{\gnu}{{\rm gnu}}
\newcommand{\gnuFF}{{\rm gnu}_{\Phi}}  
\newcommand{\rk}{{\rm rk}}
\newcommand{\revnew}[1]{{#1}}

\renewcommand{\geq}{\geqslant} 
\renewcommand{\leq}{\leqslant}

\begin{document}

\vspace*{-0.8cm}

\title{The number of groups of  cubefree order}

\author[H. Dietrich]{Heiko Dietrich}
\author[D. Jefferies]{David Jefferies}
\address{School of Mathematics, Monash University,  
Clayton VIC 3800, Australia}
\email{heiko.dietrich@monash.edu;\newline david.jefferies@monash.edu, davidjjefferies2@gmail.com}
\thanks{Dietrich was supported by a Universities Australia Grant in the Australia-Germany Joint Research Cooperation Scheme (\emph{The SmallGroups library: new classifications via symbolic computation}). He thanks the Isaac Newton Institute for Mathematical Sciences, Cambridge, for support and hospitality during the programme \emph{Latest trends in algebra and geometry}, where some of the work on this paper was undertaken; this work was supported by EPSRC grant EP/V521929/1. Jefferies was supported by an Australian RTP scholarship. The authors thank Max Horn for helpful discussions and an improvement to Cubefree \cite{cubefreeGAP} that facilitated the extensive cross-checking of our implementation.}

\date{\today}   

\begin{abstract}  Generalising H\"older's  classical group enumeration  for squarefree orders  (1895), we provide an exact formula for the number of isomorphism types of groups of a given cubefree order. After more than $130$ years, this is the first such formula that covers significantly more orders than the squarefree ones ($\approx\!83\%$ versus $\approx\!61\%$ of all integers).  Like  H\"older's formula, ours is \emph{combinatorial}: it can be evaluated from the prime
  factorisation of the order by arithmetic operations and table look-ups, without constructing a single group. The structure of our formula leads to  counting formulas for natural subclasses of cubefree groups,  with applications in computational group theory. It is also explicit enough to yield new asymptotic results.  Blackburn et al.\ (2007) conjectured that the number $\gnu(n)$ of groups of cubefree order $n$ satisfies $\gnu(n)<n^2$.
  We show that $\gnu(n)\leq n^{2+o(1)}$, which improves the bound $\gnu(n)<n^8$ recorded in their survey, and we prove that the exponent $2$ is best possible, that is,  $\gnu(n)\geq n^{2-o(1)}$ for infinitely many cubefree $n$. Lastly, we show that a much stronger form of the conjecture holds for almost every cubefree order, namely,   $\gnu(n)\leq (\log n)^{(\log\log n)^{O(1)}}$.
\end{abstract}

\thispagestyle{empty}

\maketitle

\vspace*{-0.5cm}

\enlargethispage{0.3cm}
\section{Introduction}\label{sec_intro}

\thispagestyle{empty}
\noindent Following  Conway et al.\ \cite{gnu}, we use $\gnu(n)$ (\underline{g}roup \underline{nu}mber) to denote the number of isomorphism types of groups of order $n$.  Determining $\gnu(n)$, and classifying the groups it counts, is a classical theme in group theory; we refer to the book of Blackburn et al.\ \cite{blackburn} for comprehensive background and references. Exact values are known only for restricted classes of orders. If $p$ is a prime, then the numbers of groups of orders $p,p^2,p^3,p^4$ are $1$, $2$, $5$, and $14$ (for $p=2$) or $15$ (for $p>2$), respectively, and  $\gnu(p^5)=2p+2\gcd(p-1,3)+\gcd(p-1,4)+61$ for $p\geq 5$. For larger powers of $p$, asymptotic formulas exist, and Higman's PORC Conjecture claims that, for a fixed $k$, the number of groups of order $p^k$ is a function of $p$ that is polynomial on residue classes. At the other end of the spectrum, H\"older \cite{holder} classified and enumerated the groups of squarefree order: if $n$ is squarefree, then
\begin{equation}\label{eq_holder}
\gnu(n)=\sum\nolimits_{d\mid n}\ \prod\nolimits_{r\mid (n/d)}\frac{r^{w(r)}-1}{r-1},
\end{equation}
where $d$ runs over the divisors of $n$, $r$ over the prime divisors of $n/d$, and $w(r)$ is the number of primes $p\mid d$ with $p\equiv 1\bmod r$; see also \cite[Proposition 21.5]{blackburn}. Murty-Murty \cite{murty} extended these results and  provided a counting formula and asymptotic results for groups all of whose Sylow subgroups are cyclic (C-groups). Their formula applies to every order $n$, but counts only the C-groups of order $n$.

The next natural class of orders beyond the  squarefree ones is that of cubefree integers.  By \cite[(2)]{density}, approximately $83\%$  of all integers are cubefree (and $61\%$ are squarefree). Groups of cubefree order were first studied by Taunt \cite{taunt}. Dietrich-Eick \cite{cf} developed a construction algorithm and Qiao-Li \cite{cf2} provided further structure results. An algorithmic isomorphism test for cubefree groups was presented by Dietrich-Wilson \cite{cfisom}. So far, no exact formula for the number of groups  of arbitrary cubefree order has been known. The best general bound for cubefree $n$ is $\gnu(n)<n^8$ and it is conjectured  that $\gnu(n)<n^2$; see \cite[Section 24.1]{blackburn}. Kumar-Venkataraman  recently  proved that there are at most $n^4$ solvable groups of cubefree order $n$, see \cite[Theorem 1.2]{KV}.

In this paper we establish the first exact counting formula for $\gnu(n)$ with $n$ cubefree, see Section~\ref{sec_intromain}. Our formula is \emph{combinatorial} in the sense of Definition~\ref{def_CF}: it can be evaluated from the prime factorisation of $n$ by integer arithmetic (mainly divisibility conditions) and look-ups in fixed tables, without constructing a single group. This is similar in nature to Higman's  PORC Conjecture for $\gnu(p^k)$.

Restricted to squarefree orders, our formula collapses to H\"older's formula \eqref{eq_holder}, see Remark \ref{rem_redsq}.  A \emph{closed} expression of that shape is unlikely to exist for general cubefree orders, see Remark \ref{rem_closed} and the discussion in \cite[Section~21.4]{blackburn}: already the known count of the groups of order $p^2qr$ in \cite[Theorem~2.1]{sot} has twenty summands, most of them divisibility indicators recording congruences between the primes involved. Some complexity is therefore to be expected. Indeed, our formula is more technical than \eqref{eq_holder}, and evaluating it requires some arithmetic calculations best done by a computer.

Nevertheless, the formula is explicit enough to allow new asymptotic results, established using estimates from analytic number theory. We prove that \cite[Conjecture 21.16]{blackburn} holds up to a factor $n^{o(1)}$, that is, we confirm the conjectured exponent. 

\vspace*{0.05cm}

{\bf Theorem A.} {\it If $n$ is cubefree, then $\gnu(n)\leq  n^{2+o(1)}$.}

\vspace*{0.1cm}

This improves the general bound $\gnu(n)<n^8$ and, asymptotically, the bound of $n^4$ for the number of solvable groups of cubefree order $n$ proved in \cite{KV}. In fact, using results of Dietrich-Wilson~\cite{focs}, we prove that a much stronger form of  \cite[Conjecture~21.16]{blackburn} holds for almost every cubefree~$n$.

{\bf Theorem B.} {\it For almost all cubefree $n$ we have $\gnu(n) \leq (\log n)^{(\log\log n)^{O(1)}}$, and so $\gnu(n) \leq  n^{o(1)}$.}

\vspace*{0.04cm} 

Our results also allow systematic experiments: we have verified that $\mathrm{gnu}(n) \leqslant n$ for every cubefree $n \leqslant 10^8$. This does not persist: the next result shows that $\gnu(n)>n$ for infinitely many $n$, and that the exponent $2$ in Theorem A cannot be improved; see Example \ref{ex_gnultn} for a $16$-digit $n$.

\vspace*{0.04cm}

{\bf Theorem C.} {\it There is an infinite set $\mathcal{N}$ of cubefree integers such that $\gnu(n)\geq n^{2-o(1)}$ for $n\in \mathcal{N}$.}

Our formula does more than just evaluate $\gnu(n)$: each of its outer summations is indexed by a structural invariant of the groups being counted, so truncating a summation counts the number of groups of a given order with  prescribed invariants. In computational group theory, this is exactly the data required to partition these groups into \emph{clusters} for efficient group construction and identification: this underpins many of the algorithms for the  SmallGroups library \cite{smallgroups,mill} and its recent extensions \cite{sot,29,p5,cgroups}. We discuss this in Section \ref{sec_app}.

\enlargethispage{1cm}

\subsection{Counting formulas}\label{sec_intromain} We present exact formulas for the number of groups of cubefree order $n$, see Theorem \ref{thm_mainodd} for odd $n$ and Theorem \ref{thm_maineven} for general $n$; the (easier) odd case is summarised in Theorem D below. Our formulas rely on the following structure results, which we  recall with references in Section~\ref{sec_struct}. Every group $G$ of cubefree order $n$ can be decomposed as $G=A\times L$, where $A$ is trivial or simple, and $L$ is solvable; if $A$ is simple, $A\cong {\rm PSL}_2(p)$ for a prime $p>3$. The Frattini subgroup $\Phi(L)$ of $L$ is nilpotent and its order divides the order of $L_\Phi=L/\Phi(L)$. Thus, $\Phi(L)$ is squarefree, hence cyclic, and so its isomorphism type is determined by its order. The group $L_\Phi$ is Frattini-free, that is, $\Phi(L_\Phi)=1$. It is known that $L_\Phi\cong K\ltimes S$ where $S=\soc(L_{{\Phi}})$ is the socle of $L_\Phi$ and $K\leq \Aut(S)$. The isomorphism type of $L$ is uniquely determined by $L_\Phi$ and the order  $|\Phi(L)|$.

We write $\gnuFF(n,\ell)$ for the number of solvable Frattini-free groups of order $n$ with socle of order~$\ell$. Let $\mathcal{Q}(n)$ be the product of all primes $p$ such that $p^2\mid n$; the divisors of $\mathcal{Q}(n)$ are the candidates for the order of the Frattini subgroup. The possible orders of simple factors are recorded by  $\mathcal{S}(n)=\{p(p^2-1)/2 :  p>3\text{ a prime and }p(p^2-1)/2\mid n\}\cup\{1\}$. The starting point of our counting formulas is the following expression for $\gnu(n)$ with $n$ cubefree; we prove it in  Proposition \ref{propredFF}:
\[\gnu(n) = \sum\nolimits_{a\in \mathcal{S}(n)}\ \sum\nolimits_{d\mid \mathcal{Q}(n/a)}\ \sum\nolimits_{\ell \mid n/(ad)} \gnuFF(n/(ad),\ell).\]

\smallskip

{\bf Odd order formula.}  Theorem D below restates Theorem \ref{thm_mainodd}. The displayed formula uses  notation that will be developed in Section \ref{secOdd}; below the theorem we briefly comment on the notation, provide signposts, and explain the underlying structure. For an in-depth discussion we refer to  that section.

\smallskip

{\bf Theorem D.} {\it The number of isomorphism types of groups of odd cubefree order $n$ is
\[\gnu(n)=\sum_{d\mid \mathcal{Q}(n)}\ \sum_{\ell\mid (n/d)}\ \sum_{L\in\mathcal{A}(n/(\ell d))}\ \sum_{\mathcal{U}}\ \frac{1}{|\Theta(\mathcal{U})|}\sum_{\theta\in\Theta(\mathcal{U})}\ \prod_{q\mid(n/(\ell d))}\big|\mathcal{K}_q(\mathcal{U},L)^{\theta}\big|,\]
where $\mathcal{U}$ runs over the projection tuples in $\prod_{i=1}^m\mathcal{U}(p_i,e_i,L)$ where  $\ell=p_1^{e_1}\ldots p_m^{e_m}$, and $q$ runs over all prime divisors of $n/(\ell d)$; the cardinalities $\big|\mathcal{K}_q(\mathcal{U},L)^{\theta}\big|$ are determined in Lemma \ref{lemFix}.}

\smallskip

The structure of the formula is as follows, where $G$ denotes a group of odd cubefree order $n$; recall that groups of odd order are solvable by the Odd Order Theorem \cite{oddorder}. 
\begin{iprf}
\item[$\bullet$] The two outer sums run over the possible orders $d$  and $\ell$ of the Frattini subgroup $\Phi(G)$ and the socle $S=\soc(G/\Phi(G))$, respectively. If the prime factorisation of $\ell$ is as in the theorem, then $S\cong C_{p_1}^{e_1}\times\ldots\times C_{p_m}^{e_m}$, where each $C_{p_i}$ is a cyclic group of order $p_i$; recall that each $e_i\in\{1,2\}$. 
\item[$\bullet$]  The socle complement $K$ is an abelian subgroup of $\Aut(S)=\prod_{i=1}^m \GL_{e_i}(p_i)$; we note that $K$ is not necessarily abelian for even $n$, which makes the even case more complicated. We  write $\mathcal{A}(n/(\ell d))$ for the set of isomorphism types of abelian groups of order $n/(\ell d)$, each  specified by a list of abelian invariants; the possible complements are among them. In the third sum, $L$ runs over this set and fixes the isomorphism type of the complement.
\item[$\bullet$] The projection of a complement $K\leq \Aut(S)$ into the $i$-th factor defines a subgroup $U_i\leq \GL_{e_i}(p_i)$, and since we fix $K\cong L$, each $U_i$ is a quotient of $L$. A canonical representative for each conjugacy class of such subgroups is stored in  $\mathcal{U}(p_i,e_i,L)$, see \eqref{eq_defUpel}, and represented by a list of numbers, see Remark \ref{rem_savegroup}c). The fourth sum runs over all tuples $\mathcal{U}=(U_1,\ldots,U_m)$ with  $U_i\in\mathcal{U}(p_i,e_i,L)$.  For each such $\mathcal{U}$, the set $\mathcal{K}(\mathcal{U},L)$ denotes the set of all subgroups $K\leq \Aut(S)$ that are isomorphic to $L$ and have projections $U_i\leq \GL_{e_i}(p_i)$ for each $i$. Each $U_i$ of the tuple $\mathcal{U}$ also comes with an automorphism $\theta_{U_i}$ of order dividing $2$; all these automorphisms generate the elementary abelian $2$-group $\Theta(\mathcal{U})$, see  \eqref{def_ThetaU}. Its elements are vectors of signs that record for which $i$ the component $\theta_{U_i}$ acts nontrivially.
\item[$\bullet$] The summands of $\sum_{\mathcal{U}}$ then count the number of $\Theta(\mathcal{U})$-orbits on $\mathcal{K}(\mathcal{U},L)$:  this count is evaluated with the Cauchy-Frobenius Lemma and reduced to calculations with the sets $\mathcal{K}_q(\mathcal{U},L)$ formed by the Sylow $q$-subgroups of the groups in $\mathcal{K}(\mathcal{U},L)$, see \eqref{eq_defKq}, where $q$ runs over the prime divisors of $n/(\ell d)$. The cardinalities of $\mathcal{K}_q(\mathcal{U},L)$ and of the set of fixed points $\mathcal{K}_q(\mathcal{U},L)^\theta$ under $\theta\in\Theta(\mathcal{U})$ are given by formulas that only involve the data representing $\mathcal{U}$, see Lemmas \ref{lemSizes} and~\ref{lemFix}.
\end{iprf}
While the description above is group-theoretic, we explain in the main text how $\gnu(n)$ can be calculated by arithmetic operations and table look-ups alone; cf.\ our implementation and Section~\ref{secComb}.

In Remark \ref{rem_oddsplit} we recall some Erd\H{o}s-P\'alfy \cite{erdos} factorisations $n=ab$ for odd $n$ such that  $\gnu(n)=\gnu(a)\gnu(b)$. This often leads to significant practical improvements, but only for odd $n$.

\smallskip  
 
{\bf General formula.}
Our formula for general cubefree $n$ is  more involved, but it follows the same pattern as the odd case; we therefore do not display it here and refer to Theorem~\ref{thm_maineven} and Section~\ref{sec_maineven} for a discussion. Our theorem genuinely generalises H\"older's result for squarefree orders; see Remark~\ref{rem_redsq}.

While our formula for $\gnu(n)$ is not \emph{closed} as in  \eqref{eq_holder}, it is exact, combinatorial, and highly structured. In particular, it can easily be evaluated by a straightforward calculation on a computer, and we explain in  Section \ref{secComp} how this can  be done  more efficiently than a naive summand-by-summand calculation. We have implemented our formula for the computer algebra system GAP \cite{gap}, and our implementation\label{pagelink}\footnote{See {\scriptsize\url{https://github.com/heikodietrich/NumberCubefreeGroups}} for the pre-release code.} will be included in the next release of the package Cubefree \cite{cubefreeGAP}.  We stress that our proofs are independent of any computation; the implementation is used only for evaluation and  cross-checking. For example, enumerating the  $22{,}855{,}136$ groups of order $2^2.3^2.5^2.11^2.41^2.67^2$ and the  $92{,}432{,}340$ groups of order $5^2.7^2.13^2.17^2.19^2.23^2.31^2.37^2.43^2.47^2.53^2.59^2.67^2$ takes less than a second each. Our extensive cross-checking against existing data gives additional  confidence that the technical case distinctions in our proofs are all correct, see Section~\ref{secXC}.

\enlargethispage{0.7cm}

{\bf Structure of this paper.} In Section \ref{sec_prelim} we provide the required background details, in particular on cubefree groups. We derive the formula for odd orders (Theorem D) in Section \ref{secOdd}, and we consider even orders in Section \ref{secEven}.  In Section \ref{sec_asym}, we prove Theorems~A, B, and~C. In Section \ref{secComp}, we comment on our implementation and how we have cross-checked its validity  against existing data. In Section~\ref{sec_app}, we discuss applications to group construction and identification. A few technical results are deferred to Appendix \ref{secKH}; these are required to prove that our formula for the even case is combinatorial. A summary of the notation used in this paper is collected in Appendix~\ref{secNot}.

\section{Preliminary results}\label{sec_prelim}
\noindent 
The following two subsections repeat some of the results and definitions mentioned in the introduction; we keep them here for completeness and easy referencing, see also Appendix~\ref{secNot}.

All groups are finite, and integers $p$ and $q$ usually denote primes. If $n\geq 1$ is an integer, then $C_n$ denotes the cyclic group of order $n$. If $G$ is a group, then $G_p$ denotes a Sylow $p$-subgroup and $\rk(G)$ denotes the size of a smallest possible generating set. We write a split extension $G$ of $N\unlhd G$ with complement $K$ as  $G=K\ltimes N$. Generalising the usual definition of a Kronecker delta, for a logical statement $P$ we write $\Delta_P=1$ if $P$ holds, and $\Delta_P=0$ otherwise. For a prime power $p^k$ we denote by $\GF(p^k)$ the finite field with $p^k$ elements, and denote by $\GF(p^k)^\times$ its (cyclic) multiplicative group. Superscripts $\pm$ will later often denote data associated with $\pm 1$-eigenspaces, such as dimensions.

\subsection{The structure of cubefree groups}\label{sec_struct}
We recall the relevant structure results for cubefree groups from \cite{cf,cfisom,cf2}. Let $G$ be a nontrivial group of cubefree order. Then $G=A\times L$ where $A=1$ or $A={\rm PSL}_2(p)$ for a prime $p>3$ with $p\pm 1$ cubefree, and $L$ is solvable. Moreover, the Frattini subgroup $\Phi(L)$ of $L$ is nilpotent and squarefree (hence cyclic) and $|\Phi(L)|$ divides the order of the Frattini quotient $L_\Phi=L/\Phi(L)$. The group $L_\Phi$ is Frattini-free (that is, it has a trivial Frattini subgroup) and can be decomposed as $L_\Phi=K\ltimes (B\times C)$ where $\soc(L_\Phi)=B\times C$ is the socle of $L_\Phi$ and $K\leq \Aut(B\times C)$; here $B=C_{p_1}\times\ldots\times C_{p_s}$ and $C=C_{p_{s+1}}^2\times\ldots\times C_{p_m}^2$ for pairwise distinct primes $p_1,\ldots,p_m$. Work of Gasch\"utz implies that two such solvable Frattini-free groups $K\ltimes (B\times C)$ and $\tilde K\ltimes (B\times C)$ with $K,\tilde K\leq\Aut(B\times C)$ are isomorphic if and only if $K$ and $\tilde K$ are conjugate in $\Aut(B\times C)$, see \cite[Theorem~9]{cf}. Moreover, we recall that  the isomorphism type of $L$ is determined by $L_\Phi$ and $|\Phi(L)|$ in the following sense: if $F$ is a solvable Frattini-free group and $M$ is a cyclic group such that $|F||M|$ is cubefree and $|M|$ divides $|F|$, then there is a unique isomorphism type of extension $E$ of $F$ by $M$ such that $\Phi(E)\cong M$ and $E/\Phi(E)\cong F$, see \cite[Theorem 11]{cf}.

\subsection{Counting formulas: a first reduction}
Throughout, we use the following notation.

\begin{definition}\label{def_setup} Let $n\geq 1$ be an integer.
  \begin{ithm}
  \item  Write $\gnu(n)$ for the number of isomorphism classes of groups of order $n$.
  \item Write $\gnuFF(n)$ for the number of isomorphism classes of solvable Frattini-free groups of order~$n$.
  \item Write $\gnuFF(n,\ell)$ for the number of isomorphism classes of solvable Frattini-free groups of order $n$ with socle order $\ell$.  
  \item Write $\mathcal{Q}(n)$ for the product of all primes $p$ such that $p^2\mid n$.
  \item Write $\mathcal{S}(n)=\{p(p^2-1)/2 :  p>3\text{ a prime and }p(p^2-1)/2\mid n\}\cup\{1\}$.
    \item We write $\mathcal{A}(n)$ for the set of isomorphism types of abelian groups of order $n$.
  \end{ithm}
\end{definition}

Since $|{\rm PSL}_2(p)|=p(p^2-1)/2$, the set $\mathcal{S}(n)$ records the orders of the possible simple direct factors of a group of cubefree order $n$. Moreover, if $n$ is cubefree, then  $|\mathcal{A}(n)|=2^{a(n)}$ where $a(n)$ is the number of prime divisors $q$ of $n$ with $q^2\mid n$; in particular, each group in $\mathcal{A}(n)$ is uniquely determined by its list of abelian invariants, for example, $[2,2,9,25]$ corresponds to the group $C_2\times C_2\times C_9\times C_{25}$.

The following result reduces the computation of $\gnu$ to that of $\gnuFF$.

\begin{proposition}\label{propredFF}
  Let $n\geq 1$ be a cubefree number; the following hold.
\begin{ithm}  
\item The number $\gnu(n)$ is given by $\gnu(n) = \sum\nolimits_{a\in \mathcal{S}(n)} \sum\nolimits_{d\mid \mathcal{Q}(n/a)} \gnuFF(n/(ad))$.
\item The number $\gnuFF(n)$ is given by  $\gnuFF(n)= \sum\nolimits_{\ell \mid n} \gnuFF(n,\ell)$.
\item Let $\ell\mid n$ with prime factorisation $\ell=p_1^{e_1}\ldots  p_m^{e_m}$; recall that each $e_i\in\{1,2\}$. Then \[S=C_{p_1}^{e_1}\times \ldots \times C_{p_m}^{e_m}\] is the unique isomorphism type of a socle of order $\ell$, with automorphism group
  \[ \Aut(S)=\GL_{e_1}(p_1)\times\ldots\times \GL_{e_m}(p_m).\]
Now $\gnuFF(n,\ell)$ is
the number of $\Aut(S)$-classes of solvable subgroups of $\Aut(S)$ of order $n/\ell$.
\end{ithm}
\end{proposition} 
\begin{proof}
  This follows from the  results in Section \ref{sec_struct}; for completeness, we recall the argument here.
  \begin{iprf}
  \item   Let $G$ be a cubefree group of order $n$. It follows from \cite[Theorem 9]{cf} or \cite[Theorem 1.1]{cf2} that $G=A\times L$ where $A$ is trivial or simple of order $|A|\in \mathcal{S}(n)$, and $L$ is solvable; in particular, the isomorphism type of $A$ is uniquely determined by $|A|$. Conversely, if $a\in \mathcal{S}(n)$, then there is, up to isomorphism, a unique cubefree simple group $A$ of order $a$, and if $L$ is a solvable cubefree group of order $n/a$, then $A\times L$ is a cubefree group of order $n$. This explains the sum over $\mathcal{S}(n)$. Thus, for a fixed $a\in \mathcal{S}(n)$, it remains to count the isomorphism types of solvable cubefree groups of order $n/a$. Let $H$ be such a group.   It follows from \cite[Section 3]{cf} that $d=|\Phi(H)|$ is a squarefree divisor of $\mathcal{Q}(n/a)$, the group $\Phi(H)\cong C_d$ is uniquely determined by $d$, and  $H/\Phi(H)$ is a Frattini-free group of order $n/(ad)$. Conversely, if $d\mid \mathcal{Q}(n/a)$ and $F$ is a solvable Frattini-free group of order $n/(ad)$, then there is, up to isomorphism, a unique group $H$ of order $n/a$ such that $\Phi(H)\cong C_d$ and $H/\Phi(H)\cong F$; see \cite[Theorem~11]{cf}. This explains the sum over the divisors of $\mathcal{Q}(n/a)$. The claim follows.
  \item A Frattini-free solvable group $L$ of order $n$ has the form $L=K\ltimes \soc(L)$ and $|\soc(L)|\mid n$; moreover, the isomorphism type of $\soc(L)$ is uniquely determined by $|\soc(L)|$,  see \cite[Theorem 9]{cf}.
  \item The structure of $S$ follows because it is a direct product of simple groups. The last claim follows from  \cite[Theorem~9]{cf}: Up to isomorphism, the solvable Frattini-free groups of order $n$ with socle order $\ell$ are exactly  $K\ltimes S$ where $K\leq \Aut(S)$ is solvable of order $n/\ell$;  two such groups $K\ltimes S$ and $\tilde K\ltimes S$ are isomorphic if and only if $K$ and $\tilde K$ are conjugate in $\Aut(S)$.
  \end{iprf}
\end{proof}    

We are interested in counting formulas that can be evaluated without the need to carry out computations in groups; we call such formulas combinatorial.

\begin{definition}\label{def_CF}
  A formula for $\gnu(n)$ is \emph{combinatorial} if it can be evaluated from the prime factorisation $n=p_1^{e_1}\cdots p_m^{e_m}$ using only elementary arithmetic with integers formed from $p_1,\ldots,p_m$, divisibility tests $a\mid n$ and $b\mid p_i\pm1$, where $a$ is such an integer and $b$  a prime power dividing $n$, and look-ups in  tables independent of $n$. In particular, it is not required to factor any integer other than $n$, nor to construct any group, group element, or matrix.
\end{definition}

\section{Groups of odd cubefree order}\label{secOdd}

\noindent By Proposition \ref{propredFF}, the enumeration of the groups of a given cubefree order reduces to the numbers $\gnuFF(n,\ell)$. The aim of this section is to present the first formula for these numbers when  $n$ is odd, see Proposition~\ref{propRed} and our main result Theorem \ref{thm_mainodd}. Thus, throughout this section, we consider odd order groups;  by the Odd Order Theorem, these groups are solvable.
 
\pagebreak

\begin{notation}\label{notSetup} We fix the following notation in this section.
\begin{ithm}
\item The integer $n\geq 1$ is odd and cubefree, and $\ell$ is a positive divisor of $n$ with prime factorisation 
  \[\ell= p_1^{e_1}\ldots p_m^{e_m};\]
recall that each $e_i\in\{1,2\}$. For $\ell=1$ we have $m=0$, and  empty products equal $1$.
\item We write $S=C_{p_1}^{e_1}\times\ldots\times C_{p_m}^{e_m}$ for the unique isomorphism type of socle of order $\ell$, with automorphism group \[\Aut(S)=\GL_{e_1}(p_1)\times\ldots\times \GL_{e_m}(p_m).\] We call the $i$-th direct factor  of $\Aut(S)$ the $i$-th \emph{column} of $\Aut(S)$, and we write $\pi_i\colon \Aut(S)\to\GL_{e_i}(p_i)$ for the projection onto it. For $K\leq \Aut(S)$ we call $\pi_i(K)$ the \emph{projection of $K$ into the $i$-th column} and abbreviate $K_i=\pi_i(K)$. This notation is slightly ambiguous, but it will be clear from the context whether a subscript indicates a projection $K_i$ or a Sylow subgroup $K_q$.
\item We write $\nu=n/\ell$, and for a prime $q$ we denote by $v_q$ the $q$-adic valuation, that is, $q^{v_q(\nu)}$ is the largest $q$-power dividing $\nu$.
\item We define $\mathcal K=\{K\leq \Aut(S) : K\text{ solvable with } |K|=\nu\}$, so $\gnuFF(n,\ell)=|\mathcal K/\Aut(S)|$  is the number of $\Aut(S)$-classes of subgroups in $\mathcal K$; the groups that are counted are  $K\ltimes S$ with $K\in\mathcal K$. In this section the solvability requirement is automatically satisfied since $\nu$ is odd.
\end{ithm}
\end{notation}

Let $K\in \mathcal{K}$. Each projection $K_i=\pi_i(K)$ is a subgroup of $\GL_{e_i}(p_i)$ and has order dividing $\nu$. Since $n=\nu\ell$ is odd and cubefree, it follows that $|\pi_i(K)|$ is odd, cubefree, and coprime to $p_i$. Moreover, $K$ is a subdirect product of $K_1,\ldots,K_m$, that is, a subgroup of the direct product $K_1\times\ldots\times K_m$ with surjective projections.

Let $K,\tilde K\in\mathcal{K}$ and recall from Section \ref{sec_struct} that $K\ltimes S$ and $\tilde K\ltimes S$ are isomorphic if and only if $K$ and $\tilde K$ are conjugate in $\Aut(S)$. Thus, it is sufficient to assume that each projection $K_i$ is a \emph{canonical} $\GL_{e_i}(p_i)$-conjugacy class representative, in particular, each $K_i=\tilde K_i$, and it remains to decide conjugacy of $K$ and $\tilde K$ in $N_{\GL_{e_1}(p_1)}(K_1)\times\ldots \times N_{\GL_{e_m}(p_m)}(K_m)$. We discuss these canonical projections and their normalisers in the following section.

\subsection{The projections and their normalisers}\label{secProj}

\noindent Every projection $K_i$ of a $K\in\mathcal K$ is a subgroup of $\GL_{e_i}(p_i)$ of odd cubefree order coprime to $p_i$. We now classify these subgroups up to conjugacy, fix a canonical representative in each class, and determine the action that the normaliser of such a representative induces on it. Throughout this section, $p$ is an odd prime.

\begin{definition}\label{defCan}
Let $p$ be an odd prime.
\begin{ithm}
\item For $b\mid p-1$ let $C(p,b)$ be the unique subgroup of order $b$ of the cyclic group $\GL_1(p)$.
\item Let $D(p)=\{\diag(x,y) : x,y\in\GL_1(p)\}$ be the subgroup of diagonal matrices in  $\GL_2(p)$. Let $J=\left(\begin{smallmatrix}0&1\\1&0\end{smallmatrix}\right)$, so  $M(p)=\langle D(p),J\rangle$ is the group of monomial matrices. Conjugation by $J$ induces an automorphism of $D(p)$ of order $2$, which we call the \emph{swap} and also denote by $J$. A subgroup $U\leq D(p)$ is scalar if $U\leq Z(\GL_2(p))$.
\item Fix a generator $\tau$ of ${\rm GF}(p^2)^\times$, so $\{1,\tau\}$ is a $\GF(p)$-basis of $\GF(p^2)$ and we can identify ${\rm GF}(p^2)$ with ${\rm GF}(p)^2$. Multiplication by the elements of ${\rm GF}(p^2)^\times$ then embeds ${\rm GF}(p^2)^\times$ into $\GL_2(p)$, and the image $\Sigma(p)$ is the \emph{standard Singer cycle}; it is cyclic of order $p^2-1$. For $b\mid p^2-1$ we denote by $\Sigma(p,b)$ the unique subgroup of $\Sigma(p)$ of order $b$.
\end{ithm}
\end{definition}

 We refer to \cite[Satz II.7.3]{hupp} for more details on Singer cycles.  The next result is  from \cite[Lemma 5.2]{cfisom}, which heavily relies on results in \cite{gl2}.

\begin{lemma}\label{lemProj}
Let $p$ be an odd prime and $e\in\{1,2\}$. If $U\leq\GL_e(p)$ has odd cubefree order $b$ coprime to $p$, then $U$ is abelian, and the following hold.
\begin{ithm}
\item If $e=1$, then $b\mid p-1$ and $U=C(p,b)$ is the unique subgroup of $\GL_1(p)$ of order $b$.
\item If $e=2$ and $U$ is irreducible, then $b\mid p^2-1$ and $b\nmid p-1$, and $U$ is conjugate to $\Sigma(p,b)$. Conversely, $\Sigma(p,b)$ is irreducible for every $b\mid p^2-1$ with $b\nmid p-1$. For each such $b$ there is exactly one conjugacy class of irreducible subgroups of order $b$.
\item If $e=2$ and $U$ is reducible, then $U$ is conjugate to a subgroup of $D(p)$. Two subgroups $V,\tilde V\leq D(p)$ are conjugate in $\GL_2(p)$ if and only if $\tilde V\in\{V,V^J\}$.
\end{ithm}
\end{lemma}
\begin{proof}
With the exception of the conjugacy claim in c), this follows directly from \cite[Lemma 5.2]{cfisom}. Let $V,\tilde V\leq D(p)$ and suppose $\tilde V=V^g$ for some $g\in\GL_2(p)$. If $V$ is scalar, then $\tilde V=V=V^J$. Otherwise $V$ contains some $\diag(x,y)$ with $x\neq y$, and the eigenspaces of $\diag(x,y)$ are the only $V$-invariant lines;  the same holds for $\tilde V$. As $g$ maps the $V$-eigenspaces to $\tilde V$-eigenspaces, we can assume that $g\in M(p)$. Since $D(p)$ is abelian, it follows that $\tilde V\in\{V,V^J\}$. The converse is clear.
\end{proof}

Lemma \ref{lemProj} allows us to make the following definition.
\begin{definition}\label{def_Upc}  
For an odd prime $p$ and $e\in\{1,2\}$, we define a set $\mathcal{U}(p,e)$ of \emph{canonical subgroup representatives}, containing exactly one subgroup from each $\GL_e(p)$-class of subgroups of $\GL_e(p)$ of odd cubefree order coprime to $p$, namely:
\begin{ithm}
\item[$\bullet$] $\mathcal{U}(p,1)$ consists of the groups $C(p,b)$ with $b\mid p-1$ odd and cubefree;
\item[$\bullet$]  $\mathcal{U}(p,2)$ consists of the groups $\Sigma(p,b)$ with $b\mid p^2-1$ odd, cubefree, and $b\nmid p-1$, together with one subgroup from each $\langle J\rangle$-orbit on the set of subgroups of $D(p)$ of odd cubefree order, see Remark \ref{rem_savegroup}a).
\end{ithm}
\end{definition}

We now consider the conjugation action of the normalisers.

\begin{lemma}\label{lemT}
Let $p$ be an odd prime and $e\in\{1,2\}$. For $U\in\mathcal{U}(p,e)$ define 
\[\Theta(U)=N_{\GL_e(p)}(U)/C_{\GL_e(p)}(U),\]
regarded as a subgroup of $\Aut(U)$. Then $|\Theta(U)|\leq 2$ and  one of the following holds.
\begin{ithm}
\item  $e=1$, or $e=2$ and $U\leq D(p)$ with either $U$ scalar or $U^J\neq U$: here $\Theta(U)=1$.
\item $e=2$ and $U\leq D(p)$ is non-scalar with $U^J=U$: here $\Theta(U)=\langle\theta_U\rangle$ has order $2$, where $\theta_U=J|_U$.
\item $e=2$ and $U=\Sigma(p,b)$ is irreducible: here  $\Theta(U)=\langle\theta_U\rangle$ has order $2$, where $\theta_U(u)=u^p$ for $u\in U$. For every nontrivial Sylow $q$-subgroup $U_q$,  either the restriction of $\theta_U$ to $U_q$ is trivial (if $|U_q|$ divides $p-1$) or is inversion (if $|U_q|$ divides $p+1$).
\end{ithm}
\end{lemma}
\begin{proof}Most of this follows implicitly from \cite[Lemma 5.2]{cfisom}, that is, from work in \cite{gl2}. We provide details for completeness. The cases $e=1$ and $U$ scalar are obvious.  Now let $e=2$ and let $U\leq D(p)$ be non-scalar. The proof of Lemma \ref{lemProj}c) implies that $N_{\GL_2(p)}(U)\leq M(p)=\langle D(p),J\rangle$ where $J$ swaps the two eigenspaces of $U$. Since every matrix that centralises a non-scalar diagonal matrix is diagonal, $C_{\GL_2(p)}(U)=D(p)$ and therefore $\Theta(U)=\langle J|_U\rangle$ has  order $2$ if $U^J=U$, and $\Theta(U)=1$ otherwise. It remains to consider the irreducible case $U=\Sigma(p,b)$. Since $b\nmid p-1$, it follows from  (the proof of) \cite[Satz II.7.3]{hupp} that $N_{\GL_2(p)}(U)=N_{\GL_2(p)}(\Sigma(p))=\langle \Sigma(p),\theta\rangle$ where $\theta(s)=s^p$ for $s\in\Sigma(p)$. Thus, $\Theta(U)=\langle \theta_U\rangle$ where $\theta_U$ is the restriction of $\theta$ to $U$; the claim follows. 
\end{proof} 
 
\pagebreak

  \begin{remark}\label{rem_savegroup}
We emphasise that, in practice, the groups in $\mathcal{U}(p,e)$ are never constructed as matrix groups; we comment on how they are determined and how they are stored.
\begin{iprf} 
\item The reducible groups in $\mathcal{U}(p,2)$ are obtained as follows. Let $U\leq D(p)$ be of odd cubefree order. Since $J$ normalises $D(p)$, we have $(U^J)_q=(U_q)^J$ for every prime $q$, and therefore it suffices to determine the possible Sylow subgroups $U_q\leq D(p)_q$ and to consider the $J$-action on these tuples $(U_q)_q$. We first consider what happens for a fixed odd prime $q$. In this case $U_q\leq \langle\diag(w,1),\diag(1,w)\rangle$ where $w\in\GL_1(p)$ has order $q^c$ with $c=\min(v_q(p-1),2)$. The $\langle J\rangle$-orbits of such subgroups can be listed independently of $p$. For example, if $c\geq1$, then $v=w^{q^{c-1}}$ has order $q$ and the subgroups of order $q$ of $D(p)_q$ are exactly $A_k=\langle\diag(v,v^k)\rangle$ for $k\in\{0,\ldots,q-1\}$ and $B=\langle\diag(1,v)\rangle$; now note that $A_0^J=B$ and $A_k^J=A_{k^{-1}\bmod q}$ for each $k\ne 0$. In particular, $A_1$ and $A_{q-1}$ are $J$-invariant and the other $q-1$ subgroups form $(q-1)/2$ orbits of size $2$. Now, for each odd prime $q$ dividing $p-1$, fix one such $\langle J\rangle$-orbit $\mathcal{O}_q$. Let $q_1,\ldots,q_r$ be those $q$ with $|\mathcal{O}_q|=2$, say $\mathcal{O}_{q_i}=\{X_{q_i},X_{q_i}^J\}$, and let $q_1',\ldots,q_s'$ be those with $\mathcal{O}_{q_j'}=\{X_{q_j'}\}$. If $r=0$, then there is a unique subgroup $U$ with each $U_q\in\mathcal{O}_q$, and it is $J$-invariant. If $r\geq 1$, then the $\langle J\rangle$-orbits of the subgroups $U$ with each $U_q\in \mathcal{O}_q$ are represented by the $2^{r-1}$ groups
  \[U_\varepsilon=X_{q_1}\times \prod\nolimits_{i=2}^r X_{q_i}^{\varepsilon_i}\times \prod\nolimits_{j=1}^s X_{q_j'},\qquad \varepsilon=(\varepsilon_2,\ldots,\varepsilon_r)\in \{1,J\}^{r-1}.\]
 Thus, $\langle J\rangle$-orbit representatives can be written down explicitly without group calculations. 
\item  By Lemma \ref{lemProj}, the irreducible groups in $\mathcal{U}(p,2)$ and the groups in $\mathcal{U}(p,1)$ are $\Sigma(p,b)$ and $C(p,b)$, respectively, one for each admissible $b$. Their Sylow $q$-subgroups are  $\Sigma(p,b_q)$ and $C(p,b_q)$, where $b_q$ denotes the $q$-part of $b$, and the induced action of $\Theta(U)$ is described in Lemma \ref{lemT}.

\item In practice, each $U\in\mathcal{U}(p,e)$ is represented as a list of integers recording the order and rank of each Sylow subgroup together with the action induced by $\Theta(U)$. (We also list the auxiliary integer data required for Lemma \ref{lemFix}, and, for even order, the  \emph{type} of $U$ as given in  Definition \ref{defTypes}.)  Remark~\ref{rem_storage} provides more technical details on how this data is organised.
\end{iprf}
\end{remark}

Since the nontrivial automorphisms arising in Lemma \ref{lemT} have order $2$, we record the following.

\begin{lemma}\label{lemEig}
Let $U$ be an abelian group of odd order and let $\theta\in\Aut(U)$. If $\theta^2=1$, then $U=U^+\times U^-$ where $U^\pm=\{u\in U : \theta(u)=u^{\pm1}\}$ are the $\pm 1$-eigenspaces of the action of $\theta$.
\end{lemma}
\begin{proof}
Write $|U|=2k-1$. Since $U$ is abelian and of odd order, both $U^\pm$ are subgroups with trivial intersection. Moreover, every $u\in U$ satisfies $u=u^{2k}=(u\theta(u))^k\cdot (u\theta(u)^{-1})^k$ with $(u\theta(u))^k\in U^+$ and $(u\theta(u)^{-1})^k\in U^-$; the claim follows.
\end{proof}

\subsection{Reduction 1: the acting group}\label{secReduce}

\noindent We now fix the projections, that is, a tuple  $\mathcal U=(U_1,\ldots,U_m)$ in $\prod_{i=1}^m\mathcal U(p_i,e_i)$, and we are interested in the isomorphism types of groups $K\ltimes S$ where $K$ is from 
\[\mathcal K(\mathcal U)=\{K\in\mathcal K : \pi_i(K)=U_i\text{ for $i=1,\ldots,m$}\}.\]
It follows from the discussion at the beginning of this section, and from the fact that each $C_{\GL_{e_i}(p_i)}(U_i)$ acts trivially on the subgroups of $U_1\times\ldots\times U_m$, that two groups $K\ltimes S$ and $\tilde K\ltimes S$ with $K,\tilde K\in\mathcal{K}(\mathcal U)$ are isomorphic if and only if $K$ and $\tilde K$ lie in the same orbit of
\begin{align}\label{def_ThetaU}\Theta(\mathcal{U})=\Theta_1\times\ldots\times \Theta_m\quad\text{where each}\quad \Theta_i=\Theta(U_i).
\end{align}Lemma \ref{lemT} shows that $\Theta(\mathcal{U})$ is an elementary abelian $2$-group that acts coordinate-wise on the set of subgroups of $U_1\times\ldots\times U_m$. Note that  $|\Theta(\mathcal U)|=2^{|\{i\,:\,\Theta(U_i)\ne 1\}|}$. We summarise these findings as follows.

\begin{proposition}\label{propRed}
Every $K\in\mathcal K$ is abelian, and every $\Aut(S)$-class of $\mathcal{K}$ meets $\mathcal{K}(\mathcal{U})$ for exactly one tuple of projections $\mathcal{U}$. Two subgroups  $K,\tilde K\in \mathcal{K}(\mathcal{U})$ are $\Aut(S)$-conjugate if and only if they lie in the same $\Theta(\mathcal{U})$-orbit. Consequently,
\[\gnuFF(n,\ell)=\sum\nolimits_{\mathcal{U}}|\mathcal{K}(\mathcal{U})/\Theta(\mathcal{U})|,\]
where $\mathcal U$ runs over all projection tuples and $\mathcal{K}(\mathcal{U})/\Theta(\mathcal{U})$ is a set of $\Theta(\mathcal{U})$-orbit representatives in~$\mathcal{K}(\mathcal{U})$.
\end{proposition}

\subsection{Reduction 2: fixing the isomorphism type of $K$}
Since every $K\in\mathcal{K}$ is abelian of order $\nu$, its isomorphism type is represented by a unique element in  $\mathcal{A}(\nu)$, see Definition \ref{def_setup}. For a projection tuple $\mathcal{U}=(U_1,\ldots,U_m)$ and $L\in\mathcal{A}(\nu)$ we now define
\[\mathcal{K}(\mathcal{U},L)=\{K\in \mathcal{K}(\mathcal{U}) : K\cong L\};\]
that is, the elements of $\mathcal{K}(\mathcal{U},L)$ are exactly the subgroups of $\Aut(S)$ of order $\nu$ that are isomorphic to $L$ and whose $i$-th projection is $U_i$. In particular, this requires that $U_i$ is a quotient of $L$.  That is, instead of looking at all the groups  $U_i\in \mathcal{U}(p_i,e_i)$ as defined in Definition \ref{def_Upc} we require that each $U_i\in \mathcal{U}(p_i,e_i,L)$ where
\begin{align}\label{eq_defUpel}\mathcal{U}(p_i,e_i,L)=\{U\in \mathcal{U}(p_i,e_i): U\text{ is a quotient of }L\}.
\end{align}This yields the following formula.

\begin{corollary}\label{corRed}
  We have
  \[\gnuFF(n,\ell)=\sum\nolimits_{L\in\mathcal{A}(\nu)} \sum\nolimits_{\mathcal{U}}|\mathcal{K}(\mathcal{U},L)/\Theta(\mathcal{U})|,\]
  where $\mathcal{U}$ runs over all  tuples in $\prod_{i=1}^m \mathcal{U}(p_i,e_i,L)$.
\end{corollary}
 
\begin{remark}\label{rem_savegroup2}
Remark \ref{rem_savegroup}c) also applies to the sets $\mathcal{U}(p,e,L)$ and projection tuples.
\end{remark}

\enlargethispage{0.4cm}
\subsection{Reduction 3: primary decomposition }
We continue with the previous notation; in particular, $L\in\mathcal{A}(\nu)$ is fixed and $\mathcal{U}=(U_1,\ldots,U_m)$ is a tuple in  $\prod_{i=1}^m \mathcal{U}(p_i,e_i,L)$. Our next reduction makes use of the direct product decomposition $K=\prod_{q\mid\nu}K_q$ of  $K\in\mathcal{K}(\mathcal{U},L)$; recall that $q$ denotes a prime, so $q\mid\nu$ runs over the prime divisors of $\nu$. Note that each $K_q\cong L_q$, and so $K_q\in \mathcal{K}_q(\mathcal{U},L)$, which is defined as
\begin{equation}\label{eq_defKq}\mathcal{K}_q(\mathcal{U},L)=\{ H\leq (U_1\times\ldots\times U_m)_q : H\cong L_q\text{ and }\pi_i(H)=(U_i)_q \text{ for all $i$}\}.
\end{equation}Then $K\mapsto \prod_{q\mid\nu} K_q$ is a $\Theta(\mathcal{U})$-equivariant bijection
\begin{equation}\label{eqSplit}
\mathcal{K}(\mathcal{U},L) \to \prod\nolimits_{q\mid\nu}\mathcal{K}_q(\mathcal{U},L),
\end{equation}
where each $\theta\in \Theta(\mathcal{U})$ acts by restriction on each $\mathcal{K}_q(\mathcal{U},L)$. Now the Cauchy-Frobenius Lemma \cite[Lemma 2.17]{handbook} shows that \[|\mathcal{K}(\mathcal{U},L)/\Theta(\mathcal{U})|=|\Theta(\mathcal{U})|^{-1}\sum\nolimits_{\theta\in \Theta(\mathcal{U})}\ \prod\nolimits_{q\mid\nu}|\mathcal{K}_q(\mathcal{U},L)^{\theta}|\]
where $\mathcal{K}_q(\mathcal{U},L)^{\theta}$ are the fixed points of $\theta$ in $\mathcal{K}_q(\mathcal{U},L)$. Thus, we have the following formula.

\begin{corollary}\label{cor_Burnside}
  We have
  \[\gnuFF(n,\ell)=\sum\nolimits_{L\in\mathcal{A}(\nu)} \sum\nolimits_{\mathcal{U}}|\Theta(\mathcal{U})|^{-1}\sum\nolimits_{\theta\in \Theta(\mathcal{U})}\ \prod\nolimits_{q\mid\nu}|\mathcal{K}_q(\mathcal{U},L)^{\theta}|,\]
  where $\mathcal{U}$ runs over all  tuples in $\prod_{i=1}^m \mathcal{U}(p_i,e_i,L)$ and $q$ over all prime divisors of $\nu$.
\end{corollary}

We use later that Corollary \ref{cor_Burnside} also holds for even cubefree numbers if  the even part is completely in the socle, that is, if $n/\ell$ is odd, see Lemma \ref{lemCasesEven}.

\subsection{Subdirect product subgroups: the numbers $|\mathcal{K}_q(\mathcal{U},L)^{\theta}|$}\label{secStep1}

\noindent We continue using the previous notation; in particular, $L\in\mathcal{A}(\nu)$ is a fixed isomorphism type and $\mathcal U=(U_1,\ldots,U_m)$ a projection tuple in $\prod_{i=1}^m\mathcal U(p_i,e_i,L)$. The aim of this section is to exhibit a closed formula for  $|\mathcal{K}_q(\mathcal U,L)^\theta|$ where $q$ is a fixed prime divisor of $\nu$ and $\theta\in\Theta(\mathcal{U})$ is a fixed automorphism.

To simplify the notation, throughout we write $W=X_1\times\ldots\times X_m$ where each $X_i=(U_i)_q$, so that $W=(U_1\times\ldots\times U_m)_q$ and
\[\mathcal{K}_q(\mathcal{U},L)=\{H\leq W : H\cong L_q\ \text{ and }\ \pi_i(H)=X_i\ \text{ for all }i\}.\]
Note that $\mathcal{K}_q(\mathcal{U},L)$ is the set of all subdirect products of $X_1,\ldots,X_m$ isomorphic to $L_q$. Since $U_i$ is a quotient of $L$, each $X_i$ is a quotient of $L_q$. Since $L_q$ is one of $C_q$, $C_{q^2}$, and $C_q^2$, all $X_i$ are cyclic if $L_q$ is cyclic, and  all $X_i$ are elementary abelian if $L_q\cong C_q^2$. In particular, for each $i$ the values
\[j_i=v_q(|X_i|)\quad\text{and}\quad d_i=\rk(X_i)\]
lie in $\{0,1,2\}$, and $j_i=d_i$ whenever $X_i$ is elementary abelian. We will need the following functions on $\{0,1,2\}$; the triples list the images of $0,1,2$, respectively.
\[\begin{array}{@{}lll@{}}
g(j)=\text{number of generators of $C_{q^j}$}, & &g=(1,\,q-1,\,q(q-1)),\\[2pt]
c(d)=\text{number of }1\times d\text{ matrices over $\GF(q)$  of rank }d, & &c=(1,\,q-1,\,0),\\[2pt]
r(d)=\text{number of }2\times d\text{ matrices over $\GF(q)$  of rank }d, & &r=(1,\,q^2-1,\,(q^2-1)(q^2-q)),\\[2pt]
z(d)=\Delta_{d=0}, & &z=(1,\,0,\,0).
\end{array}\]

\begin{lemma}\label{lemSizes}
With the notation above,
\[|\mathcal K_q(\mathcal{U},L)|=\begin{cases}
\tfrac{1}{q-1}\Big(\prod\nolimits_{i=1}^mc(d_i)-\prod\nolimits_{i=1}^mz(d_i)\Big)&\text{if }L_q\cong C_q,\\[10pt]
\tfrac{1}{q(q-1)}\Big(\prod\nolimits_{i=1}^mg(j_i)-\prod\nolimits_{i=1}^mc(j_i)\Big)&\text{if }L_q\cong C_{q^2},\\[10pt]
\tfrac{1}{|\GL_2(q)|}\Big(\prod\nolimits_{i=1}^mr(d_i)-(q+1)\prod\nolimits_{i=1}^mc(d_i)+q\prod\nolimits_{i=1}^mz(d_i)\Big)&\text{if }L_q\cong C_q^2.
\end{cases}\]
\end{lemma}
\begin{proof}
  \begin{iprf} 
 \item[{\bf (1)}] Let $L_q$ be cyclic of order $q^e$, so all $X_i$ are cyclic of order $q^{j_i}$ with $j_i\leq e$.  In particular, there is $v=(v_1,\ldots,v_m)\in W$ such that each $\langle v_i\rangle=X_i$. The group $\langle v\rangle$ projects onto every $X_i$ and has order $\max_i|v_i|$. Conversely, every $H\in\mathcal K_q(\mathcal U,L)$ (if it exists) is generated by such an element $v$. Note that $\mathcal{K}_q(\mathcal{U},L)\ne \emptyset$ if and only if $j_i=e$ for some $i$. It follows that  the groups in $\mathcal K_q(\mathcal U,L)$ are exactly  $\langle v\rangle$ with $v$ as above such that $j_i=e$ for some $i$. If $j_i=e$ for some $i$, then there are $\prod_{i=1}^mg(j_i)$ such elements $v$, and each $H\in\mathcal K_q(\mathcal U,L)$ is generated by $g(e)$ of them, so
   \[|\mathcal K_q(\mathcal{U},L)|=\Delta_{e=\max_i j_i} \cdot g(e)^{-1}\prod\nolimits_{i=1}^mg(j_i).\]
\item[{\bf (1.1)}] If $e=1$, then each $j_i=d_i\in\{0,1\}$ and so $g(j_i)=c(j_i)=c(d_i)$ and $\prod_{i=1}^m z(d_i)=\Delta_{j_1=\ldots=j_m=0}$. Thus, if all $j_i=d_i=0$, then $\Delta_{e=\max j_i}=0$ and each $c(d_i)=1=z(d_i)$, and so 
\[\tfrac{1}{q-1}\Big(\prod\nolimits_{i=1}^mc(d_i)-\prod\nolimits_{i=1}^mz(d_i)\Big)=0=\Delta_{e=\max j_i}\cdot  g(e)^{-1}\prod\nolimits_{i=1}^mg(j_i).\]If there exists $i$ such that $j_i=d_i=1$, then $\Delta_{e=\max j_i}=1$ and $\prod_{i=1}^m z(d_i)=0$, hence
\[\tfrac{1}{q-1}\Big(\prod\nolimits_{i=1}^mc(d_i)-\prod\nolimits_{i=1}^mz(d_i)\Big)=\tfrac{1}{q-1}\prod\nolimits_{i=1}^mc(d_i)=\Delta_{e=\max j_i}\cdot  g(e)^{-1}\prod\nolimits_{i=1}^mg(j_i).\]
This proves the displayed formula in the lemma for $L_q\cong C_q$.
\item[{\bf (1.2)}] If $e=2$, then  $c(j_i)=g(j_i)$ if $j_i\leq1$ and $c(2)=0$, so $\prod_{i=1}^mc(j_i)=\Delta_{j_1,\ldots,j_m \leq 1}\prod_{i=1}^m g(j_i)$. This implies
  \begin{align*}\tfrac{1}{q(q-1)}\Big(\prod\nolimits_{i=1}^mg(j_i)-\prod\nolimits_{i=1}^mc(j_i)\Big)&=
  \tfrac{1}{q(q-1)}\Big(\prod\nolimits_{i=1}^mg(j_i)-\Delta_{j_1,\ldots,j_m \leq 1}\prod_{i=1}^m g(j_i)\Big)\\&=
  \Delta_{e=\max_i j_i} \cdot g(e)^{-1}\prod\nolimits_{i=1}^mg(j_i),
  \end{align*}
which proves the displayed formula in the lemma for $L_q\cong C_{q^2}$.
 
\item[{\bf (2)}] Let $L_q\cong C_q^2$, so each $X_i$ is elementary abelian and $W$ is a ${\rm GF}(q)$-space of dimension $\sum_{i=1}^m d_i$. A subgroup $H\leq W$ with $H\cong L_q$ is a $2$-dimensional subspace, and we can describe it as the row space of a $(2\times\sum_{i=1}^m d_i)$-matrix $M=(M_1|\ldots|M_m)$ with blocks $M_i$ of size $2\times d_i$. Then $\pi_i(H)=X_i$ if and only if $M_i$ has full rank $d_i$, and $\dim H=2$ if and only if $M$ has rank $2$; two matrices have the same row space if and only if they differ by an element of $\GL_2(q)$. Hence $|\mathcal{K}_q(\mathcal{U},L)|$ is $|\GL_2(q)|^{-1}$ times the number of matrices $M$ as above that have rank $2$ and all of whose blocks have full rank.

  There are $\prod_ir(d_i)$ matrices $M=(M_1|\ldots|M_m)$ all of whose blocks $M_i$ have full rank, so it remains to subtract those $M$ of rank at most $1$. Such a matrix has all its columns in a common $1$-dimensional subspace of ${\rm GF}(q)^2$, which implies that each block $M_i$ has rank $d_i\leq1$. Such a matrix is obtained by choosing one of the $q+1$ subspaces of dimension $1$ and then a nonzero vector in it for each $d_i=1$; this gives $(q+1)\prod_ic(d_i)$ matrices. If all $d_i=0$ then we overcount the `empty matrix' $q$ times and we correct this by subtracting $q\prod_{i=1}^m z(d_i)$. (Note that if some $d_i=2$ then  both products are $0$.)  This yields the displayed formula for $L_q\cong C_q^2$.
  \end{iprf} 
\end{proof} 

\enlargethispage{0.6cm}
In the next lemma we continue with the above notation; recall that  $\theta\in\Theta(\mathcal{U})$ is fixed.
 \begin{lemma}\label{lemFix}
We continue with the above notation and decompose each  $X_i=X_i^+\times X_i^-$ into the $\theta$-eigenspaces as in Lemma \ref{lemEig}, and define $d_i^\pm=\rk(X_i^\pm)$.  Then
\begin{align}\label{eq_kqut}|\mathcal{K}_q(\mathcal{U},L)^\theta|=\Big(\prod\nolimits_{i=1}^mz(d_i^-)+\prod\nolimits_{i=1}^mz(d_i^+)\Big)|\mathcal K_q(\mathcal{U},L)| + \Delta_{L_q\cong C_q^2}\ \Lambda^+\Lambda^-
\end{align}
where
\[\Lambda^\pm=\tfrac{1}{q-1}\Big(\prod\nolimits_{i=1}^mc(d_i^\pm)-\prod\nolimits_{i=1}^mz(d_i^\pm)\Big).\]
\end{lemma}
\begin{proof}
Let $W=W^+\times W^-$ with $W^\pm=X_1^\pm\times\ldots\times X_m^\pm$ be the eigenspace decomposition of  $\theta$. Since $z(d_i^\pm)=1$ if and only if $X_i^\pm=1$, it follows that
\[(\ast)\quad \prod\nolimits_iz(d_i^-)=\Delta_{\theta\text{ acts trivially on }W}\qquad\text{and}\qquad \prod\nolimits_iz(d_i^+)=\Delta_{\theta\text{ inverts }W}.\]
Since $|W|$ is odd, both products are $1$ if and only if $W=1$: in that case $\mathcal K_q(\mathcal U,L)=\emptyset$, and since each $c(d_i^\pm)=1=z(d_i^\pm)$, we have  $\Lambda^+\Lambda^-=0$ as well; so both sides of \eqref{eq_kqut} are $0$. We now  assume that $W\ne 1$ and define the following sets of subgroups
 \[\mathcal{A}^\pm=\{A\leq X_1^\pm\times\ldots\times X_m^\pm: A\cong C_q\text{ and } \pi_i(A)=X_i^\pm \text{ for all $i$}\}.\]

 \begin{iprf}
 \item[{\bf (1)}] We prove that $|\mathcal{A}^\pm|=\Lambda^\pm$ if $L_q\cong C_q^2$; this is the only case in which the sets $\mathcal{A}^\pm$ are used, see (4). Note  that then each $X_i$, and hence each $X_i^\pm$, is elementary abelian, so  $|X_i^\pm|=q^{d_i^\pm}$. If some $d_i^\pm=2$, then no group of order $q$ maps onto $X_i^\pm$, so $|\mathcal{A}^\pm|=0$, and $\Lambda^\pm=0$  since $c(2)=z(2)=0$. Now suppose each $d_i^\pm\leq 1$, and set $\delta^\pm=|\{i : X_i^\pm\neq1\}|$. If $\delta^\pm =0$, then each $X_i^\pm=1$ and $d_i^\pm=0$, so $\mathcal{A}^\pm=\emptyset$, and  $\Lambda^\pm=0$ since each $c(d_i^\pm)=z(d_i^\pm)=1$. If $\delta^\pm\geq 1$, then it is easy to see that  $|\mathcal{A}^\pm|=(q-1)^{\delta^{\pm}-1}=\Lambda^\pm$;  note that $\prod_iz(d_i^\pm)=0$.
\item[{\bf (2)}] We make two observations. First, suppose that $W=W^+$ or $W=W^-$, that is, $\theta$ acts trivially on $W$ or inverts $W$. Then $\theta$ stabilises every subgroup of $W$, so $|\mathcal{K}_q(\mathcal{U},L)^\theta|=|\mathcal{K}_q(\mathcal{U},L)|$; moreover, in the first case each $d_i^-=0$ and therefore $\Lambda^-=0$, and in the second case each $d_i^+=0$ and therefore $\Lambda^+=0$. Since $W\ne1$, exactly one of the two products in $(\ast)$ equals $1$, so the two sides of \eqref{eq_kqut} are equal. Second, if some $H\in\mathcal{K}_q(\mathcal{U},L)$ satisfies $H\leq W^+$, then each $X_i=\pi_i(H)\leq X_i^+$ and hence $W=W^+$; analogously, $H\leq W^-$ forces $W=W^-$.
\item[{\bf (3)}] Let $L_q$ be cyclic. The case that $\theta$ acts trivially or by inversion on $W\ne 1$ is discussed in (2), so suppose that $\theta$ is neither trivial nor inversion, so $W\ne W^\pm$ and both products in $(\ast)$ are $0$; since $L_q$ is cyclic, also $\Delta_{L_q\cong C_q^2}=0$, so the right side of \eqref{eq_kqut} is $0$.   If $H\in\mathcal{K}_q(\mathcal{U},L)^\theta$, then $H$ is cyclic, say $H=\langle v\rangle$. Since $\theta^2=1$, the automorphism induced by $\theta$ on $H$ has order at most $2$; as $\Aut(H)$ is cyclic, it contains a unique involution, namely inversion, so the induced automorphism is the identity or  inversion. Thus, $v\in W^+$ or $v\in W^-$, and so $H\leq W^+$ or $H\leq W^-$. Now (2) yields $W=W^+$ or $W=W^-$, a contradiction. Thus, $|\mathcal{K}_q(\mathcal{U},L)^\theta|=0$, as required.
 
\item[{\bf (4)}] Let $L_q\cong C_q^2$. The case that $\theta$ acts trivially or by inversion on $W\ne 1$ is discussed in (2). Thus, suppose that $\theta$ is neither trivial nor inversion; then $W\ne W^\pm$, both products in $(\ast)$ are $0$, and it remains to show that $|\mathcal{K}_q(\mathcal{U},L)^\theta|=\Lambda^+\Lambda^-$. Note that $H\in\mathcal{K}_q(\mathcal{U},L)$ is fixed by $\theta$ if and only if $H=H^+\oplus H^-$ with $H^\pm=H\cap W^\pm$. If there exists  $H\in\mathcal{K}_q(\mathcal{U},L)^\theta$ with $H^+=H$ or $H^-=H$, then $H\leq W^+$ or $H\leq W^-$, and (2) yields $W=W^+$ or $W=W^-$, a contradiction. Thus, every $H$ in $\mathcal{K}_q(\mathcal{U},L)^\theta$ satisfies $\dim H^+=\dim H^-=1$. Since $\pi_i(H)=\pi_i(H^+)\oplus\pi_i(H^-)\leq X_i^+\oplus X_i^-=X_i$, we have $\pi_i(H)=X_i$ for all $i$ if and only if $\pi_i(H^\pm)=X_i^\pm$ for all $i$, that is, $H^\pm\in\mathcal{A}^\pm$. These groups $H$ are exactly $A^+\oplus A^-$ with each $A^\pm\in\mathcal{A}^\pm$, and there are $|\mathcal{A}^+|\,|\mathcal{A}^-|$ of them; by (1) this number is $\Lambda^+\Lambda^-$, so  \eqref{eq_kqut} holds.
\end{iprf}
\end{proof}

We note that the determination of the $d_i^\pm$ requires no group calculations, see Remark~\ref{rem_storage}.

\subsection{The main formula for odd order: Theorem D}\label{sec_mainodd}
Based on the previous results,  we determine $\gnu(n)$ for every odd cubefree $n$ by  combining Proposition \ref{propredFF}a,b), Corollary \ref{cor_Burnside},  and Lemmas \ref{lemSizes} and \ref{lemFix}. Since $n$ is odd, $\mathcal{S}(n)=\{1\}$ in Proposition \ref{propredFF}a); this yields the formula shown in the following theorem, which is the theorem we have quoted (and discussed) in Section \ref{sec_intromain}. Recall the definitions of $\mathcal{Q}(n)$, $\mathcal{A}(n/(\ell d))$, and $\Theta(\mathcal{U})$ from Definition \ref{def_setup} and \eqref{def_ThetaU}. The sets $\mathcal{U}(p_i,e_i,L)$ and $\mathcal{K}_q(\mathcal{U},L)$ are defined in \eqref{eq_defUpel} and \eqref{eq_defKq}, respectively, and $\mathcal{K}_q(\mathcal{U},L)^\theta$ denotes the subset of $\theta$-fixed~points.
 
\begin{theorem}\label{thm_mainodd}
 The number of isomorphism types of groups of odd cubefree order $n$ is
\[\gnu(n)=\sum_{d\mid \mathcal{Q}(n)}\ \sum_{\ell\mid (n/d)}\ \sum_{L\in\mathcal{A}(n/(\ell d))}\ \sum_{\mathcal{U}}\ \frac{1}{|\Theta(\mathcal{U})|}\sum_{\theta\in\Theta(\mathcal{U})}\ \prod_{q\mid(n/(\ell d))}\big|\mathcal{K}_q(\mathcal{U},L)^{\theta}\big|,\]
where $\mathcal{U}$ runs over the projection tuples in $\prod_{i=1}^m\mathcal{U}(p_i,e_i,L)$ where  $\ell=p_1^{e_1}\ldots p_m^{e_m}$, and $q$ runs over all prime divisors of $n/(\ell d)$; the cardinalities $\big|\mathcal{K}_q(\mathcal{U},L)^{\theta}\big|$ are determined  in Lemma \ref{lemFix}.
\end{theorem}  

This formula is combinatorial in the sense of Definition~\ref{def_CF}: by Remark~\ref{rem_savegroup} the projection tuples and the group $\Theta(\mathcal{U})$ are represented by lists of integers, and Lemmas~\ref{lemSizes} and~\ref{lemFix} evaluate the cardinalities $|\mathcal{K}_q(\mathcal{U},L)^{\theta}|$ arithmetically; see Section~\ref{secComb} for a summary.

\enlargethispage{0.5cm}
\begin{example}
  Consider  $n=75=3\cdot 5^2$ with $\mathcal{Q}(n)=5$, so $d$ runs over $\{1,5\}$ and $\ell$ over the divisors of $n/d$. First, we show that only  $(d,\ell)=(5,15)$, $(1,75)$, and $(1,25)$ contribute to $\gnu(n)$: if $(d,\ell)$ is another pair with $d\mid 5$ and $\ell\mid n/d$, then we always have $\ell\in\{1,3,5,15\}$ and $\nu=n/(\ell d)>1$. The factorisation of $\ell$ implies that $|\Aut(S)|$ divides $(3-1)\cdot (5-1)=8$, which implies that $\Aut(S)$ cannot have a subgroup of (odd) order $\nu$. We see this in our formula since each set $\mathcal{U}(p_i,e_i,L)$ contains only one element (representing the  trivial subgroup), so the only tuple $\mathcal{U}$ we consider has all $j_i=d_i=0$, and so every case of Lemma \ref{lemSizes} yields $|\mathcal{K}_q(\mathcal{U},L)|=0$.

  If $(d,\ell)$ is $(5,15)$ or $(1,75)$, then $\nu=1$, so $L=1$, $\mathcal{U}=(1,1)$, and $\Theta(\mathcal{U})=1$; the product over the primes $q\mid\nu$ is empty and each term is $1$. Thus, $(5,15)$ counts the unique group with Frattini subgroup $C_5$ and Frattini quotient $C_{15}$, namely $C_{75}$; and $(1,75)$ counts the Frattini-free group $C_3\times C_5^2$.

  For $(d,\ell)=(1,25)$ we have $m=1$, $p_1=5$, $e_1=2$,  $S=C_5^2$,  $\Aut(S)=\GL_2(5)$, $\nu=3$, and so $L\cong C_3$. Since $3\nmid 5-1$ and $3\mid 5^2-1$, we only have to consider $\mathcal{U}(5,2,C_3)=\{1,\Sigma(5,3)\}$, see Lemma \ref{lemProj}. The tuple $\mathcal{U}=(1)$ contributes $0$ as above. For $\mathcal{U}=(\Sigma(5,3))$ we need to consider $q=3$, and $X_1=\Sigma(5,3)\cong C_3$ gives $j_1=d_1=1$, so  $|\mathcal{K}_3(\mathcal{U},L)|=(c(1)-z(1))/(q-1)=1$ by Lemma~\ref{lemSizes}. Moreover,  $\Theta(\mathcal{U})=\langle\theta_{U_1}\rangle$ has order $2$ and acts on $U_1$ by inversion since $3\mid 5+1$, see Lemma \ref{lemT}; thus $d_1^+=0$ and $d_1^-=1$, and Lemma \ref{lemFix} yields $|\mathcal{K}_3(\mathcal{U},L)^{\theta_{U_1}}|=(z(1)+z(0))\cdot 1=1$. Averaging over $\Theta(\mathcal{U})$ gives $\tfrac{1}{2}(1+1)=1$, and the corresponding group is $C_3\ltimes C_5^2$, the unique (up to isomorphism) nonabelian group of order $75$.

  In conclusion,  $\gnu(75)=1+1+1=3$, which agrees with GAP's {\small \tt{NumberSmallGroups}}. We emphasise that the calculation above is phrased in terms of groups for expository purposes only: in our implementation no group is ever constructed, see Remark~\ref{rem_savegroup}c) and Remark~\ref{rem_storage}.
\end{example}

\begin{remark}\label{rem_oddsplit}
 We recall a result of   Erd\H{o}s-P\'alfy \cite{erdos} that allows us to compute $\gnu(n)$ more efficiently for many odd $n$. For a cubefree $n\geq 2$, the graph $\Gamma(n)$ has  vertex set  $\pi(n)$, the set of prime divisors of $n$, and distinct $p,q\in\pi(n)$ are adjacent if and only if $p\mid q-1$, or $q\mid p-1$, or  $p\mid q+1$ and $q^2\mid n$, or  $q\mid p+1$ and $p^2\mid n$. For a subset $\mathcal{C}\subseteq \pi(n)$ we denote by $n_{\mathcal{C}}$ the largest divisor of $n$ such that $n/n_{\mathcal{C}}$ has no prime divisors in $\mathcal{C}$. If $\mathcal{C}_1,\ldots,\mathcal{C}_r$ are the connected components of $\Gamma(n)$ and $G$ is a group of order $n$, then   \cite[Corollary 2.4]{erdos} implies that  $G=G_1\times\ldots\times G_r$ where each $G_k$ is a Hall $\mathcal{C}_k$-subgroup. In particular, this implies that 
  \begin{align}\label{eq_splitgnu}\gnu(n)=\prod\nolimits_{k=1}^r\gnu(n_{\mathcal{C}_k}).
  \end{align}
  For even $n$, the graph is always connected. However, it is shown in  \cite[Corollary 4.2]{erdos} that $\Gamma(n)$ is disconnected for almost all odd $n$, so \eqref{eq_splitgnu} applies to almost all odd orders, and in particular to many orders of practical interest: for example,   $3{,}515{,}706$ of the  $4{,}753{,}758$ odd cubefree integers $n\leq 10^7$ (that is, $\approx 74\%$ of these integers) have a disconnected graph.
\end{remark}


\section{Groups of even cubefree order}\label{secEven}

\enlargethispage{0.6cm}

\noindent The case of even order is more complicated, mainly because Proposition \ref{propRed} fails and the socle complement $K$ is not necessarily abelian. Since Proposition \ref{propredFF} holds for every cubefree $n$, it remains to determine the numbers $\gnuFF(n,\ell)$. Recall that $\ell$  denotes the socle order and the  order of the complement $K$ is denoted $\nu=n/\ell$. We treat the even case by making a case distinction on whether the $2$-part of $\nu$ is $1$, $2$, or $4$. We first show that $K$ is a semidirect product $E\ltimes O$ where $E$ has order dividing $4$ and $O$ has odd order. Throughout, we write $\nu=2^t\nu_0$ where $\nu_0$ is odd.

If $t=0$, then the results of Section \ref{secOdd} hold: the only difference is that the socle might contain a direct factor $C_2$ or $C_2^2$. Note that $\Aut(C_2)=1$ and $\Aut(C_2^2)\cong \GL_2(2)=C_2\ltimes C_3$ where the outer $C_2$ is generated by the swap matrix $J$. It follows that the results in Section \ref{secProj} can be easily extended to this case. It therefore remains to consider the cases $t=1$ and $t=2$, see Lemma \ref{lemCasesEven}. 

\begin{notation}\label{notEvenN} 
  Throughout this section, $n\geq1$ is cubefree and $\ell\mid n$. We continue with  Notation \ref{notSetup}, with the following additions.
  \begin{iprf}
  \item The number  $n$ is no longer assumed to be odd, and we write $\nu=n/\ell=2^t\nu_0$ where $\nu_0$ is odd.
  \item We denote by $L$ an element of $\mathcal{A}(\nu_0)$, that is, an isomorphism type of abelian group of order $\nu_0$.
  \item   For $U\leq\GL_{e_i}(p_i)$, write $N_i(U)=N_{\GL_{e_i}(p_i)}(U)$ and $Z_i(U)=C_{\GL_{e_i}(p_i)}(U)$; recall from Lemma~\ref{lemT} that $\Theta(U)=N_i(U)/Z_i(U)$, and we write $\epsilon_U\colon N_i(U)\to\Theta(U)$ for the natural epimorphism.
 \item   For a projection tuple $\mathcal{U}=(U_1,\ldots,U_m)$ we  abbreviate $N_i=N_i(U_i)$ and $Z_i=Z_i(U_i)$, and write
\[N(\mathcal{U})=\prod\nolimits_{i=1}^mN_i\quad\text{and}\quad C(\mathcal{U})=\prod\nolimits_{i=1}^mZ_i\]for the normaliser and centraliser of $U_1\times\ldots\times U_m$ in $\Aut(S)$. (We write $Z_i$ since $C_i$ denotes a cyclic group.) We define $\epsilon\colon N(\mathcal{U})\to N(\mathcal{U})/C(\mathcal{U})=\Theta(\mathcal{U})$ componentwise, that is, $\epsilon=(\epsilon_{U_1},\ldots,\epsilon_{U_m})$.
  \end{iprf}  
\end{notation}

Let $K\in\mathcal{K}$ with projections $K_i=\pi_i(K)\leq \GL_{e_i}(p_i)$ as in Notation \ref{notSetup}b). Since $K$ is solvable, so is $K_i$. Moreover, $|K_i|$ divides $\nu$, hence it is cubefree, and coprime to $p_i$: if $e_i=1$, then $K_i\leq\GL_1(p_i)$ has order dividing $p_i-1$; if $e_i=2$, then $p_i^2$ divides $\ell$, so $p_i\nmid\nu=n/\ell$ because $n$ is cubefree.

\begin{lemma}\label{lemCasesEven}
With the previous notation, let $n$ be cubefree, $\ell\mid n$ and $t=v_2(\nu)$ where  $\nu=n/\ell$.
\begin{ithm}
\item If $t=2$, then $\ell$ is odd and each $\GL_{e_i}(p_i)\ne 1$.
\item If $t=1$ and $2\mid \ell$, then $v_2(\ell)=1$ and the Sylow $2$-subgroup of the socle is a central direct factor.
\item If $t=0$, then Definitions \ref{defCan} and \ref{def_Upc} and Lemmas \ref{lemProj}, \ref{lemT}, and \ref{lemEig} remain valid for $p=2$, and all results of Section \ref{secOdd} hold; in particular, $\gnuFF(n,\ell)$ is given by Corollary \ref{cor_Burnside}.
\end{ithm}
\end{lemma}  
\begin{proof}
\begin{iprf}  
\item This follows since $\nu\ell$ is cubefree.
\item If $2\mid \ell$, then $v_2(\ell)=1$ since $\nu\ell$ is cubefree. It follows that the socle $S$ in the groups we count has  a direct factor $C_2$ with trivial automorphism group $\GL_1(2)=1$. Thus, if $K\leq \Aut(S)$ is such that $G=K\ltimes S$ is cubefree, then $G=(K\ltimes \tilde S)\times C_2$ where $S=\tilde S\times C_2$. 
\item If $t=0$, then each $K\in\mathcal{K}$ has odd order $\nu$ and $|\pi_i(K)|$ is coprime to $p_i$. The only hypothesis of Section \ref{secOdd} that is not necessarily true is that each $p_i$ is odd. If $p_i=2$, then  $\GL_{1}(p_i)$ and $D(p_i)$ are both trivial. Thus,  $\mathcal{U}(p_i,1)=\{1\}$, and the reducible part of $\mathcal{U}(p_i,2)$ is $\{1\}$. Moreover, $\Sigma(2)\cong C_3$ with normaliser  $\GL_2(2)=\langle\Sigma(2),J\rangle$ of order $6$; so here we have $\mathcal{U}(2,2)=\{1,\Sigma(2,3)\}$. This implies that Lemmas \ref{lemProj} and \ref{lemT} still hold. For example, if $U=\Sigma(2,3)$, then $C_{\GL_2(2)}(U)=\Sigma(2)$ and $N_{\GL_2(2)}(U)=\GL_2(2)$, so $\Theta(U)=\langle\theta_U\rangle$ acts by inversion.
\end{iprf}  
\end{proof}
\enlargethispage{0.7cm}
Thus, the case $t=0$ is essentially dealt with by the results in Section \ref{secOdd}. Throughout the following, we consider $t\in\{1,2\}$. In this case, by Lemma \ref{lemCasesEven}, we can assume in practice that all the prime divisors $p_i$ of the socle order $\ell$ are odd: if not, then the socle has a central direct factor $C_2$ which does not add  anything to the number of groups.

\subsection{Decomposing the  socle complement}\label{secEvenNhall}
We first show that every socle complement $K$ is a semidirect product $E\ltimes O$ where $O$ is a normal Hall $2'$-subgroup and $E$ is a Sylow $2$-subgroup.

\begin{lemma}\label{lem_kHall}
  Let $K\in\mathcal{K}$, so $K\leq \Aut(S)$ has order $\nu=2^t\nu_0$. Then $K=E\ltimes O$ where $E$ is a Sylow $2$-subgroup of $K$ and $O=O_{2'}(K)$ is a normal abelian Hall $2'$-subgroup; we have $O^g=O_{2'}(K^g)$ for every $g\in\Aut(S)$.
\end{lemma}
\begin{proof}
By the above, each $K_i=\pi_i(K)$ is a solvable cubefree subgroup of $\GL_{e_i}(p_i)$ of order coprime to $p_i$; now it follows from  \cite[Lemma 5.2]{cfisom} that  each  $K_i$ has an abelian normal subgroup $T_i\leq K_i$ of index at most $2$. Then $\prod_{i=1}^m T_i$ is normal in $\prod_{i=1}^mK_i$, and  $\hat K=K\cap\prod_{i=1}^mT_i$ is normal in $K$. Moreover, $K/\hat K$ embeds in $\prod_{i=1}^m(K_i/T_i)$ and therefore  is an elementary abelian $2$-group. The group $O=O_{2'}(\hat K)$ is characteristic in $\hat K$, so normal in $K$, and $|K/O|=|K/\hat K||\hat K/O|$ is a power of $2$. Thus, $O$ is a normal abelian Hall $2'$-subgroup of $K$ of order $\nu_0$. By the Schur-Zassenhaus Theorem \cite[(9.1.2)]{rob}, every Sylow $2$-subgroup $E$ of $K$ satisfies $K=E\ltimes O$. The last claim about $\Aut(S)$ follows since a normal Hall $2'$-subgroup is invariant under isomorphisms.
\end{proof}

Recall from Proposition \ref{propredFF} and  Notation \ref{notSetup}d) that we are interested in the number of $\Aut(S)$-classes in  $\mathcal{K}$. Lemma~\ref{lem_kHall} shows that if $K\in\mathcal{K}$, then $K=E\ltimes O$ where $O=O_{2'}(K)\leq K$ is characteristic in $K$ and $E\leq K$ is a Sylow $2$-subgroup (which is unique up to conjugacy). The next proposition is the main result of this section and shows that the $\Aut(S)$-classes in $\mathcal{K}$ can be computed by considering the $\Aut(S)$-orbits on 
\[\mathcal{P}=\big\{(E,O) : O\leq\Aut(S),\ |O|=\nu_0,\ E\leq N_{\Aut(S)}(O),\ |E|=2^t\big\},\]
where $\Aut(S)$ acts by simultaneous conjugation on both entries.

\begin{proposition}\label{propPairs}
The map $(E,O)\mapsto EO$ induces a bijection $\mathcal{P}/\Aut(S)\to\mathcal{K}/\Aut(S)$, so \[\gnuFF(n,\ell)=|\mathcal{P}/\Aut(S)|.\]
\end{proposition}
\begin{proof}
  If $(E,O)\in\mathcal{P}$, then $O\unlhd EO$ and $E\cap O=1$, so $EO=E\ltimes O$ is a subgroup of $\Aut(S)$ of order $\nu_0 2^t=\nu$; it is solvable since $O$ has odd order and $|E|\leq 4$, and therefore $EO\in\mathcal{K}$. By construction, the map is $\Aut(S)$-equivariant and surjective, and it maps conjugate pairs in $\mathcal{P}$ to conjugate groups in $\mathcal{K}$. Lastly, suppose $(E_1,O_1),(E_2,O_2)\in\mathcal{P}$ with $(E_1O_1)^g=E_2O_2$ for some $g\in\Aut(S)$. Then $O_1^g=O_{2'}(E_1O_1)^g=O_{2'}(E_2O_2)=O_2$; moreover,  $E_1^g$ and $E_2$ are Sylow $2$-subgroups of $E_2O_2$, hence conjugate by some $h\in E_2O_2$. Since $O_2E_2$ normalises $O_2$, the element $gh\in\Aut(S)$ conjugates  $(E_1,O_1)$ to $(E_2,O_2)$. The claim follows.
\end{proof}

\subsection{Strategy and preliminary results}\label{secStrat}
 
\noindent By Proposition \ref{propPairs}, we need to count the $\Aut(S)$-orbits on the set $\mathcal{P}$ of pairs $(E,O)$; we do this as follows. 
\begin{ithm}
\item[1)] Fix the $\Aut(S)$-class of the odd part $O$: since $|O|=\nu_0$ is odd, $O$ is one of the groups classified in Section \ref{secOdd}, with associated projection tuple $\mathcal{U}$. 
\item[2)] Count the possible $E$ (componentwise for each Sylow subgroup of the socle).
\item[3)] Count the $\Theta(\mathcal{U})$-orbits on the possible $E$.
\end{ithm}
If $t=1$, then $E$ is generated by an involution and this can essentially be dealt with by the machinery developed for the odd case. If  $t=2$, then $E$ might not be cyclic and we have to work with generators and the automorphisms of $E$. 

Throughout we use Notations \ref{notSetup} and \ref{notEvenN}. By Lemma \ref{lemCasesEven} (with $\nu_0$ instead of $\nu$) the results of Section~\ref{secOdd} apply to the subgroups of odd order $\nu_0$ of $\Aut(S)$: every $\Aut(S)$-class of such subgroups meets $\mathcal{K}(\mathcal{U},L)$ for exactly one pair $(L,\mathcal{U})$ with $L\in\mathcal{A}(\nu_0)$ and $\mathcal{U}\in\prod_{i=1}^m\mathcal{U}(p_i,e_i,L)$. Here and below, the sets of Section \ref{secOdd} are used with $\nu_0$ in place of $\nu$; for example, 
\[\mathcal{K}(\mathcal{U},L)=\{O\leq\Aut(S) : |O|=\nu_0,\ O\cong L,\ \pi_i(O)=U_i\ \text{ for } i=1,\ldots,m\}.\]
Let
\begin{align}\label{eq_mcE}\mathcal{E}_t(\mathcal{U})= \{ E\leq N(\mathcal{U}): |E|=2^t\}/C(\mathcal{U}),
\end{align}
considered as a set of $C(\mathcal{U})$-class representatives. Since  $C(\mathcal{U})=\ker\epsilon$, it follows that if $E\in\mathcal{E}_t(\mathcal{U})$ and $g\in C(\mathcal{U})$, then $\epsilon(E^g)=\epsilon(E)$; that is, the image under $\epsilon$ does not depend on the chosen class representative. An element $\xi\in\Theta(\mathcal{U})=N(\mathcal{U})/C(\mathcal{U})$ acts on $\mathcal{E}_t(\mathcal{U})$ by mapping $E$ to the class representative of  $E^\xi$. Lastly, we define
\[\mathcal{M}(\mathcal{U},L)=\big\{(E,O) : O\in\mathcal{K}(\mathcal{U},L),\; E\in\mathcal{E}_t(\mathcal{U}),\; O^{\epsilon(E)}=O\big\}.\]
This allows us to describe $\gnuFF(n,\ell)$ as follows.

\enlargethispage{0.4cm}

\begin{proposition}\label{propRedS} 
With the notation above, 
\[\gnuFF(n,\ell)=\sum\nolimits_{L\in\mathcal{A}(\nu_0)} \sum\nolimits_{\mathcal{U}} \big|\mathcal{M}(\mathcal{U},L)/\Theta(\mathcal{U})\big|,\]
where $\mathcal{U}$ runs over the tuples in $\prod_{i=1}^m\mathcal{U}(p_i,e_i,L)$.
\end{proposition}
\begin{proof}
Abbreviate $\mathcal{P}(\mathcal{U},L)=\{(E,O)\in\mathcal{P} : O\in\mathcal{K}(\mathcal{U},L)\}$ where $\mathcal{U}=(U_1,\ldots,U_m)$. As mentioned above, every $\Aut(S)$-orbit on $\mathcal{P}$ meets $\mathcal{P}(\mathcal{U},L)$ for a unique pair $(L,\mathcal{U})$. If $(E_1,O_1),(E_2,O_2)$ in $\mathcal{P}(\mathcal{U},L)$ satisfy $(E_1,O_1)^g=(E_2,O_2)$ for some $g\in\Aut(S)$,  then for each $i$ we have  $U_i=\pi_i(O_2)=\pi_i(O_1)^{g_i}=U_i^{g_i}$, so $g\in N(\mathcal{U})$. This shows that $\mathcal{P}/\Aut(S)$ is the disjoint union of the sets $\mathcal{P}(\mathcal{U},L)/N(\mathcal{U})$, and Proposition \ref{propPairs} gives the displayed sum with $\mathcal{P}(\mathcal{U},L)/N(\mathcal{U})$ in place of $\mathcal{M}(\mathcal{U},L)/\Theta(\mathcal{U})$. Note that $C(\mathcal{U})$ acts trivially on $\mathcal{K}(\mathcal{U},L)$, thus the $C(\mathcal{U})$-orbits on $\mathcal{P}(\mathcal{U},L)$ are represented by the pairs $(E,O)$ with $E\in\mathcal{E}_t(\mathcal{U})$, and the remaining  action is that of $N(\mathcal{U})/C(\mathcal{U})=\Theta(\mathcal{U})$. Lastly, $E$ normalises $O\leq U_1\times\ldots\times U_m$ if and only if $\epsilon(E)$ stabilises $O$; the claim follows.
\end{proof}

We conclude this section with two definitions. Let $p$ be a prime, $e\in\{1,2\}$, and $U\in\mathcal{U}(p,e)$; write $C=C_{\GL_e(p)}(U)$, $N=N_{\GL_e(p)}(U)$, and $\Theta(U)=N/C$. As before, $J$ denotes the nontrivial $2\times 2$ permutation matrix. By Definition \ref{def_Upc} and Lemma \ref{lemT}, exactly one of the six cases listed below holds.
 
\begin{definition}\label{defTypes}
With the notation above, the \emph{column type} of $U$ is ($i$) if condition ($i$) holds:
\begin{ithm}
\item[\rm(0)] $p=2$ and $e=1$: here $N=C=1$.
\item[\rm(1)] $p$ is odd and $e=1$: here $N=C=\GL_1(p)$ is cyclic of order $p-1$.
\item[\rm(2)] $e=2$ and $U$ is scalar: here $N=C=\GL_2(p)$.
\item[\rm(3)] $e=2$ and $U\leq D(p)$ is non-scalar with $U^J\ne U$: here $N=C=D(p)$.
\item[\rm(4)] $e=2$ and $U\leq D(p)$ is non-scalar with $U^J=U$: here $C=D(p)$ and $N=M(p)$.
\item[\rm(5)] $e=2$ and $U=\Sigma(p,b)$ is irreducible: here $C=\Sigma(p)$ and $N=\langle\Sigma(p),\theta\rangle$ with $\theta|_{\Sigma(p)}\colon s\mapsto s^p$.
\end{ithm}  
\end{definition}

Recall that $\ker\epsilon_U=C$; thus $\Theta(U)=1$ in the types {\rm(0)--(3)}, and $|\Theta(U)|=2$ in the types (4) and (5). Since we assume that all $p_i$ are odd (see Lemma \ref{lemCasesEven}), a column of type (0) does not occur.
\enlargethispage{0.4cm}
For a subgroup $H\leq\Theta(\mathcal{U})$ we define the following set of fixed points
\[\mathcal{K}(\mathcal{U},L)^H=\{O\in \mathcal{K}(\mathcal{U},L): O^\theta=O\text{ for every $\theta\in H$}\}.\]
In Appendix \ref{secKH} we describe how to calculate $|\mathcal{K}(\mathcal{U},L)^H|$ by a combinatorial formula.

\subsection{The case $t=1$}\label{secT1}  
Recall that $\mathcal{U}=(U_1,\ldots,U_m)$ and that $C(\mathcal{U})=\prod_{i=1}^m Z_i$ and $N(\mathcal{U})=\prod_{i=1}^m N_i$ where each $Z_i=C_{\GL_{e_i}(p_i)}(U_i)$ and  $N_i=N_{\GL_{e_i}(p_i)}(U_i)$. For the case $t=1$  we consider $\mathcal{E}_1(\mathcal{U})$ as defined in \eqref{eq_mcE}: every subgroup of $N(\mathcal{U})$ of order $2$ is generated by a unique involution, so $\mathcal{E}_1(\mathcal{U})$ consists of the groups $\langle x\rangle$, where $x$ runs over a set of $C(\mathcal{U})$-class representatives of the involutions in $N(\mathcal{U})$. For $U\in\mathcal{U}(p,e)$ write $C=C_{\GL_e(p)}(U)$ and $N=N_{\GL_e(p)}(U)$, and let $\mathcal{I}(U)$ be a set of $C$-class representatives of all elements in $N$ of order dividing $2$. Since $C(\mathcal{U})$ acts componentwise on $N(\mathcal{U})$, these representatives can be chosen componentwise, that is, $x$ runs over the tuples in $\prod_{i=1}^m\mathcal{I}(U_i)$, except $x=(1,\ldots,1)$. Step 2) therefore reduces to a count inside each $\GL_{e_i}(p_i)$, and this count turns out not to depend on $p_i$. For $U\in\mathcal{U}(p,e)$ and $\xi,\psi\in\Theta(U)$ we define
\[A(U;\xi,\psi)=\big|\{\,x\in\mathcal{I}(U) : \epsilon_U(x)=\psi\ \text{ and }\ x^\xi\ \text{is $C$-conjugate to}\ x\,\}\big|;\]
here $x^\xi$ denotes the image of $x$ under conjugation by a preimage of $\xi$. These numbers determine $\gnuFF(n,\ell)$ as follows.

\begin{proposition}\label{prop_t1S}
Let $n$ be cubefree and $\ell\mid n$ with  $t=v_2(n/\ell)=1$. Then
\[\gnuFF(n,\ell)=\sum_{L\in\mathcal{A}(\nu_0)}\ \sum_{\mathcal{U}}\ \frac{1}{|\Theta(\mathcal{U})|}\sum_{\xi\in\Theta(\mathcal{U})}\ \sum_{\psi\in\Theta(\mathcal{U})}\ \Big(\prod_{i=1}^mA(U_i;\xi_i,\psi_i)\ -\ \Delta_{\psi=1}\Big)\ \big|\mathcal{K}(\mathcal{U},L)^{\langle\xi,\psi\rangle}\big|,\]
where $\mathcal{U}$ runs over the tuples in $\prod_{i=1}^m\mathcal{U}(p_i,e_i,L)$.
\end{proposition}  
\begin{proof}
  By Proposition \ref{propRedS}, for a fixed $L$ and $\mathcal{U}$ we need to evaluate $\big|\mathcal{M}(\mathcal{U},L)/\Theta(\mathcal{U})\big|$.  The Cauchy-Frobenius Lemma \cite[Lemma 2.17]{handbook} yields
\[\big|\mathcal{M}(\mathcal{U},L)/\Theta(\mathcal{U})\big|=\frac{1}{|\Theta(\mathcal{U})|}\sum\nolimits_{\xi\in\Theta(\mathcal{U})}\big|\{(E,O)\in\mathcal{M}(\mathcal{U},L) : O^\xi=O\ \text{ and }\ E^\xi=E\}\big|,\]
where $E^\xi$ uses the $\Theta(\mathcal{U})$-action on $\mathcal{E}_1(\mathcal{U})$ defined in Section \ref{secStrat}, that is, $E^\xi=E$ requires that the image of $E$ under conjugation by a preimage of $\xi$ is $C(\mathcal{U})$-conjugate to $E$. Fix $\xi\in\Theta(\mathcal{U})$ and sort the groups $E=\langle x\rangle$ in $\mathcal{E}_1(\mathcal{U})$ by $\psi=\epsilon(x)\in\Theta(\mathcal{U})$. By the description above, those with $\epsilon(x)=\psi$ and $E^\xi=E$ correspond to the tuples $(x_1,\ldots,x_m)\in\prod_{i=1}^m\mathcal{I}(U_i)$ with each $\epsilon_{U_i}(x_i)=\psi_i$ and with each $x_i^{\xi_i}$ conjugate to $x_i$ in $Z_i$. There are $\prod_{i=1}^mA(U_i;\xi_i,\psi_i)$ of these tuples, and exactly one of them (the tuple with all $x_i=1$) has to be discarded because it does not give a subgroup of order $2$;  this tuple occurs only when $\psi=1$, which explains the term $-\Delta_{\psi=1}$. For such an $E$, the pairs $(E,O)$ that lie in $\mathcal{M}(\mathcal{U},L)$ and are fixed by $\xi$ are those with $O$ fixed by $\xi$ and by $\epsilon(E)=\langle\psi\rangle$, that is, the elements of $\mathcal{K}(\mathcal{U},L)^{\langle\xi,\psi\rangle}$. Summing over $\psi$ and $\xi$ then yields the claimed formula.
\end{proof}

It remains to evaluate the numbers $A(U;\xi,\psi)$. As announced above, they do not depend on the prime $p$, but only on the column type of $U$ introduced in Definition \ref{defTypes}.

\begin{lemma}\label{lemAtabS}
Let $U\in\mathcal{U}(p,e)$ and $\xi,\psi\in\Theta(U)$. Then $A(U;\xi,\psi)$ depends only on the column type of $U$ and is given in Table \ref{tab_Afirst}, where ``$-$''  means that $\Theta(U)=1$, so no nontrivial pair $(\xi,\psi)$ exists.
\end{lemma} 
\begin{proof} 
In the types (0), (1), (2), (3) we have $N=C$ and $\Theta(U)=1$, so $\xi=\psi=1$ and $A(U;1,1)$ is the number of $C$-classes of elements $x\in C$ with $x^2=1$. For type (0) this is $1$. For the abelian groups $C=\GL_1(p)$ and $C=D(p)$ of types (1) and (3) the classes are singletons, and $C_{p-1}$ and $C_{p-1}\times C_{p-1}$ contain $2$ and $4$ such elements, respectively. For type (2), every $x\in\GL_2(p)$ with $x^2=1$ is diagonalisable with eigenvalues in $\{1,-1\}$, so there are three classes with representatives $\diag(1,1)$, $-\diag(1,1)$, and $\diag(1,-1)$.

Now let $U$ be of type (4), so $C=D(p)$ and $N=M(p)=\langle J\rangle\ltimes D(p)$, and $\epsilon_U(x)=1$ if and only if $x\in D(p)$. Recall that $A(U;\xi,\psi)$ counts those $x\in\mathcal{I}(U)$ with $\epsilon_U(x)=\psi$ for which $x^\xi$ is $C$-conjugate to $x$; thus for $\psi=1$ we consider  $x\in D(p)$ with $x^2=1$ and for $\psi\ne1$ we consider $x\in D(p)J$ with $x^2=1$. For $\psi=1$ these are the four elements $\diag(\pm1,\pm1)$, and since $D(p)$ is abelian each of them forms its own $C$-class, so $A(U;1,1)=4$. If $\xi\ne1$, then conjugation by a preimage of $\xi$ acts as conjugation by $J$, which interchanges the two coordinates: it fixes $\diag(1,1)$ and $\diag(-1,-1)$, and swaps $\diag(\pm1,\mp1)$; thus $A(U;\xi,1)=2$. For $\psi\ne1$ we have $x=\diag(u,v)J$ with $x^2=\diag(uv,uv)=1$, that is, $v=u^{-1}$; conjugating by $\diag(s,t)\in D(p)$ replaces $u$ by $us^{-1}t$, so these $p-1$ elements form a single $C$-class. Every $\xi\in\Theta(U)$ maps this class to itself, hence $A(U;\xi,\psi)=1$ for both $\xi=1$ and $\xi\ne1$.

Finally let $U$ be of type (5), so $C=\Sigma(p)$ is cyclic of order $p^2-1$ and $N=\langle C,\theta\rangle$, and $\epsilon_U(x)=1$ if and only if $x\in C$. For $\psi=1$ we consider $x\in C$ with $x^2=1$: these are  $x=\pm\diag(1,1)$ and they form two $C$-classes, so $A(U;1,1)=2$. Both are fixed by $\theta$ since $p$ is odd, so $A(U;\xi,1)=2$ for $\xi\ne1$. For $\psi\ne1$ we consider  $x=s\theta$ where $s\in C$ satisfies $1=x^2=s s^\theta=s^{1+p}$. These $s$ form the unique subgroup of order $p+1$ of $C$. Conjugating by $s_0\in C$ replaces $s$ by $s s_0^{p-1}$. Since the image of $s_0\mapsto s_0^{p-1}$ has order $p+1$, it follows  that there is a unique $C$-class. As for type (4), this class is mapped to itself by every $\xi\in\Theta(U)$, hence $A(U;\xi,\psi)=1$ for both $\xi=1$ and $\xi\ne1$.
\end{proof}

\begin{table} 
{\footnotesize  
\renewcommand{\arraystretch}{1.1}  
 \[\begin{array}{@{}l|cccccc@{}}
\text{type}&{\rm(0)}&{\rm(1)}&{\rm(2)}&{\rm(3)}&{\rm(4)}&{\rm(5)}\\\hline
\psi=1,\ \xi=1     &1&2&3&4&4&2\\
\psi=1,\ \xi\ne1   &-&-&-&-&2&2\\  
\psi\ne1            &-&-&-&-&1&1
 \end{array}\]}
 \caption{The value of $A(U;\xi,\psi)$ for each column type of $U$ (Definition \ref{defTypes}).}\label{tab_Afirst}
\end{table}

\subsection{The case $t=2$} We continue with the previous notation and consider $t=2$. In this case,  $\ell$ is odd by Lemma \ref{lemCasesEven}a) and the subgroups  $E$ we consider are isomorphic to $C_4$ or $C_2^2$. Recall that $\Aut(C_4)\cong C_2$ is generated by inversion and that $\Aut(C_2^2)\cong {\rm Sym}_3$ permutes the three subgroups of order $2$ in all possible ways. We briefly explain how Steps 2) and 3) of Section \ref{secStrat} will change.

Step 2) changes because $E$ is not necessarily cyclic, and is no longer determined by a single involution. Here a subgroup of $N(\mathcal{U})$ is isomorphic to $E$ if and only if it is the image of an injective homomorphism $E\to N(\mathcal{U})$, and two such homomorphisms have the same image if and only if they differ by an element of $\Aut(E)$. Since $N(\mathcal{U})=\prod_iN_i$, a homomorphism $E\to N(\mathcal{U})$ is a tuple $\phi=(\phi_1,\ldots,\phi_m)$ of homomorphisms $\phi_i\colon E\to N_i$, and it is injective if and only if $\bigcap_i\ker\phi_i=1$. Passing to $C(\mathcal{U})$-classes, we define
\[\mathcal{H}(E)=\Big\{(\phi_1,\ldots,\phi_m)\in\prod\nolimits_{i=1}^m\Hom(E,N_i)/Z_i\ :\ \bigcap\nolimits_{i=1}^m\ker\phi_i=1\Big\},\]
and taking coordinate projections yields a bijection between the groups $\tilde E\in\mathcal{E}_2(\mathcal{U})$ isomorphic to $E$ and the $\Aut(E)$-orbits on $\mathcal{H}(E)$: the orbit of $\phi$ corresponds to $\tilde E$ if and only if  $\phi$ is the tuple of $Z_i$-classes of the coordinate projections of an injective homomorphism $E\to N(\mathcal{U})$ with image $\tilde E$. Composing such a homomorphism with $\epsilon$ therefore gives
\begin{align}\label{eq_newpsi}\psi=(\epsilon_{U_1}\circ\phi_1,\ldots,\epsilon_{U_m}\circ\phi_m)\in\Hom(E,\Theta(\mathcal{U}))=\prod\nolimits_{i=1}^m\Hom(E,\Theta(U_i))
\end{align}
with  $\epsilon(\tilde E)=\psi(E)$. We abbreviate $\psi=(\psi_1,\ldots,\psi_m)$; note that $\psi$ is now a homomorphism, whereas in Section \ref{secT1} it was an element of $\Theta(\mathcal{U})$.
 
Step 3) changes because $\Aut(E)$ acts on $\mathcal{H}(E)$ as well, commuting with the action of $\Theta(\mathcal{U})$; counting orbits therefore becomes a Cauchy-Frobenius count for $\Theta(\mathcal{U})\times\Aut(E)$ rather than for $\Theta(\mathcal{U})$ alone. 

It turns out that the number `replacing' $\prod_{i=1}^mA(U_i;\xi_i,\psi_i)-\Delta_{\psi=1}$ of Proposition \ref{prop_t1S} is
\begin{align}\label{eq_newA}A(\mathcal{U};E,\xi,\alpha,\psi)=\big|\{\phi\in\mathcal{H}(E) : \epsilon_{U_i}\circ\phi_i=\psi_i\ \text{ and }\ \phi_i^{\xi_i}\circ\alpha=\phi_i\ \text{ for all }i\}\big|,
\end{align}
where  $\xi\in\Theta(\mathcal{U})$, $\alpha\in\Aut(E)$, $\psi\in\Hom(E,\Theta(\mathcal{U}))$, and the meaning of $ \phi_i^{\xi_i}\circ\alpha=\phi_i$ is explained in the following remark.

\begin{remark}\label{rem_classes}
 By definition, the elements of $\mathcal{H}(E)$ are $C$-classes of homomorphisms, so the condition  $\phi_i^{\xi_i}\circ\alpha=\phi_i$ in \eqref{eq_newA} is an equation of $Z_i$-classes. Recall that  $\phi_i^{\xi_i}$ is obtained by composing $\phi_i$ with conjugation by a preimage of $\xi_i$ in $N_i$, and choosing a different preimage  replaces $\phi_i^{\xi_i}$ by a $Z_i$-conjugate. This means that only the $Z_i$-class of $\phi_i^{\xi_i}$ is well defined. The condition $\epsilon_{U_i}\circ\phi_i=\psi_i$ does not depend on the chosen representative since $Z_i=\ker\epsilon_{U_i}$.
\end{remark}

\enlargethispage{0.6cm}

\begin{proposition}\label{prop_t2S}
Let $n$ be cubefree and $\ell\mid n$ with $t=v_2(n/\ell)=2$. Then 
\begin{equation*}
  \begin{array}{lll} 
   \gnuFF(n,\ell) = & \sum_{L\in\mathcal{A}(\nu_0)}\ \sum_{\mathcal{U}}\ \sum_{E\in\{C_2^2,C_4\}}\ \frac{1}{|\Theta(\mathcal{U})|\,|\Aut(E)|}\;\cdot \\[2ex] & \cdot\; \sum_{\xi\in\Theta(\mathcal{U})}\ \sum_{\alpha\in\Aut(E)}\ \sum_{\psi\in\Hom(E,\Theta(\mathcal{U}))}\ A(\mathcal{U};E,\xi,\alpha,\psi)\ \big|\mathcal{K}(\mathcal{U},L)^{\langle\xi,\,\psi(E)\rangle}\big|,
  \end{array}
\end{equation*}
where $\mathcal{U}$ runs over the tuples in $\prod_{i=1}^m\mathcal{U}(p_i,e_i,L)$.\end{proposition}
\begin{proof}   
By Proposition \ref{propRedS}, for a fixed $L$ and $\mathcal{U}$ we need to evaluate $\big|\mathcal{M}(\mathcal{U},L)/\Theta(\mathcal{U})\big|$. By the bijection above, the pairs $(\tilde E,O)\in\mathcal{M}(\mathcal{U},L)$ with $\tilde E\cong E$ correspond to the pairs $(\phi,O)$, where $\phi\in\mathcal{H}(E)$ is an $\Aut(E)$-orbit representative corresponding to $\tilde E$, and $O\in\mathcal{K}(\mathcal{U},L)$ is stabilised by $\psi(E)$; here $\psi=(\epsilon_{U_1}\circ\phi_1,\ldots,\epsilon_{U_m}\circ\phi_m)$ is as in \eqref{eq_newpsi}, so that $\psi(E)=\epsilon(\tilde E)$. Note that the actions of $\Theta(\mathcal{U})$ and $\Aut(E)$  on $\mathcal{H}(E)$ commute. Both actions extend to the pairs $(\phi,O)$: the group $\Theta(\mathcal{U})$ acts by $(\phi,O)\mapsto(\phi^\xi,O^\xi)$, which is well defined on $Z_i$-classes by Remark \ref{rem_classes}; the group  $\Aut(E)$ acts by $(\phi,O)\mapsto(\phi\circ\alpha,O)$, which is well defined because it leaves $O$ unchanged and because $\psi\circ\alpha$ has the same image $\psi(E)$ as $\psi$. Since the correspondence above is $\Theta(\mathcal{U})$-equivariant, it follows that $|\mathcal{M}(\mathcal{U},L)/\Theta(\mathcal{U})|$ equals the sum of the numbers of $(\Theta(\mathcal{U})\times\Aut(E))$-orbits on the pairs $(\phi,O)$, where $E$ runs over the two isomorphism types. By the Cauchy-Frobenius Lemma \cite[Lemma 2.17]{handbook}, we determine $|\mathcal{M}(\mathcal{U},L)/\Theta(\mathcal{U})|$ as
\[\sum_{E\in\{C_4,C_2^2\}}\frac{1}{|\Theta(\mathcal{U})|\,|\Aut(E)|}\sum_{\xi\in\Theta(\mathcal{U})}\ \sum_{\alpha\in\Aut(E)}\big|\{(\phi,O) : \phi^\xi\circ\alpha=\phi,\ O^\xi=O,\ \psi(E)\ \text{stabilises}\ O\}\big|.\]
For a fixed $\phi$ the occurring subgroups $O$ are exactly the ones in $\mathcal{K}(\mathcal{U},L)^{\langle\xi,\psi(E)\rangle}$, and for a fixed $\psi$ the number of relevant $\phi$ is $A(\mathcal{U};E,\xi,\alpha,\psi)$ by definition. Sorting by $\psi$  yields the claim.
\end{proof}

In Appendix \ref{secKH} we show how the numbers $A(\mathcal{U};E,\xi,\alpha,\psi)$ and cardinalities $|\mathcal{K}(\mathcal{U},L)^{\langle \xi,\psi(E)\rangle}|$ can be determined by a combinatorial calculation; these arguments are technical and mainly relevant for the implementation, which is why we have deferred them.

\subsection{The main formula for arbitrary order}\label{sec_maineven}
Combining the three cases $t=0,1,2$ with Proposition \ref{propredFF} yields our formula for arbitrary cubefree orders. Recall the definitions of $\mathcal{S}(n)$ and $\mathcal{Q}(n)$ from Definition~\ref{def_setup}.

\begin{theorem}\label{thm_maineven}
The number of isomorphism types of groups of cubefree order $n\geq 1$ is
\[\gnu(n)=\sum\nolimits_{a\in\mathcal{S}(n)}\ \sum\nolimits_{d\mid\mathcal{Q}(n/a)}\ \sum\nolimits_{\ell\mid (n/(ad))}\gnuFF\big(n/(ad),\ell\big),\]
where $\gnuFF(n/(ad),\ell)$ is given by Corollary \ref{cor_Burnside},  Proposition \ref{prop_t1S}, or Proposition \ref{prop_t2S} depending on whether the value $t= v_2(n/(ad\ell))$ is $0$, $1$, or $2$, respectively. In particular, if $n$ is odd, then $\mathcal{S}(n)=\{1\}$ and $t=0$, and this formula reduces to Theorem \ref{thm_mainodd}.
\end{theorem} 

The formula in each of the three cases $t\in\{0,1,2\}$ is combinatorial in the sense of Definition \ref{def_CF}; see Section \ref{secComb} for a summary and Appendix \ref{secKH} for the technical results.

\begin{remark}\label{rem_redsq}\revnew{We briefly explain how for squarefree $n$ our formula for $\gnu(n)$ collapses to H\"older's formula \eqref{eq_holder}. Clearly $\mathcal{Q}(n)=1$ and $\mathcal{S}(n)=\{1\}$, so $\gnu(n)=\sum_{d\mid n}\gnuFF(n,d)$. Fix $d\mid n$ with prime factorisation $d=p_1\cdots p_m$, so $S\cong C_d$ and $\Aut(S)=\prod_{i=1}^mC_{p_i-1}$ is abelian. For odd $n$, Corollary \ref{cor_Burnside} shows that $\gnuFF(n,d)=\sum_\mathcal{U}\prod_{q\mid \nu} |\mathcal{K}_q(\mathcal{U},L)|$ where $\nu=n/d$ and $L\cong C_\nu$; recall that $\Theta(\mathcal{U})=1$ for every tuple. For a prime $q\mid\nu$ let $\mathcal{S}_q=\{i : p_i\equiv1\bmod q\}$, so $|\mathcal{S}_q|=w(q)$ as in \eqref{eq_holder}. Choosing a tuple   $\mathcal{U}$ amounts to choosing, for each  $q\mid\nu$, the set of those $i$ with $q$ dividing $|U_i|$, that is, a subset of $\mathcal{S}_q$.  If this subset has size $s\geq1$, there are $|\mathcal{K}_q(\mathcal{U},C_\nu)|=(q-1)^{s-1}$ corresponding subgroups  by Lemma~\ref{lemSizes}. Since a tuple is an independent choice of subset for each $q$, summing over all tuples counts $\prod_{q\mid\nu}\sum_{s=1}^{w(q)}\binom{w(q)}{s}(q-1)^{s-1}$; this yields a summand of H\"older's formula:
\[(\ast)\quad \gnuFF(n,d)=\prod\nolimits_{q\mid\nu}\ \sum\nolimits_{s=1}^{w(q)}
  \binom{w(q)}{s}(q-1)^{s-1} 
  =\prod\nolimits_{q\mid\nu}\frac{q^{w(q)}-1}{q-1}.\]
 For even $n$, the term $\gnuFF(n,d)$ is evaluated as before if $\nu$ is odd or via  Proposition \ref{prop_t1S} if  $\nu$ is even. In the latter case, the formula has an additional factor $(\prod_{i=1}^m A(U_i;\xi_i,\psi_i)-\Delta_{\psi=1})$ where  $\xi_i=\psi_i=1$ and $U_i$ has column type (1). Lemma \ref{lemAtabS} shows that each $A(U_i;\xi_i,\psi_i)=2$, so this factor becomes $2^{w(2)}-1=\sum_{s=1}^{w(2)}\binom{w(2)}{s}(2-1)^{s-1}$. We obtain H\"older's summand $(\ast)$  again, including $q=2$. H\"older's formula evaluates the closed form on the right, whereas ours calculates the sum on the left. For general cubefree $n$ the corresponding sum still factorises over the columns, but its factors involve the sets $\mathcal{U}(p,2,L)$ of Definition \ref{def_Upc} and the average over $\Theta(\mathcal{U})$, for which we know no closed form. Our implementation therefore evaluates the sum term by term also for squarefree $n$, which is why it is markedly slower than H\"older's formula. } 
\end{remark}

\begin{remark}\label{rem_closed}
  While our formula for general cubefree $n$ is combinatorial, one wonders whether there is also a closed formula similar to H\"older's result for squarefree   $n$. The results of \cite{sot} make this seem unlikely: for example, the precise formula for $\gnu(p^2qr)$ in \cite[Theorem 2.1]{sot} already has twenty summands, most of them \emph{divisibility indicators}; these indicators record congruences between the involved primes, which is essentially the data encoded by our projection tuples.  We therefore doubt that a uniform closed formula exists for general cubefree $n$; see also the discussion on \cite[p.~235]{blackburn}.
\end{remark}

\begin{remark}\label{rem_subclasses}
By the structure of our formula for $\gnu$, it is easy to derive cubefree counting formulas for solvable groups or groups with specified Frattini subgroup or socle order. Here we consider supersolvable groups, which are solvable groups that have a normal series with cyclic factors.  It follows from   \cite[Satz VI.8.6]{hupp} that a solvable $G$ is supersolvable if and only if $G/\Phi(G)=K\ltimes S$ is, where $S=\soc(G/\Phi(G))$ and $K\leq \Aut(S)$. It is now easy to see that $G$ is supersolvable if and only if $G$ is solvable and each projection $K_i\leq \GL_{e_i}(p_i)$ of $K$ with $e_i=2$ is reducible. We can count supersolvable groups by restricting the projections we allow. If $n$ is odd, then by Lemma \ref{lemProj} it suffices to remove from each set $\mathcal{U}(p,2,L)$ of \eqref{eq_defUpel} the irreducible subgroups of the Singer cycle. If $n$ is even, then each $K_i=E_i\ltimes U_i$ where $|U_i|$ is odd and $E_i$ is the projection of the Sylow $2$-subgroup $E$ of $K$. In the notation of Definition \ref{defTypes}, here we have to discard the columns of type (5), the columns of type (4) with $\epsilon(E)_i\ne1$, and the columns of type (2) for which $E_i\cong C_4$ is irreducible (which forces $p_i\equiv3\bmod 4$). Thus, a straightforward  modification of our formula for $\gnu$ allows us to count  cubefree supersolvable groups. Another small adjustment allows us to count only the C-groups \cite{cgroups,murty} of a given cubefree order $n$. Functions for these special classes of groups are provided by our implementation.
\end{remark}


\section{Asymptotic bounds: Theorems A, B, and C}\label{sec_asym}
Our formulas allow a new  asymptotic analysis of $\gnu(n)$. For cubefree order $n$, the survey in \cite[Section 24.1]{blackburn} shows that $\gnu(n)<n^8$, and it is conjectured in \cite[Conjecture 21.16]{blackburn} that $\gnu(n)<n^2$. Better bounds have recently been proved by Kumar-Venkataraman \cite{KV}, for example,  they showed that there are at most $n^4$ solvable groups of cubefree order~$n$.

The aim of this section is to prove Theorems A, B, and C.  We start with the following preliminary lemma; recall the  definition of $\mathcal{U}(p,e)$ from Definition \ref{def_Upc}.

\begin{lemma}\label{lemE4}
Let $p$ be a prime and $e\in\{1,2\}$. For $U\in\mathcal{U}(p,e)$, the number of $C_{\GL_e(p)}(U)$-conjugacy classes of subgroups of $N_{\GL_e(p)}(U)$ of order dividing $4$ is at most $19$.
\end{lemma}
\begin{proof}
  Write $N=N_{\GL_e(p)}(U)$ and $C=C_{\GL_e(p)}(U)$. If $e=1$, or if $e=2$ and $U$ is scalar, then $N=C=\GL_e(p)$ and the claim follows from  \cite[Lemma 5.2]{cfisom} or a direct calculation.   Now suppose $e=2$ and  $U$ is not scalar. The proof of Lemma \ref{lemT} (with Lemma \ref{lemCasesEven} for $p=2$) shows that $C\in\{D(p),\Sigma(p)\}$ is abelian of rank at most $2$, and $|N/C|\leq2$. Now let $E\leq N$ with $|E|$ dividing~$4$.

  If $E\leq C$, then its elements lie in $V=\{x\in C: x^4=1\}$. Since $V$ is abelian, $\rk(V)\leq 2$, and $|V|\leq 16$, it is easy to see that there  are at most $11$ classes of such subgroups. If $E\not\leq C$, then $E\cap C$ has index $2$ in $E$, and $E=\langle x,z\rangle$ for some $x\in N\setminus C$ with $x^4=1$ and some $z\in C_C(x)$ with $z^2=1$. We first show that there are at most two $C$-classes of such $x$.

  If $C=D(p)$, then $x=J\diag(u,v)$ where $J$ is the nontrivial permutation matrix and $\diag(u,v)\in C$. Note that  $x^2=\diag(uv,uv)$ has order dividing $2$, so $uv\in\{\pm 1\}$ and $u=\pm v^{-1}$. Thus, there are $2(p-1)$ choices for $x$. The $C$-orbit of $x$ has size $p-1$, so there are at most two $C$-classes of elements~$x$. 

  If $C=\Sigma(p)$, then $x=\theta s$ where $s\in C$ and $\theta$ has order $2$ and acts as $t\mapsto t^p$ on $\Sigma(p)$. Since $x^2=s^{1+p}$ has order dividing $2$, the element $s$ lies in the subgroup of $\Sigma(p)$ of order $\gcd(2(p+1),p^2-1)$, which divides $2(p+1)$. Note  that $x^u=xu^{1-p}$ for $u\in C$, so the $C$-classes in $N\setminus C$ correspond to the cosets of $\{u^{p-1}: u\in\Sigma(p)\}$, which is the subgroup of order $p+1$ of $\Sigma(p)$. Together, there are at most two $C$-classes of elements~$x$. 

  For both $C\in\{D(p),\Sigma(p)\}$, the group $C_C(x)$ is abelian of rank at most $2$, so it has at most $3$ involutions. Thus, there are at most $2\cdot 4=8$ classes of subgroups $E\not\leq C$. Together with the at most $11$ classes of subgroups $E\leq C$, there are at most $19$ classes if $e=2$ and $U$ is non-scalar. 
\end{proof}

We can now prove Theorem A, which establishes \cite[Conjecture 21.16]{blackburn} up to a factor $n^{o(1)}$.

\begin{proposition}\label{prop_upper}
If $n$ is cubefree, then $\gnu(n)\leq n^{2+o(1)}$.
\end{proposition}
\begin{proof}
We first recall some number theory.  For an integer $b>1$, denote by $\omega(b)$ the number of distinct prime divisors of $b$ and by $\tau(b)$ the number of divisors of $b$. By \cite[Theorem~13.12]{apostol}, for every $\varepsilon>0$ we have $\tau(b)<2^{(1+\varepsilon)\log b/\log\log b}$ for all sufficiently large $b$; in particular, $\tau(b)=b^{o(1)}$. Since $\tau(b)\geq 2^{\omega(b)}$, this also gives $\omega(b)\leq\log_2\tau(b)=O(\log b/\log\log b)$.

  Proposition \ref{propredFF}a,b) shows that $\gnu(n)=\sum_{a\in\mathcal{S}(n)}\sum_{d\mid\mathcal{Q}(n/a)}\sum_{{\ell\mid (n/(ad))}}\gnu_\Phi(n/(ad),\ell)$. Note that each $a\in \mathcal{S}(n)$ satisfies $a=1$ or $a=p(p^2-1)/2$ for some prime divisor $p\mid n$, so $|\mathcal{S}(n)|\leq \omega(n)+1$. Thus, the number of triples $(a,d,\ell)$ is at most $(\omega(n)+1)\cdot 2^{\omega(n)}\cdot \tau(n)$ and can be bounded by $n^{o(1)}$. Since every triple $(a,d,\ell)$ occurring in that sum satisfies $n'=n/(ad)\mid n$ and $\ell\mid n'$, we have
  \[\gnu(n) \leq n^{o(1)}\max\big\{\gnu_\Phi(n',\ell) :  n'\mid n \text{ and } \ell\mid n'\big\},\]
and it remains to bound  $\gnu_\Phi(n',\ell)$ by $n^{2+o(1)}$ for every such $n'$ and $\ell$. We emphasise that all bounds established below depend on $\ell$ and $n$ only, and not on $n'$.
 
So let $n'\mid n$ and $\ell\mid n'$, and write $\nu={n'/\ell}=2^t\nu_0$ with $\nu_0$ odd; note that $t\leq 2$. Let $S$ be the isomorphism type of a socle of order $\ell=p_1^{e_1}\ldots p_m^{e_m}$, so $\Aut(S)=\prod_{i=1}^m\GL_{e_i}(p_i)$. Proposition \ref{propPairs} shows that $\gnuFF(n',\ell)$ is the number of $\Aut(S)$-orbits on the set $\mathcal{P}$ of pairs $(E,O)$ with $O\leq \Aut(S)$ of order $|O|=\nu_0$ and $E\leq N_{\Aut(S)}(O)$ of order $|E|=2^t$. Recall that each such $O$ is abelian: the proof of Proposition \ref{propPairs} shows that $EO\in\mathcal{K}$, and $O$ is a normal Hall $2'$-subgroup of $EO$, hence $O=O_{2'}(EO)$ and Lemma \ref{lem_kHall} applies.

\smallskip 

{\bf (1)} We normalise the orbit representatives. Each projection $\pi_i(O)$ of $O$ into $\GL_{e_i}(p_i)$ has cubefree odd order coprime to $p_i$, so up to $\Aut(S)$-conjugacy we can assume that each $\pi_i(O)$ is one of the representatives $U_i\in\mathcal{U}(p_i,e_i)$. (Recall from Lemma \ref{lemCasesEven} that Definition~\ref{def_Upc} also covers the case $p_i=2$; if $t\geq 1$ and $p_i=2$, then Lemma \ref{lemCasesEven} forces $e_i=1$ and $\GL_1(2)=1$.) Thus $O\in\mathcal{K}(\mathcal{U},L)$, where $\mathcal{U}=(U_1,\ldots,U_m)$ is a projection tuple and $L\in\mathcal{A}(\nu_0)$ is the isomorphism type of $O$; in particular, each $U_i$ lies in $\mathcal{U}(p_i,e_i,L)$, see \eqref{eq_defUpel}. Recall the definitions of $N_i(U_i)$ and $Z_i(U_i)$ from  Notation~\ref{notEvenN}.
Since $E$ normalises $O$, it normalises each $U_i=\pi_i(O)$, so $E\leq N(\mathcal{U})=\prod_{i=1}^mN_i(U_i)$. The group $C(\mathcal{U})=\prod_{i=1}^m Z_i(U_i)$ acts trivially on the set of subgroups of $U_1\times\ldots\times U_m$, so replacing $E$ by a $C(\mathcal{U})$-conjugate does not change  $O$ or the $\Aut(S)$-orbit of $(E,O)$. We can therefore assume that  $E\in\mathcal{E}_t(\mathcal{U})$, see \eqref{eq_mcE}. Together, we have
\[(\ast)\quad \gnuFF(n',\ell)\ \leq \sum\nolimits_{L\in\mathcal{A}(\nu_0)}\ \sum\nolimits_{\mathcal{U}}\ \big|\mathcal{E}_t(\mathcal{U})\big|\cdot\big|\mathcal{K}(\mathcal{U},L)\big|,\]
where $\mathcal{U}$ runs over the tuples in $\prod_{i=1}^m\mathcal{U}(p_i,e_i,L)$; note that $|\mathcal{A}(\nu_0)|\leq 2^{\omega(n)}=n^{o(1)}$.
 
\smallskip

{\bf (2)} We consider the groups $E$. Each projection $\pi_i(E)$ is a subgroup of $N_i(U_i)$ of order dividing $4$, and $C(\mathcal{U})$ acts columnwise. By Lemma \ref{lemE4}, there are at most $19^m$ possibilities for the tuple $(\pi_1(E),\ldots,\pi_m(E))$ up to $C(\mathcal{U})$-conjugacy. Once such a tuple is fixed, $E$ is a subgroup of $\pi_1(E)\times\ldots\times\pi_m(E)$  generated by at most two elements. Thus, there are at most $(4^m)^2=16^m$ possibilities for $E$, and so  $|\mathcal{E}_t(\mathcal{U})|\leq 304^m$. Since $m\leq\omega(n)=o(\log n)$, we have $|\mathcal{E}_t(\mathcal{U})|\leq n^{o(1)}$ for all $\mathcal{U}$ and $t$.

\smallskip

{\bf (3)} We consider the groups $O$. The bound $|\mathcal{E}_t(\mathcal{U})|\leq 304^m$ holds for every $\mathcal{U}$ and $t$, and since  $|\mathcal{A}(\nu_0)|$ is bounded by $n^{o(1)}$, it follows from  $(\ast)$ that
\[\gnuFF(n',\ell) \leq 304^m\sum\nolimits_{L\in\mathcal{A}(\nu_0)} \sum\nolimits_{\mathcal{U}}\big|\mathcal{K}(\mathcal{U},L)\big| \leq n^{o(1)}\max_{L\in\mathcal{A}(\nu_0)} \sum\nolimits_{\mathcal{U}}\big|\mathcal{K}(\mathcal{U},L)\big|.\]
Since $2^m\leq 2^{\omega(n)}=n^{o(1)}$ and $\ell\leq n$, the claim follows once we have proved that for every $L\in\mathcal{A}(\nu_0)$:
\[(\dagger)\quad \sum\nolimits_{\mathcal{U}}\big|\mathcal{K}(\mathcal{U},L)\big| \leq 2^m\ell^2.\]

{\bf (4)} We prove $(\dagger)$. Fix $L\in\mathcal{A}(\nu_0)$. We want to use Lemma \ref{lemSizes} and for this we introduce the following notation where  the functions $c$, $g$, and $r$ are given before Lemma \ref{lemSizes}. For a prime $q\mid\nu_0$ and a finite abelian $q$-group $X$ define
\[\big(F_q(X),D_q\big)=\begin{cases}\big(c(\rk(X)),\ q-1\big)&\text{if }L_q\cong C_q,\\[2pt] \big(g(v_q(|X|)),\ q(q-1)\big)&\text{if }L_q\cong C_{q^2},\\[2pt] \big(r(\rk(X)),\ |\GL_2(q)|\big)&\text{if }L_q\cong C_q^2.\end{cases}\]
Since $g(d)\geq c(d)\geq z(d)\geq 0$ for every $d\in\{0,1,2\}$, in each case of Lemma \ref{lemSizes} the terms following the leading product contribute $\leq 0$: this is clear for cyclic $L_q$, and for $L_q\cong C_q^2$ we have $(q+1)\prod_ic(d_i)\geq q\prod_iz(d_i)$. Hence Lemma \ref{lemSizes} implies that
\[\big|\mathcal{K}_q(\mathcal{U},L)\big|\ \leq\ D_q^{-1}\prod\nolimits_{i=1}^mF_q\big((U_i)_q\big).\]
The bijection \eqref{eqSplit} gives $|\mathcal{K}(\mathcal{U},L)|=\prod_{q\mid\nu_0}|\mathcal{K}_q(\mathcal{U},L)|$, and the entries of a projection tuple are chosen independently of one another. Since each  $D_q\geq1$, we therefore have the following estimate
\[\sum\nolimits_{\mathcal{U}}\big|\mathcal{K}(\mathcal{U},L)\big| \leq \prod\nolimits_{i=1}^mT_i,\qquad\text{where each}\quad T_i=\sum_{U\in\mathcal{U}(p_i,e_i,L)} \prod\nolimits_{q\mid\nu_0}F_q(U_q).\]
Since $\ell^2=\prod_{i=1}^mp_i^{2e_i}$, the estimate $(\dagger)$ is established once we prove that each $T_i\leq 2p_i^{2e_i}$. Thus, we fix one $i$ and write $p=p_i$ and $e=e_i$. Recall from Definition \ref{def_Upc} that every $U\in\mathcal{U}(p,e,L)$ lies in $\GL_1(p)$ if $e=1$, and that for $e=2$ it either lies in $D(p)$ or is one of the irreducible groups $\Sigma(p,b)$. Accordingly, we split $T_i=T^{\rm red}+T^{\rm irr}$, where $T^{\rm red}$ is the contribution of those $U$ with $e=1$, or with $e=2$ and $U\leq D(p)$, and where $T^{\rm irr}$ is the contribution of the irreducible $U$ for $e=2$.

{\bf (4.1)} Every $U$ counted by $T^{\rm red}$ is the direct product of its Sylow subgroups, so $T^{\rm red}\leq\prod_{q\mid\nu_0}\sigma_q$, where $\sigma_q$ is the sum of $F_q(X)$ over all subgroups $X$ of $\GL_1(p)_q$ and $D(p)_q$, respectively, that are quotients of $L_q$. We now make a case distinction.
 
$\bullet$\;Let $e=2$. If $q\nmid p-1$, then $\sigma_q=1$, so let $q\mid p-1$. Note that  $D(p)_q\cong C_{q^u}\times C_{q^u}$ for some $u\geq1$. It has $q+1$ subgroups of order $q$, exactly one elementary abelian subgroup of order $q^2$, and at most $q^2+q$ cyclic subgroups of order $q^2$. Since the quotients of $C_q$ are $\{1,C_q\}$, for $L_q\cong C_q$ we have
$\sigma_q\leq 1\cdot c(\rk(1))+(q+1)\cdot c(\rk(C_q))=1+(q+1)(q-1)=q^2$. Similarly, the quotients of $C_{q^2}$ and of $C_q^2$ are $\{1,C_q,C_{q^2}\}$ and $\{1,C_q,C_q^2\}$, respectively, and we have
\begin{align*} 
L_q\cong C_{q^2}:\qquad &\sigma_q\leq 1+(q+1)(q-1)+(q^2+q)\,q(q-1)=q^4,\\
L_q\cong C_q^2:\qquad &\sigma_q\leq 1+(q+1)(q^2-1)+(q^2-1)(q^2-q)=q^4.
\end{align*}
As the primes $q$ are pairwise distinct divisors of $p-1$, we conclude $T^{\rm red}\leq (p-1)^4<p^4=p^{2e}$.

$\bullet$\; If $e=1$, then the same calculation gives $\sigma_q\leq q^2$ and $T^{\rm red}\leq (p-1)^2<p^2=p^{2e}$.

{\bf (4.2)} Now let $e=2$ and consider the irreducible groups $\Sigma(p,b)$ with $b\mid p^2-1$ and $b\nmid p-1$. These groups are cyclic and pairwise distinct, so, as above, $T^{\rm irr}\leq\prod_{q\mid\nu_0}\sigma_q$ and $\sigma_q\leq q^2$ for every $q$. The relevant primes $q$ are pairwise distinct divisors of $p^2-1$, whence $T^{\rm irr}\leq (p^2-1)^2<p^4=p^{2e}$.

Together, $T_i=T^{\rm red}+T^{\rm irr}\leq 2p_i^{2e_i}$ for every $i$, which proves $(\dagger)$, and hence the proposition.  
\end{proof}

We complement Proposition \ref{prop_upper} by the following lower bound, which proves Theorem C. It shows that the exponent $2$ in Theorem A cannot be improved; the proof relies on \cite[Theorem 21.7]{blackburn}.

\begin{proposition}\label{prop_lower}
There is an infinite set $\mathcal{N}$ of cubefree integers such that $\gnu(n)\geq n^{2-o(1)}$ for $n\in\mathcal{N}$.
\end{proposition}
\begin{proof}
  Recall the functions $c$, $r$, and $z$ from Section \ref{secStep1}. Recall from the proof of Proposition \ref{prop_upper} that $\tau(b)=b^{o(1)}$, where $\tau(b)$ is the number of divisors of the integer $b\geq 1$. We denote by $\omega(b)$ the number of distinct prime divisors of $b$.

  {\bf(1)} By \cite[Theorem 21.7]{blackburn} there is an infinite set $\mathcal{M}$ of squarefree integers with $\gnu(N)\geq N^{1-o(1)}$ for $N\in\mathcal{M}$; from each such $N$ we construct a cubefree $n=n(N)\leq N^2$. Throughout, all asymptotic statements refer to $N\to\infty$ in $\mathcal{M}$. We now consider such an $N\in\mathcal{M}$. H\"older's formula \eqref{eq_holder} exhibits $\gnu(N)=\sum_{d\mid N} A_d$ as a finite sum of $\tau(N)$ non-negative terms $A_d$; if $d=d(N)$ maximises $A_d$, then $\gnu(N)\leq \tau(N)A_d$, so $A_d\geq \gnu(N)/\tau(N)$. Since $\tau(N)=N^{o(1)}$ and $\gnu(N)\geq N^{1-o(1)}$, the term $A_d$ of this divisor $d$ satisfies  \[(\ast_1)\quad A_d=\prod\nolimits_{q\mid (N/d)}\tfrac{q^{w(q)}-1}{q-1} \geq N^{1-o(1)},\]  where $w(q)$ is as in \eqref{eq_holder} and  $q$ runs over the primes dividing $N/d$. The factor for $q=2$ is $2^{w(2)}-1\leq 2^{\omega(N)}=N^{o(1)}$, so we may delete it and assume that all primes $q$ occurring in the product are odd. Let $q_1,\ldots,q_s$ be these primes and let $p_1,\ldots,p_m$ be those primes $p\mid d$ with $p\equiv 1\bmod q_j$ for at least one $j$; since the product above is nonzero, we have $s\geq 1$ and each $w(q)\neq 0$, so that $m\geq 1$. Note that $p_1,\ldots,p_m$ are odd and pairwise distinct, and distinct from $q_1,\ldots,q_s$ as $N$ is squarefree. Write $m_j=|\{i : q_j\mid p_i-1\}|$, so that  $m_j=w(q_j)>0$. As $(q^{w}-1)/(q-1)\leq 2q^{w-1}$ for all $w\geq 1$ and $2^s\leq 2^{\omega(N)}=N^{o(1)}$, it follows from $(\ast_1)$ that
\[(\ast_2)\qquad\prod\nolimits_{j=1}^sq_j^{m_j-1}\geq  N^{1-o(1)}.\]
We set $n=(q_1\cdots q_s\,p_1\cdots p_m)^2$, which is cubefree and satisfies $n\leq N^2$.
  
{\bf(2)} We show that \[(\ast_3)\quad \gnu(n)\geq 2^{-m}\prod\nolimits_{j=1}^s\big((q_j^2-1)(q_j^2-q_j)\big)^{m_j-1}.\] Write $\kappa=q_1\cdots q_s$ and let $b_i$ be the product of those $q_j$ that divide $p_i-1$, so $b_i>1$ divides both $\kappa$ and $p_i-1$. All $q_j$ and $p_i$ are odd, hence so is $n$, and our results of Section \ref{secOdd} apply. Let $\ell=(p_1\cdots p_m)^2$ and let $S=C_{p_1}^2\times\ldots\times C_{p_m}^2$ be a socle type of order $\ell$, so that $n=\ell\kappa^2$ and $\Aut(S)=\prod_{i=1}^m\GL_2(p_i)$. Let $L=C_{\kappa}\times C_{\kappa}$ and for each $i$ define $U_i=\{x\in D(p_i) : x^{b_i}=1\}$. Then each $U_i\cong C_{b_i}\times C_{b_i}$ has odd cubefree order coprime to $p_i$ and is a quotient of $L$. By construction, $U_i$ is invariant under the swap $J$, so $U_i\in\mathcal{U}(p_i,2,L)$, see Definition \ref{def_Upc}. Since $b_i>1$, it is not scalar, hence of type (4). Thus, the tuple $\mathcal{U}=(U_1,\ldots,U_m)$ satisfies $|\Theta(\mathcal{U})|=2^m$. Now fix $q=q_j$ and use the notation of Section \ref{secStep1}: here $L_q\cong C_q^2$, and $\rk((U_i)_q)=2$  for exactly $m_j$ indices $i$, and $\rk((U_i)_q)=0$ for the other ones. Since $r(2)=(q^2-1)(q^2-q)=|\GL_2(q)|$, $r(0)=1$, and $c(2)=z(2)=0$, Lemma \ref{lemSizes} yields $|\mathcal{K}_q(\mathcal{U},L)|=\big((q^2-1)(q^2-q)\big)^{m_j-1}$. Thus \eqref{eqSplit}, Corollary \ref{corRed}, and Proposition \ref{propredFF}a,b) (with $a=d=1$) give $\gnu(n)\geq \gnuFF(n,\ell)\geq|\mathcal{K}(\mathcal{U},L)|/|\Theta(\mathcal{U})|$, which is $(\ast_3)$.
 
  {\bf(3)} Write  $M=\prod\nolimits_{j=1}^s\big[(1-q_j^{-1})(1-q_j^{-2})\big]^{m_j-1}$; since $(q^2-1)(q^2-q)=q^4(1-q^{-1})(1-q^{-2})$, the estimates $(\ast_2)$ and $(\ast_3)$ yield
\[(\ast_4)\quad \gnu(n) \geq 2^{-m}M\Big(\prod\nolimits_{j=1}^sq_j^{m_j-1}\Big)^{4} \geq 2^{-m}MN^{4-o(1)}.\]
Note that
if $0<x\leq 1/2$, then  $-\log(1-x)\leq 2x$, which yields
\[(\ast_5)\quad -\log M\leq 4\sum\nolimits_{j=1}^sm_j/q_j.\] Moreover, $\sum_{j=1}^s m_j/q_j=\sum_{i=1}^m\sum_{q_j\mid p_i-1}q_j^{-1}$, and each inner sum runs over at most $s\leq\omega(N)$ distinct primes, hence is at most $\sum_{q\leq P}1/q$ where $P$ is the $\omega(N)$-th prime.  Recall from \cite[Theorem~13.12]{apostol} that $\omega(N) = O(\log N/\log\log N)$, and so  $P=O(\log N)$ by \cite[Theorem~4.7]{apostol}. Now \cite[Theorem~4.12]{apostol} shows  that each  $\sum_{q\leq P}1/q$ is $O(\log\log\log N)$. As $m\leq\omega(N)=O(\log N/\log\log N)$, we conclude that \[\sum\nolimits_{j=1}^s m_j/q_j\leq O(\log N/ \log\log N)O(\log\log\log N)=o(\log N)\] and therefore $(\ast_5)$ implies $M=N^{-o(1)}$. Since $2^m\leq 2^{\omega(N)}=N^{o(1)}$ and $n\leq N^2$, it follows from $(\ast_4)$ that \[\gnu(n)\geq N^{4-o(1)}\geq n^{2-o(1)}.\] Since $\gnu(n)\to\infty$ as $N\to\infty$, the integers $n$ obtained in this way form an infinite set $\mathcal{N}$. 
\end{proof}

\begin{example}\label{ex_gnultn}
Our implementation (see Section \ref{secComp}) confirms that  $\gnu(n)\leq n$ for all cubefree $n\leq 10^8$, but Proposition \ref{prop_lower} shows that this  is not true in general. Indeed, a $7$-second calculation reveals that $\gnu(n)>10n$ for  $n=(1009.10091.12109.30271.40361.42379.64577.72649)^2$: we calculate
\[\gnu(n)=9696567810757330308711244727502817860233447772217755022123581571740512.\]
\revnew{The order $n= 2^2.3^2.5^2.7^2.11^2.31^2.61^2.151=2{,}881{,}270{,}049{,}219{,}100$ satisfies $\gnu(n)\approx 1.64n$; we calculate $\gnu(n)=4726416682534350$ in about $136$ seconds.}
\end{example}

The practical observation  that $\gnu(n)$ is often small is supported by the following result; it shows  that a stronger form of \cite[Conjecture~21.16]{blackburn} holds for almost all cubefree~$n$. This proves  Theorem B.\enlargethispage{0.6cm}
  
\begin{proposition}\label{prop_almostall}
For almost all cubefree $n$ (that is, for all $n$ in a set of integers that is dense in the set of all cubefree integers), we have
\[\gnu(n) \leq (\log n)^{(\log\log n)^{O(1)}};\] 
in particular, $\gnu(n)=n^{o(1)}$ for almost all cubefree $n$.
\end{proposition}  
\begin{proof} 
Let $\Upsilon\subseteq\mathbb{N}$ be the set defined in \cite[Definition II.3]{focs}, which is dense in $\mathbb{N}$ by \cite[Theorem II.4]{focs}. Since $\mathbb{N}\setminus\Upsilon$ has density $0$, while the cubefree integers have density $\approx 83\%$  by \cite[(2)]{density}, almost every cubefree integer lies in $\Upsilon$; it therefore suffices to prove the asserted bound for cubefree $n\in\Upsilon$.

So let $n\in\Upsilon$ be cubefree. It follows from  \cite[Theorem II.5]{focs} that every group $G$ of order $n$ can be decomposed as $G=H\ltimes B$, where $B$ is a cyclic Hall subgroup of $G$ whose order $b$ is squarefree with at most $2\log \log n$ prime divisors, and where the order $h$ of $H$ satisfies $h\leq(\log n)^{O((\log\log n)^{c})}$ for some constant $c$; in particular, $b$ and $h$ only depend on $n$, not on $G$. Since $h$ and $b$ are coprime, \cite[Lemma II.1]{focs} shows that the isomorphism type of $G$ is determined by that of $H$ together with the $(\Aut(H)\times\Aut(B))$-orbit of the associated homomorphism $H\to\Aut(B)$. Thus,
\[\gnu(n)\ \leq\ \gnu(h)\cdot\max\nolimits_{H}\big|\Hom(H,\Aut(B))\big|,\]
where $H$ runs over the groups of order $h$. As $b$ is squarefree, $\Aut(B)\cong\prod_{p\mid b}C_{p-1}$ is abelian, so every homomorphism $H\to\Aut(B)$ factors through $H/H'$, and therefore
\[\big|\Hom(H,\Aut(B))\big|=\prod\nolimits_{p\mid b}\big|\Hom(H/H',C_{p-1})\big|.\]
If $H/H'\cong C_{d_1}\times\ldots\times C_{d_r}$, then  $|\Hom(H/H',C_{p-1})|\leq\prod_{i=1}^rd_i\leq h$ for every $p\mid b$. Since $b$ has at most $2\log \log n$ prime divisors, we have
\[\big|\Hom(H,\Aut(B))\big|\leq h^{2\log \log n}.\] Proposition \ref{prop_upper} implies that $\gnu(h)\leq h^K$ for an absolute constant $K$, and we have $K\leq 8$ by  \cite[p.~233]{blackburn}. Thus, together, we obtain \[\gnu(n)\leq h^{K+2\log\log n}\leq(\log n)^{(\log\log n)^{O(1)}},\]and therefore $\gnu(n)=n^{o(1)}$, cf.\ \cite[p.\ 459]{focs}.
\end{proof}

\section{Computational evaluations}\label{secComp}

In Section~\ref{secComb} we first summarise why our formulas are combinatorial in the sense of Definition~\ref{def_CF}; Sections~\ref{secImpl} and~\ref{secXC} then describe the implementation and the cross-checks we have carried out. 
\enlargethispage{0.6cm}

\subsection{Combinatorial formulas}\label{secComb}
For easy reference, we summarise how the group-theoretic ingredients in our formulas can be dealt with solely by arithmetic calculations and table look-ups.  Remarks~\ref{rem_savegroup} and~\ref{rem_savegroup2} explain that running with $\mathcal{U}$ over the set of projection tuples and with $\theta$ over the group $\Theta(\mathcal{U})$ reduces to working with lists of integers; no group is ever constructed.   Lemmas~\ref{lemSizes} and~\ref{lemFix} determine the cardinalities of $\mathcal{K}_q(\mathcal{U},L)$ and $\mathcal{K}_q(\mathcal{U},L)^\theta$ via  arithmetic formulas; it follows that $\gnu(n)$ is combinatorial for odd $n$. The result for the even case now follows from Lemmas~\ref{lemAtabS}, \ref{lemFixHS}, \ref{lemHcountS}, and \ref{lemColS}, which consider the numbers $A(U;\xi,\psi)$,  $A(\mathcal{U};E,\xi,\alpha,\psi)$, and $|\mathcal{K}(\mathcal{U},L)^H|$. How the data of Remark \ref{rem_savegroup}c) is organised is the crux of an efficient evaluation.

   \begin{remark}\label{rem_storage}
     In practice, the data necessary to characterise each $U\in\mathcal{U}(p,e,L)$ is stored as one record, listing the column type of $U$, a flag recording whether $\Theta(U)$ is trivial or of order $2$, and, for each prime $q$ dividing $|L|$, the integers $j$, $d$, and $d^\pm$  that characterise $U_q$. The flag is stored per column, whereas the induced action is stored per Sylow subgroup via $d^\pm$. Since the column type is an attribute of $U$, it cannot vary with $q$; e.g., no single $U$ has type (5) for one prime and type (4) for another.  The group $\Theta(\mathcal{U})$ of \eqref{def_ThetaU} is the elementary abelian $2$-group on the flagged columns, and each $\theta\in\Theta(\mathcal{U})$ is a $0/1$-vector over these columns; this vector determines, in the $i$-th column, whether the eigenspace ranks entering Lemma~\ref{lemFix} are $(d_i,0)$ or the stored pair $(d_i^+,d_i^-)$. The stored $d^\pm$ require no group calculations: they are $(d,0)$ if $\Theta(U)=1$, and otherwise are determined by a divisibility test or by comparing the canonical generators of Remark \ref{rem_savegroup}a) with their images under $J$.
\end{remark}

In view of Remarks  \ref{rem_redsq} and \ref{rem_closed}, it is not expected that there is a closed formula for $\gnu(n)$ of the shape of  \eqref{eq_holder}; our combinatorial formula requires some calculations to evaluate $\gnu(n)$. This explains why the runtime for evaluating $\gnu(n)$ is expected to grow  with the number of prime divisors of $n$, in particular with the number of those $p,q$  with $q,p^2\mid n$ and $q\mid p^2-1$. For odd $n$, we can utilise \eqref{eq_splitgnu}, but this does not apply to even $n$; \emph{even} worse, evaluating the general formula for even $n$ (in particular, the case $t=2$ in Theorem \ref{thm_maineven}) is more involved than in the odd case.

\subsection{Implementation}\label{secImpl}
The formulas of Theorems \ref{thm_mainodd} and \ref{thm_maineven} (and of Remark \ref{rem_subclasses}) have been implemented in GAP \cite{gap}, see the footnote on page \pageref{pagelink}, and will be included in the next release of the GAP package Cubefree \cite{cubefreeGAP}. Here we comment on some practical improvements.

Evaluating our formula for $\gnu(n)$  naively leads to  considerable redundant work. For example, in Theorem \ref{thm_mainodd}, most of the data occurring in the sums does not depend on all three of $d$, $\ell$, and $L$. Our implementation  reduces this redundancy as follows. First, the set $\mathcal{U}(p,e,L)$ of \eqref{eq_defUpel} depends only on $(p,e,L)$, so it can be computed once and  reused. Second, by Lemmas \ref{lemSizes} and \ref{lemFix}, a projection $U_i$ plays a role in $|\mathcal{K}_q(\mathcal{U},L)^\theta|$ only via its $q$-part and the corresponding integers $j_i$, $d_i$, and $d_i^\pm$. The same holds for the generator $\theta_{U_i}$ of the group $\Theta(U_i)$. Distinct elements of $\mathcal{U}(p,e,L)$ frequently share these integers, so it suffices to store the distinct integer tuples with a multiplicity. Moreover, most triples $(d,\ell,L)$ contribute $0$ and  often this is  because $L_q$ cannot be embedded into $(U_1)_q\times\ldots\times (U_m)_q$ due to order or rank obstructions; our implementation tries to detect this before any tuples are formed.

\revnew{For even $n$, evaluating the inner sums over $\xi$ and $\psi$ (and $\alpha$) in Propositions~\ref{prop_t1S} and \ref{prop_t2S} is expensive because their number of terms is exponential in the number of columns with $\Theta(U_i)\neq 1$, even though most  terms are $0$. We illustrate for $t=2$ how to improve this; the case $t=1$ is similar. Every term has a factor $|\mathcal{K}(\mathcal{U},L)^H|=\prod_{q\mid\nu_0}|\mathcal{K}_q(\mathcal{U},L)^H|$ where $H=\langle\xi,\psi(E)\rangle$. Lemma \ref{lemFixHS} evaluates each $|\mathcal{K}_q(\mathcal{U},L)^H|$ by inspecting a certain set $\mathcal{Y}$ of characters $\chi\colon H\to\{1,-1\}$, and proves that $|\mathcal{K}(\mathcal{U},L)^H|=0$ if $|\mathcal{Y}|>2$ for some $q\mid\nu_0$. The construction of $\mathcal{Y}$ is column-wise, and once $|\mathcal{Y}|>2$ during that construction, one can  abort it. Our implementation exploits this and enumerates the pairs $(\xi,\psi)$  column by column by backtracking.  For each computed pair $(\xi,\psi)$ it then remains to run over $\alpha\in\Aut(E)$ with $\psi\circ\alpha=\psi$, see Lemma \ref{lemHcountS}. Since the value $|\mathcal{K}(\mathcal{U},L)^H|$ is independent of $\alpha$, it can be computed and stored once per pair $(\xi,\psi)$, and not for each $\alpha$.} 

All the optimisations described above were implemented with the assistance of the large language model Claude Opus 5, which has also streamlined and significantly improved the final code. The implementation has been checked extensively against independent data, see Section \ref{secXC}.

 \enlargethispage{0.3cm}

\subsection{Cross-checking}\label{secXC}
Theorems \ref{thm_mainodd} and \ref{thm_maineven} are proved mathematically: no part of their proofs relies on a computation. Nevertheless, we have cross-checked our implementation against existing data. These tests serve two purposes. First, they validate our implementation, which uses optimisations (Section~\ref{secImpl}) that are not part of the formulas themselves. Second, they provide independent evidence for the correctness of our proofs: the values we compare against are obtained by group calculations, so a mathematical error in one of our case distinctions or in the technical proofs in Appendix \ref{secKH} would be expected to lead to a disagreement. We ran the following tests; all runtimes were obtained on a 2026 Apple M5. 
\begin{iprf}
\item[$\bullet$] The SmallGroups Library  \cite{smallgroups} stores the cubefree groups of order at most $50{,}000$. The values of our formula implementation agree with these numbers of groups, which range from $1$ to $3{,}093$; the maximum is attained for $n=2^2.3^2.5^2.7^2=44{,}100$. It took our implementation $2$ seconds to determine the numbers of all groups of cubefree order $n\leq 50{,}000$. 
\item[$\bullet$] The SmallGroups Library provides H\"older's formula for squarefree orders. The values of our formula implementation agree with these numbers for all squarefree $n\leq 10^8$; see also Remark~\ref{rem_redsq}. The paper \cite{sot} classifies all groups whose orders factorise into at most four primes; for every cubefree $n\leq 10^8$ of that form, our values agree with those obtained from the implementation of \cite{sot}.  The paper \cite{cgroups} provides group identification and construction for C-groups; we verified that our formula counts the same number of C-groups for every cubefree $n\leq 10^8$. Since every group of squarefree order is a $C$-group, this test reconfirms our squarefree test against H\"older's formula.
\item[$\bullet$] The   package Cubefree  \cite{cubefreeGAP} provides a function {\small {\tt NumberCFGroups}} that counts the number of groups of a given cubefree order. This function has only limited applicability since it relies on computations in groups and not on a formula. For example, it took our code  $2.5$ seconds to enumerate the $265{,}681{,}924$ groups of order $3^2.5^2.11^2.29^2.41^2.59^2.101^2$, and we stopped {\small{\tt NumberCFGroups}} after one hour without a result.  We ran {\small{\tt NumberCFGroups}} for a few days on multiple cores to enumerate all groups of cubefree order $n\leq 10^6$, and the values of our implemented formula always agreed.   
\item[$\bullet$] Since the above tests  are not guaranteed to reach every `branch' of our formulas, we have also checked our implementation against {\small\tt NumberCFGroups} for some specifically designed larger orders whose prime factors satisfy many divisibility relations  $q\mid p\pm 1$. Some of these orders are listed in the upper part of Table \ref{tabTests}; in each case the two values agreed. It took {\small\tt NumberCFGroups} about $17$ hours to enumerate the groups for $n=821{,}760{,}380{,}100$; our code needed $0.8$ seconds. We were not able to run {\small\tt NumberCFGroups} to completion for the last five orders in Table \ref{tabTests}; we list them here to indicate some runtimes for calculating $\gnu(n)$. For the last (odd) number, $\gnu(n)$ can be computed readily  due to \eqref{eq_splitgnu}; the graph $\Gamma(n)$ splits into four components. Without applying \eqref{eq_splitgnu}, the calculation of $\gnu(n)$ was aborted after $5$ minutes.  
\end{iprf}

\begin{table}
\footnotesize  
\renewcommand{\arraystretch}{1.1}  
{\[\begin{array}{@{}rcr|rr@{}}
\multicolumn{3}{c|}{\hspace*{2cm}\text{factorised order $n$}} & \gnu(n) & \text{time (sec)}\\\hline\hline
2^2.3^2.5^2.7^2.11^2.13^2 &=& 901800900& 8228379& 0.5\\
2^2.3^2.7.13^2.19.37^2&=&1107756468 & 6508264& 0.1\\
2.3^2.5^2.7^2.17^2.19^2 & = & 2300454450 & 298863 & 0.1 \\
2^2.3^2.17^2.19^2.37^2&=&5141750436& 1122161 & 0.1\\
2^2.3^2.7^2.11^2.197^2 &=& 8283548196 &  133543 & 0.1\\
2^2.3^2.5^2.101^2.199^2 &=& 363572820900 & 338109 & 0.1\\
2^2.3^2.5^2.11^2.41^2.67^2 &=&821760380100 & 22855136 & 0.8\\
2^2.3^2.7^2.11^2.43^2.131^2 & = & 6772725182916 & 8434136 & 0.8\\
2^2.5^2.11^2.109^2.331^2 &=& 15750500316100 & 5816325 & 0.1\\
3^2.5^2.7^2.199^2.251^2 &=&  27506301176025 & 32050 & 0.1\\\hline
 3^2.5^2.7^2.11^2.29^2.31.37.41&=&52760297880675 & 12333191 & 0.6\\ 
 2.3^2.7^2.11.13^2.17^2.19^2.23^2.29.31 &=&81352027912197042 & 5212072308 & 4.6\\
 2^2.3^2.7^2.13^2.2731^2.2729^2  & =&16559062945858385316 & 2114627120 & 3.5\\
2.3.5^2.11^2.23.43^2.53^2.79^2.83.103^2.113 &=&1346412960864517896277950 &137118184 & 3.5\\
5^2.7^2.13^2.17^2.19^2.23^2.31^2.37^2.43^2.47^2.53^2.59^2.67^2 &=& \approx 2.7\cdot 10^{36} & 92432340 & 0.1 
\end{array}\]}
\caption{Some larger cubefree orders for which $\gnu(n)$ was calculated; for the orders above the line, the values of $\gnu(n)$ were cross-checked against \cite{cubefreeGAP}.}\label{tabTests}
\end{table}

\section{Applications: group construction and identification}\label{sec_app}
Computer algebra systems such as GAP \cite{gap} and Magma \cite{magma} contain the SmallGroups Library, which stores isomorphism type representatives for the groups of certain ``small'' orders, together with an {\emph{identification function}}: given a group $G$ of such an order $n$, one can determine its ID $[n,i]$, meaning that $G$ is isomorphic to the $i$-th group in the stored list for order $n$. Since the ID is an isomorphism invariant, reducing a list of $m$ groups up to isomorphism amounts to comparing IDs, rather than employing $m(m-1)/2$ isomorphism tests; we refer to Besche-Eick-O'Brien \cite{smallgroups,mill} for more details.

Extending this library is an important topic in computational group theory: recent contributions are Dietrich-Eick-Pan \cite{sot} for orders that factorise into at most four primes, Dietrich-Low \cite{cgroups} for C-groups, and Eick-Schanze \cite{29} and Dietrich-Eick-Schanze \cite{p5} for the orders $2^9$ and $p^5$. With the exception of order $2^9$, these works provide dynamic databases that describe infinitely many groups: for example, \cite{sot} allows one to construct, enumerate, and identify all the groups of order $p^4$, $p^3q$, $p^2q^2$, $p^2qr$, $pqrs$, for arbitrary distinct primes $p,q,r,s$. One reason why this can be done efficiently is that the groups of a specific order are partitioned into \emph{clusters} according to some structure invariants, with counting formulas for each cluster: for example, the groups of order $p^3q$ form $12$ clusters, see \cite[Table~2]{sot}, and constructing the group with ID $[p^3q,j]$ directly (\emph{construction-by-ID}) reduces to constructing a group in a uniquely defined cluster that is determined by $j$ and the counting functions. All the  routines developed in \cite{sot,29,p5,cgroups} rely on variations of this strategy, which illustrates the importance of suitable counting formulas.

As discussed in Remark~\ref{rem_subclasses}, each of the first sums in Theorems~\ref{thm_mainodd} and \ref{thm_maineven} is indexed by a structural invariant, so that truncating the summation at any of these levels returns the number of cubefree groups of order~$n$ with the corresponding invariants. This is exactly the information required to partition the groups  of a given order into clusters for efficient group construction and identification. Such a clustering of groups is particularly relevant for orders that admit \emph{many} isomorphism types, see Table \ref{tabTests}, since it allows one to identify or construct a specific group without considering all groups of that order. The counting formulas presented in this paper provide the basis for the development of identification and construction-by-ID for cubefree groups; this is beyond our present scope and will be investigated in future work.

\appendix
\section{Technical arguments for the even case}\label{secKH} 

The aim of this appendix is to show that the numbers  $|\mathcal{K}(\mathcal{U},L)^H|$ and $A(\mathcal{U};E,\xi,\alpha,\psi)$ required for the even case in Theorem~\ref{thm_maineven} and Proposition \ref{prop_t2S} can be evaluated by a combinatorial calculation and table look-ups; this is used in our implementation. The technical results are presented here to not disrupt the narrative in the main text.

\subsection*{The cardinalities $|\mathcal{K}(\mathcal{U},L)^H|$}
We first consider a generalisation of Lemma \ref{lemFix} and show how the size $|\mathcal{K}(\mathcal{U},L)^H|$ can be evaluated combinatorially. Since the bijection \eqref{eqSplit} is $\Theta(\mathcal{U})$-equivariant, it restricts to a bijection \begin{align*}\mathcal{K}(\mathcal{U},L)^H\to\prod\nolimits_{q\mid\nu_0}\mathcal{K}_q(\mathcal{U},L)^H,
\end{align*}so that $|\mathcal{K}(\mathcal{U},L)^H|=\prod_{q\mid\nu_0}|\mathcal{K}_q(\mathcal{U},L)^H|$. The next  lemma shows how each $|\mathcal{K}_q(\mathcal{U},L)^H|$ can be determined. We use the notation of Section \ref{secStep1}, that is, our tuple is denoted $\mathcal{U}=(U_1,\ldots,U_m)$, $q$ is a prime divisor of $\nu_0$, and we write $W=X_1\times\ldots\times X_m$ where each $X_i=(U_i)_q$.

\begin{lemma}\label{lemFixHS} 
We use the above notation and for a homomorphism $\chi\colon H\to\{1,-1\}$ define the $\chi$-eigenspace as
\[W_\chi=\{w\in W : \theta(w)=w^{\chi(\theta)} \text{ for all }\theta\in H\}.\]
If $\mathcal{Y}$ is the set of those $\chi$ with $W_\chi\ne1$, then $W=\prod_{\chi\in\mathcal{Y}}W_\chi$.  For $i\in\{1,\ldots,m\}$ and any homomorphism $\chi$ define  $d_i^{(\chi)}=\rk(X_i\cap W_\chi)$ and
\[\Lambda^{(\chi)}=\tfrac{1}{q-1}\left(\prod\nolimits_{i=1}^mc(d_i^{(\chi)})-\prod\nolimits_{i=1}^mz(d_i^{(\chi)})\right).\]
Then 
\begin{align*}\big|\mathcal{K}_q(\mathcal{U},L)^H\big|=\begin{cases}
\big|\mathcal{K}_q(\mathcal{U},L)\big|&\text{if }|\mathcal{Y}|\leq1,\\
\Delta_{L_q\cong C_q^2}\ \Lambda^{(\chi)}\Lambda^{(\chi')}&\text{if }\mathcal{Y}=\{\chi,\chi'\}\text{ with }\chi\ne\chi',\\
0&\text{if }|\mathcal{Y}|\geq3.
\end{cases}
\end{align*}
\end{lemma}
\begin{proof}
  Since $|W|$ is odd and $\theta^2=1$ for every $\theta\in H$, applying Lemma \ref{lemEig} iteratively to the generators of  $H$ decomposes $W$ into the nontrivial subgroups on which each of these generators acts as a scalar; these subgroups are exactly the nontrivial $W_\chi$, so $W=\prod_{\chi\in\mathcal{Y}}W_\chi$, as claimed.  (This is possible since $H$ is abelian, so each generator preserves the eigenspaces of the others.)

  As $H$ acts coordinate-wise, we also have $W_\chi=\prod_{i=1}^m(X_i\cap W_\chi)$ and $X_i=\prod_{\chi\in\mathcal{Y}}(X_i\cap W_\chi)$ for every~$i$. Applying the same decomposition to an $H$-invariant  $O\leq W$ yields $O=\prod_{\chi\in\mathcal{Y}}(O\cap W_\chi)$. Conversely, each $\theta\in H$ acts on $W_\chi$ as $w\mapsto w^{\chi(\theta)}$, so every subgroup of $W_\chi$ is $H$-invariant; in particular, $\prod_{\chi\in\mathcal{Y}}A_\chi$ is $H$-invariant for any  $A_\chi\leq W_\chi$. If $O\in\mathcal{K}_q(\mathcal{U},L)^H$, then $\pi_i(O)=X_i$ forces  $\pi_i(O\cap W_\chi)=X_i\cap W_\chi$ for all $i$ and $\chi$. 

  If $\mathcal{Y}=\emptyset$, then $W=1$ and $\mathcal{K}_q(\mathcal{U},L)^H=\mathcal{K}_q(\mathcal{U},L)$ are both empty; if $\mathcal{Y}=\{\chi\}$, then $W=W_\chi$ and every $O\in\mathcal{K}_q(\mathcal{U},L)$ is $H$-invariant. Thus, the first case holds.

  Now let $|\mathcal{Y}|\geq2$ and $O\in\mathcal{K}_q(\mathcal{U},L)^H$.  For every $\chi\in\mathcal{Y}$ we have  $W_\chi\ne 1$, so there is $i$ with $X_i\cap W_\chi\ne1$, and so $\pi_i(O\cap W_\chi)=X_i\cap W_\chi\ne1$. This shows that  $O\cap W_\chi\ne1$ for every $\chi\in\mathcal{Y}$.  Since $O=\prod_{\chi\in\mathcal{Y}}(O\cap W_\chi)$ and $O\cong L_q$, we have $\rk(L_q)=\rk(O)\geq|\mathcal{Y}|\geq2$, and since $|L_q|$ divides $q^2$, this forces $L_q\cong C_q^2$ and $|\mathcal{Y}|=2$, say $\mathcal{Y}=\{\chi,\chi'\}$. (In particular, $\mathcal{K}_q(\mathcal{U},L)^H=\emptyset$ if $|\mathcal{Y}|\geq3$ or $L_q$ is cyclic.)   Then $O=A\times A'$ with $A=O\cap W_\chi$ and $A'=O\cap W_{\chi'}$ of order $q$. Thus the groups in $\mathcal{K}_q(\mathcal{U},L)^H$  are exactly $A\times A'$  where $A$ and $A'$ run  through the subgroups of order $q$ of $W_\chi$ and $W_{\chi'}$ that project onto $X_i\cap W_\chi$ and $X_i\cap W_{\chi'}$ for every $i$. As in step (1) of the proof of Lemma \ref{lemFix} the number of these is $\Lambda^{(\chi)}$ and $\Lambda^{(\chi')}$, respectively.
\end{proof} 
 
We explain how $\mathcal{Y}$ and $d_i^{(\chi)}$ can be read off the integers of Remark~\ref{rem_storage} and the $0/1$-vectors representing $H$, with no group calculation.  For each $i\in\{1,\ldots,m\}$, let $\chi_i\colon H\to\{1,-1\}$ be the homomorphism with $\chi_i(\theta)=1$ if and only if $\theta_i=1$; note that the $0/1$-vector describing $\theta$ is $(s_1,\ldots,s_m)$ where each $\chi_i(\theta)=(-1)^{s_i}$. Let $X_i=X_i^+\times X_i^-$ be the eigenspace decomposition of $\theta_{U_i}$ with $d_i^\pm=\rk(X_i^\pm)$ (where $\theta_{U_i}={\rm id}$ if $\Theta(U_i)=1$). Let $\chi_0\colon H\to\{1,-1\}$ be the trivial character. We claim that $X_i\cap W_\chi\neq1$ forces $\chi\in\{\chi_0,\chi_i\}$: if $x\in X_i\cap W_\chi$ is nontrivial and $\theta\in\ker\chi_i$, then $x=\theta(x)=x^{\chi(\theta)}$, so $\chi(\theta)=1$  and $\ker\chi_i\leq\ker\chi$. Since $|H/\ker\chi_i|\leq2$ and a homomorphism into $\{1,-1\}$ is determined by its kernel, $\chi\in\{\chi_0,\chi_i\}$, as claimed. Since $W_\chi=\prod_{i=1}^m(X_i\cap W_\chi)$, this shows that $\chi\in\mathcal{Y}$ implies $\chi\in\{\chi_0,\chi_i\}$ for some $i$; moreover, $\chi\in\mathcal{Y}$ if and only if $d_i^{(\chi)}\neq0$ for some $i$, and $d_i^{(\chi)}=0$ whenever $\chi\notin\{\chi_0,\chi_i\}$. If $\chi_i=\chi_0$, then $H$ acts trivially on $X_i$, so $X_i\leq W_{\chi_0}$ and $d_i^{(\chi_0)}=d_i$. If $\chi_i\neq\chi_0$, then $X_i\cap W_{\chi_0}=X_i^+$ and $X_i\cap W_{\chi_i}=X_i^-$;  thus, $d_i^{(\chi_0)}=d_i^+$ and $d_i^{(\chi_i)}=d_i^-$.

\subsection*{The numbers $A(\mathcal{U};E,\xi,\alpha,\psi)$}
It remains to evaluate the numbers $A(\mathcal{U};E,\xi,\alpha,\psi)$, see \eqref{eq_newA}. We continue with the notation introduced prior to Proposition \ref{prop_t2S}, so we consider homomorphisms $\phi=(\phi_1,\ldots,\phi_m)$ where each $\phi_i\in \Hom(E,N_i)/Z_i$, such that $\bigcap_{i=1}^m \ker\phi_i=1$. Without this trivial intersection condition in the definition of $\mathcal{H}(E)$, the number $A(\mathcal{U};E,\xi,\alpha,\psi)$ would simply factor over the columns, where the $i$-th factor, denoted  $T(U_i;E,\xi_i,\alpha,\psi_i)$, is the number of $Z_i$-classes of homomorphisms $\phi_i\colon E\to N_i$ with $\epsilon_{U_i}\circ\phi_i=\psi_i$ and $\phi_i^{\xi_i}\circ\alpha=\phi_i$; recall that the last equation is an equation of $Z_i$-classes, see Remark \ref{rem_classes}. Motivated by this, for a group $P$ of order dividing $4$ and $U\in\mathcal{U}(p,e)$ with $C=C_{\GL_e(p)}(U)$ and $N=N_{\GL_e(p)}(U)$ we define
\begin{equation}\label{eq_defT}T(U;P,\xi,\alpha,\psi)=\big|\{\phi\in\Hom(P,N)/C : \epsilon_U\circ\phi=\psi\ \text{ and }\ \phi^\xi\circ\alpha=\phi\}\big|,\end{equation}
where $\xi\in\Theta(U)$, $\alpha\in\Aut(P)$, and $\psi\in\Hom(P,\Theta(U))$. We first reduce $A(\mathcal{U};E,\xi,\alpha,\psi)$ to these numbers by an inclusion-exclusion over the subgroup lattice of $E$, where each subgroup represents a kernel $\bigcap_{i=1}^m\ker\phi_i$.
 
\begin{lemma}\label{lemHcountS}
We continue with the previous notation, and consider  $\xi\in\Theta(\mathcal{U})$, $\alpha\in\Aut(E)$, and $\psi\in\Hom(E,\Theta(\mathcal{U}))$. If $\psi\circ\alpha\ne\psi$, then $A(\mathcal{U};E,\xi,\alpha,\psi)=0$. If $\psi\circ\alpha=\psi$, then
\[A(\mathcal{U};E,\xi,\alpha,\psi)=\sum\nolimits_{R\leq E}\mu(R)\ \Delta_{\langle R\rangle_\alpha\leq\ker\psi}\ \prod\nolimits_{i=1}^mT\big(U_i;E/\langle R\rangle_\alpha,\xi_i,\bar\alpha,\bar\psi_i\big),\]
with the following notation: $\langle R\rangle_\alpha$ is the smallest $\alpha$-invariant subgroup of $E$ containing $R$, the maps $\bar\alpha\in\Aut(E/\langle R\rangle_\alpha)$ and $\bar\psi_i\in\Hom(E/\langle R\rangle_\alpha,\Theta(U_i))$ are induced by $\alpha$ and $\psi_i$ (where $\bar\psi_i$ is only defined if $\Delta_{\langle R\rangle_\alpha\leq\ker\psi}=1$), and $\mu$ is the M\"obius function of the subgroup lattice of $E$, that is, $\mu$ is determined by $\sum_{R\leq \tilde R}\mu(R)=\Delta_{\tilde R=1}$, and  explicitly, $\mu(1)=1$, $\mu(R)=-1$ for $|R|=2$, $\mu(C_4)=0$, and $\mu(C_2^2)=2$.
\end{lemma}
\begin{proof}
Since $\Theta(U_i)$ is abelian, we have $\epsilon_{U_i}\circ\phi_i^{\xi_i}=\epsilon_{U_i}\circ\phi_i$, so $\phi_i^{\xi_i}\circ\alpha=\phi_i$ forces $\psi_i\circ\alpha=\psi_i$; thus $A(\mathcal{U};E,\xi,\alpha,\psi)=0$ if $\psi\circ\alpha\ne\psi$. Now let $\psi\circ\alpha=\psi$, and for $R\leq E$ let $\mathcal{T}(R)$ be the number of tuples $\phi=(\phi_1,\ldots,\phi_m)$ of $Z_i$-classes of homomorphisms $\phi_i\colon E\to N_i$ with $R\leq\ker\phi_i$ and $\epsilon_{U_i}\circ\phi_i=\psi_i$ and $\phi_i^{\xi_i}\circ\alpha=\phi_i$ for all $i$. Thus $\mathcal{T}(1)$ counts all tuples $\phi$ satisfying the last two conditions, with no constraint on the kernels, and $\mathcal{T}(R)$ counts those with  $R\leq\bigcap_{i=1}^m\ker\phi_i$. For $\tilde R\leq E$ let $\mathcal{T}_{\ker}(\tilde R)$ be the number of tuples counted by $\mathcal{T}(1)$ with $\bigcap_{i=1}^m\ker\phi_i=\tilde R$. Sorting the tuples counted by $\mathcal{T}(R)$ according to the value of $\bigcap_{i=1}^m\ker\phi_i$ therefore gives 
\[(\ast)\quad \mathcal{T}(R)=\sum\nolimits_{R\leq \tilde R\leq E}\mathcal{T}_{\ker}(\tilde R)\quad\text{for every }R\leq E.\]
We have $A(\mathcal{U};E,\xi,\alpha,\psi)=\mathcal{T}_{\ker}(1)$ by definition, see \eqref{eq_newA}. Since $\sum_{R\leq \tilde R}\mu(R)=\Delta_{\tilde R=1}$ by definition, we can rewrite
\[\mathcal{T}_{\ker}(1) = \sum\nolimits_{\tilde R\leq E}\mathcal{T}_{\ker}(\tilde R)\sum\nolimits_{R\leq \tilde R}\mu(R)=\sum\nolimits_{R\leq E}\ \sum\nolimits_{R\leq\tilde R\leq E}\mu(R)\mathcal{T}_{\ker}(\tilde R),\]
and now $(\ast)$ yields 
\[A(\mathcal{U};E,\xi,\alpha,\psi)=\mathcal{T}_{\ker}(1)=\sum\nolimits_{R\leq E}\mu(R)\mathcal{T}(R).\]
Note that $\mu$ is the M\"obius function  $\bar\mu_1$ of \cite[(2.41), (2.42)]{hall} for the group $E$. If $\phi_i^{\xi_i}\circ\alpha=\phi_i$, then $\ker\phi_i$ is $\alpha$-invariant; recall  that $Z_i$-conjugate homomorphisms have the same kernel, so  this statement is meaningful for a $Z_i$-class $\phi_i$. Thus, $R\leq\ker\phi_i$ if and only if $\langle R\rangle_\alpha\leq\ker\phi_i$. This shows that $\mathcal{T}(R)=\mathcal{T}(\langle R\rangle_\alpha)$. The lemma follows once we have proved that  
\[(\dagger)\quad \mathcal{T}(\langle R\rangle_\alpha)=\Delta_{\langle R\rangle_\alpha\leq\ker\psi}\prod\nolimits_{i=1}^mT(U_i;E/\langle R\rangle_\alpha,\xi_i,\bar\alpha,\bar\psi_i),\]
where $R\leq E$ is a subgroup, and $\psi\in\Hom(E,\Theta(\mathcal{U}))$ is the homomorphism fixed in the lemma with components $\psi=(\psi_1,\ldots,\psi_m)$. Write $P=E/\langle R\rangle_\alpha$ and let $\pi\colon E\to P$ be the natural epimorphism; since $\langle R\rangle_\alpha$ is $\alpha$-invariant, the induced $\bar\alpha\in\Aut(P)$ is well defined and  $\bar\alpha\circ\pi=\pi\circ\alpha$. Recall that if $G$ is any group, then  $\bar\phi\mapsto\bar\phi\circ\pi$ defines a  bijection between $\Hom(P,G)$ and those homomorphisms in $\Hom(E,G)$ that have $\langle R\rangle_\alpha$ in their kernel. We make a case distinction on whether $\langle R\rangle_\alpha\leq\ker\psi$.

First, let $\langle R\rangle_\alpha\not\leq\ker\psi$, that is, $\langle R\rangle_\alpha\not\leq\ker\psi_j$ for some $j$. Every tuple counted by $\mathcal{T}(\langle R\rangle_\alpha)$ satisfies $\langle R\rangle_\alpha\leq\ker\phi_j\leq\ker(\epsilon_{U_j}\circ\phi_j)=\ker\psi_j$, so there is no such tuple, hence $\mathcal{T}(\langle R\rangle_\alpha)=0$. As $\Delta_{\langle R\rangle_\alpha\leq\ker\psi}=0$, both sides of $(\dagger)$ are $0$.
 
Now let $\langle R\rangle_\alpha\leq\ker\psi$, so each $\psi_i=\bar\psi_i\circ\pi$ for a unique $\bar\psi_i\in\Hom(P,\Theta(U_i))$. We consider a fixed~$i$. Let $\bar\phi\in\Hom(P,N_i)$ and define  $\phi=\bar\phi\circ\pi$. Since $(\bar\phi\circ\pi)^c=\bar\phi^{\,c}\circ\pi$ for every $c\in Z_i$, the above bijection is compatible with the $Z_i$-actions, hence it induces a bijection between $\Hom(P,N_i)/Z_i$ and the $Z_i$-classes of homomorphisms $E\to N_i$ whose kernel contains $\langle R\rangle_\alpha$. Note that
\[\epsilon_{U_i}\circ\phi=(\epsilon_{U_i}\circ\bar\phi)\circ\pi\qquad\text{and}\qquad \phi^{\xi_i}\circ\alpha=(\bar\phi^{\,\xi_i}\circ\bar\alpha)\circ\pi,\]
where the right equation follows from $\pi\circ\alpha=\bar\alpha\circ\pi$ and the fact that $\phi^{\xi_i}$ and $\bar\phi^{\,\xi_i}$ are obtained by composing $\phi$ and $\bar\phi$ with conjugation by a preimage $n\in N_i$ of $\xi_i$. Since $\psi_i=\bar\psi_i\circ\pi$ and composition with $\pi$ is injective and compatible with the $Z_i$-actions, these two equations show that the $Z_i$-class of $\bar\phi$ satisfies $\epsilon_{U_i}\circ\bar\phi=\bar\psi_i$ and $\bar\phi^{\,\xi_i}\circ\bar\alpha=\bar\phi$ if and only if the $Z_i$-class of $\phi$ satisfies $\epsilon_{U_i}\circ\phi=\psi_i$ and $\phi^{\xi_i}\circ\alpha=\phi$. Thus, the above bijection identifies the $Z_i$-classes counted by $T(U_i;P,\xi_i,\bar\alpha,\bar\psi_i)$ with the $Z_i$-classes that are possible in the $i$-th component of $\mathcal{T}(\langle R\rangle_\alpha)$; hence there are $T(U_i;P,\xi_i,\bar\alpha,\bar\psi_i)$ of the latter. Since the conditions defining $\mathcal{T}(\langle R\rangle_\alpha)$ are imposed componentwise, $(\dagger)$ holds.
\end{proof}
\enlargethispage{0.5cm}
The next lemma shows that the numbers $T(U;P,\xi,\alpha,\psi)$ in Lemma \ref{lemHcountS} only take  finitely many values and can be determined by a table look-up. This requires the following technical definition; recall that $v_2$ denotes the $2$-adic valuation. By slight abuse of notation, if $\phi\colon A\to B$ is a group homomorphism into an abelian group $B$, then $\phi^{-1}$ denotes the homomorphism $\phi$ composed with $B\to B$, $b\mapsto b^{-1}$.

\begin{definition}\label{def_table2stuff} For an odd prime $p$, a group $P$ of order dividing $4$, $\alpha\in\Aut(P)$, and   $c\geq 1$ define
\begin{align*} \varphi_c(P,\alpha)&=\big|\{\phi\in\Hom(P,C_{2^c}) : \phi\circ\alpha=\phi\}\big|,\\
  \varphi_c^-(P,\alpha)&=\big|\{\phi\in\Hom(P,C_{2^c}) : \phi\circ\alpha=\phi^{-1}\}\big|,
\end{align*}
and let $\rho(P,\alpha)$ be the number of $\GL_2(p)$-classes of homomorphisms $\tau\colon P\to\GL_2(p)$ such that $\tau\circ\alpha$ is $\GL_2(p)$-conjugate to $\tau$. We write $a=v_2(p-1)$ and $\bar a=v_2(p+1)$, and $\delta=\Delta_{4\mid (p-1)}$; note that $a,\bar a\geq1$, that $\delta=1$ if and only if $a\geq2$, and that $a+\bar a=v_2(p^2-1)\geq3$. 
\end{definition}

\begin{lemma}\label{lemTabPhiValues}
  With the notation of Definition \ref{def_table2stuff}, the values of $\varphi_a$, $\varphi_{a+\bar a}$, $\varphi^-_{a+\bar a}$ and $\rho$ are as in Table~\ref{tabPhiS}.
\end{lemma}  
\begin{proof}Throughout we use that  $\Hom(P,C_{2^c})$ has order $1,2,4,\gcd(2^c,4)$ for $P=1,C_2,C_2^2,C_4$, and so $|\Hom(C_4,C_{2^a})|=2+2\delta$ and $|\Hom(C_4,C_{2^{a+\bar a}})|=4$. This already determines the entries in the first three columns for $\alpha={\rm id}$.  We continue with the first three columns, and consider  column $\rho$ last.
  
  If $P\in\{1,C_2,C_2^2\}$, then every $\phi\in\Hom(P,C_{2^c})$ takes values in the unique subgroup of order $2$ of $C_{2^c}$ and satisfies $\phi=\phi^{-1}$; a short calculation confirms the values in the first three columns. The values for $P=C_4$ can also be determined readily. For example, note that  $\varphi_c^-(C_4,{\rm id})$ counts the number of homomorphisms $\phi\colon C_4\to C_{2^c}$ whose image lies in the unique subgroup $C_2\leq C_{2^c}$. 
  
  Now consider column    $\rho$. Note that $\rho(P,\alpha)$ is the number of $\alpha$-invariant isomorphism classes of $2$-dimensional  $\GF(p)$-representations of $P$. As $|P|$ is coprime to $p$, all these representations are semisimple. If we are not in the case $P=C_4$ and $\delta=0$, then $\exp(P)$ divides $p-1$, and therefore every irreducible $P$-representation is a $1$-dimensional character. There are $|P|$ such characters, and  the $2$-dimensional representations of $P$ correspond to the multisets $\{\chi,\chi'\}$ of two characters; such a multiset is $\alpha$-invariant if and only if $\chi\mapsto\chi\circ\alpha$ either fixes both $\chi$ and $\chi'$, or interchanges them. Thus, if this map fixes exactly $f$ characters and interchanges $s$ pairs of characters, then $\rho(P,\alpha)=\tbinom{f}{2}+s+f$. For $\alpha={\rm id}$ we have $f=|P|$ and $s=0$, so $\rho(P,{\rm id})=\binom{|P|+1}{2}$, which is $1$, $3$, $10$ and $10$ for $P=1$, $C_2$, $C_4$ with $\delta=1$, and $C_2^2$, respectively. The other values are determined by the following observations. If $P=C_2^2$, then a transposition has $f=2$ and $s=1$, and a $3$-cycle has $f=1$ and $s=0$. If $P=C_4$ and $\delta=1$, then inversion has $f=2$ and $s=1$. For $P=C_4$ and $\delta=0$, the claimed values follow because $P$ has two characters with values in $\{\pm1\}$ and a unique $\GL_2(p)$-class of irreducible $2$-dimensional representations that maps $P$ to $\Sigma(p,4)$.
\end{proof} 

\begin{table}\footnotesize  
\renewcommand{\arraystretch}{1.2}  
\[\begin{array}{@{}l|l|ccccc@{}}
P&\alpha&\varphi_a&\varphi_{a+\bar a}&\varphi_{a+\bar a}^-&\rho\\\hline\hline
1&{\rm id}&1&1&1&1\\
C_2&{\rm id}&2&2&2&3\\
C_4&{\rm id}&2+2\delta&4&2&4+6\delta\\
C_4&\text{inversion}&2&2&4&4\\
C_2^2&{\rm id}&4&4&4&10\\
C_2^2&\text{transposition}&2&2&2&4\\
C_2^2&\text{$3$-cycle}&1&1&1&1
\end{array}\] 
\caption{The functions $\varphi_c$, $\varphi_c^-$ and $\rho$ of Definition {\ref{def_table2stuff}}.}\label{tabPhiS}
\end{table} 

\enlargethispage{0.4cm}

\begin{lemma}\label{lemColS}We continue with the previous notation; recall the functions $\varphi_c$, $\varphi_c^-$, and $\rho$ and the integers $a=v_2(p-1)$, $\bar a=v_2(p+1)$, and $\delta=\Delta_{4\mid (p-1)}$  for an odd prime $p$. Let $U\in\mathcal{U}(p,e)$, let $P$ be a group of order dividing $4$, and let $\xi\in\Theta(U)$, $\alpha\in\Aut(P)$, and $\psi\in\Hom(P,\Theta(U))$.
\begin{ithm}
\item  If $\psi\ne1$, then $T(U;P,\xi,\alpha,\psi)=0$ unless $U$ has type {\rm(4)} or {\rm(5)} and $\psi\circ\alpha=\psi$, in which case
\[T(U;P,\xi,\alpha,\psi)=\big|\Hom(\ker\psi,C_{2^a})\big|,\]
which equals $1$ if $|P|=2$ and equals $2$ if $|P|=4$.
\item If $\psi=1$, then, depending on the type {\rm (1), \ldots, (5)} of $U$, we have
\[T(U;P,\xi,\alpha,1)=\begin{cases}
\varphi_a(P,\alpha)&{\rm (1)},\\
\rho(P,\alpha)&{\rm (2)},\\
\varphi_a(P,\alpha)^2&{\rm (3)},\\
\varphi_a(P,\alpha)^2\ \text{ if $\xi=1$;\quad} \varphi_a(P,\alpha^2)\ \text{ if $\xi\ne1$}&{\rm (4)},\\
\varphi_{a+\bar a}(P,\alpha)\ \text{ if $\xi=1$ or $\delta=1$;\quad}\varphi^-_{a+\bar a}(P,\alpha)\ \text{ if $\xi\ne1$ and $\delta=0$}&{\rm (5)}
\end{cases}\]
which can be evaluated via Lemma \ref{lemTabPhiValues} and Table \ref{tabPhiS}. 
\end{ithm} 
\end{lemma} 
\begin{proof}
  Throughout, $\phi\colon P\to N$ is a homomorphism and $T(U;P,\xi,\alpha,\psi)$ is the number of $C$-classes of those $\phi$ with $\epsilon_U\circ\phi=\psi$ and $\phi^\xi\circ\alpha=\phi$; recall that the latter is an equation of $C$-classes, see Remark \ref{rem_classes}. Since $C=\ker\epsilon_U$ and $\epsilon_U\circ\phi=\psi$, we have  $\phi(\ker\psi)\leq C$; moreover, $\phi(P)\leq C$ holds if and only if $\epsilon_U\circ\phi$ is trivial, that is, if and only if $\psi=1$. 

\begin{iprf}
\item Let $\psi\ne1$. This requires $|\Theta(U)|=2$, so $U$  has type (4) or (5). Write $N=\langle\sigma\rangle\ltimes C$ with $\sigma\in\{J,\theta\}$. Since $\Theta(U)$ is abelian,  $\epsilon_U\circ\phi^\xi=\epsilon_U\circ\phi$, and the condition $\phi^\xi\circ\alpha=\phi$ forces $\psi\circ\alpha=\psi$. Since $|\Theta(U)|=2$ and $\psi$ is nontrivial, $P_0=\ker\psi$ has index $2$ in $P$, and $\psi\circ\alpha=\psi$ implies  $\alpha(P_0)=P_0$. We now fix $u\in P\setminus P_0$. As mentioned above, $\phi(P_0)\leq C$, and  $\phi(u)$ centralises $\phi(P_0)$ since $P$ is abelian. Since $\phi(u)\in C\sigma$ and $C$ is abelian,  the image of the restriction $\phi_0=\phi|_{P_0}$ lies in  $C^\sigma$, the group of scalar matrices. Writing $\phi(u)=x\sigma$ with $x\in C$, we have $\phi(u^2)=\phi(u)^2=xx^\sigma$ with $u^2\in P_0$.

 We show that $\phi\mapsto\phi_0$ is a bijection from the $C$-classes of  homomorphisms $\phi\colon P\to N$ with $\epsilon_U\circ\phi=\psi$ onto $\Hom(P_0,C^\sigma)$. Note first that $C$-conjugate homomorphisms have the same restriction to $P_0$, since $\phi(P_0)\leq C^\sigma\leq C$ and $C$ is abelian; thus $\phi\mapsto\phi_0$ is well defined on $C$-classes. We claim that $f\colon C\to C^\sigma$, $ s\mapsto ss^\sigma$, is a surjective homomorphism whose kernel is the image of $s\mapsto s^\sigma s^{-1}$. If $\sigma=J$, then $f(\diag(x,y))=\diag(xy,xy)$ and the claim holds. If $\sigma=\theta$, then $\sigma$ acts on $C=\Sigma(p)$ by $s\mapsto s^p$, see Definition \ref{defTypes}(5), so that $f(s)=ss^\sigma=s^{1+p}$. Since $C\cong C_{p^2-1}$, it follows that $f$ is surjective and $\ker f$ is the  image of $s\mapsto s^\sigma s^{-1}=s^{p-1}$. Thus, the claim on $f$ holds. Now let $\sigma\in\{J,\theta\}$ be arbitrary again, and consider a given homomorphism $\beta_0\colon P_0\to C^\sigma$ with  $t=\beta_0(u^2)\in C^\sigma$. Since $f$ is surjective, there is $x\in C$ with $xx^\sigma=t$, and we can define $\beta\colon P\to N$ by $\beta(c)=\beta_0(c)$ and $\beta(uc)=x\sigma\beta_0(c)$ for $c\in P_0$. Since $\beta_0(P_0)\leq C^\sigma$ is centralised by $\sigma$ and by $x\in C$, and since $\beta(u)^2=(x\sigma)^2=xx^\sigma=t=\beta_0(u^2)$, this $\beta$ is a homomorphism with $\beta|_{P_0}=\beta_0$.  Moreover, $\beta(P_0)\leq C^\sigma \leq \ker\epsilon_U$ and $\beta(u)=x\sigma\in N\setminus C$, so $\epsilon_U\circ\beta$ and $\psi$ are homomorphisms $P\to\Theta(U)$ with kernel $P_0$; since $|\Theta(U)|=2$, we have $\epsilon_U\circ\beta=\psi$. A different choice of $x$ replaces $x$ by $xk$ for some $k\in\ker f$. Thus, as shown above,  $k=h^\sigma h^{-1}$ for some $h\in C$, and then the resulting homomorphism is $C$-conjugate to $\beta$. Conversely, every $\beta'$ with $\epsilon_U\circ\beta'=\psi$ and $\beta'|_{P_0}=\beta_0$ satisfies $\beta'(u)=x'\sigma$ for some $x'\in C$ with $x'x'^\sigma=\beta'(u^2)=t$, so it arises from this construction. The claim follows; in particular, $C$-conjugacy of the homomorphisms we seek can be determined by checking equality of their restrictions to $P_0$.

 Since $\xi$ acts either trivially or as $\sigma$, and since $\sigma$ centralises $C^\sigma$, it follows that $\phi^\xi$ and $\phi$ have the same restriction to $P_0$, hence they are $C$-conjugate. Moreover, $\alpha$ restricts to the identity on $P_0$ since $|P_0|\leq 2$, and so $\phi^\xi\circ\alpha$ and $\phi$ have the same restriction to $P_0$. It follows that we always have $\phi^\xi\circ\alpha=\phi$ as $C$-classes. Thus, the condition $\phi^\xi\circ\alpha=\phi$ excludes no class, and so $T(U;P,\xi,\alpha,\psi)$ is just the number of $C$-classes of homomorphisms $\phi$ with $\epsilon_U\circ\phi=\psi$, that is, $T(U;P,\xi,\alpha,\psi)=|\Hom(P_0,C^\sigma)|=|\Hom(P_0,C_{2^a})|$ by the previous paragraph; the last equality holds because $C^\sigma\cong C_{p-1}$ and $P_0$ is a $2$-group. The last claim of a) follows since $P_0$ has index $2$ in $P$ and $|P|\leq 4$. 

\item Let $\psi=1$, so $\phi(P)\leq C$. If $U$ has type (2), then $N=C=\GL_2(p)$, so $\xi=1$ and $T(U;P,\xi,\alpha,1)=\rho(P,\alpha)$ by the definition of $\rho$. For types (1), (3), (4), and (5), the group $C$ is abelian, so $\Hom(P,C)/C=\Hom(P,C)$. The Sylow $2$-subgroups of $C=\GL_1(p)$, $D(p)$, $D(p)$, and $\Sigma(p)$ are $S=C_{2^a}$, $C_{2^a}\times C_{2^a}$, $C_{2^a}\times C_{2^a}$, and $C_{2^{a+\bar a}}$, respectively. Note that $\phi(P)\leq S$, so $\phi\in\Hom(P,S)$.

First let $\xi=1$. This holds for types (1), (3) since they have $\Theta(U)=1$. A preimage of $\xi$ lies in $C$, so the condition $\phi^\xi\circ\alpha=\phi$ becomes $\phi\circ\alpha=\phi$. Recall that $\Hom(P,A\times B)=\Hom(P,A)\times\Hom(P,B)$ for any groups $A,B$, and $\Aut(P)$ acts diagonally. Thus, the number of $\phi\in\Hom(P,S)$ with $\phi\circ\alpha=\phi$ is  $\varphi_a(P,\alpha)$, $\varphi_a(P,\alpha)^2$, $\varphi_a(P,\alpha)^2$, and $\varphi_{a+\bar a}(P,\alpha)$ for types (1), (3), (4), and (5), respectively.

Now consider $\xi\ne1$, and type (4) or (5). For type (4) a preimage of $\xi$ acts on $C=D(p)$ as $J$, hence on $\Hom(P,C_{2^a})^2$ by swapping  the two factors. Thus, $\phi^\xi\circ\alpha=\phi$ if and only if  $\phi=(\phi_1,\phi_2)$ with $\phi_1=\phi_2\circ\alpha$ and $\phi_2=\phi_1\circ\alpha$, and then $\phi$ is determined by any $\phi_1$ with $\phi_1\circ\alpha^2=\phi_1$. Hence there are $\varphi_a(P,\alpha^2)$ such $\phi$. For type (5) a preimage of $\xi$ acts on $C=\Sigma(p)$ as $\theta\colon s\mapsto s^p$. As $\exp(P)$ divides $4$, only the action on the elements of order dividing $4$ is relevant, and there $\theta$ is the identity if $p\equiv1\bmod4$ and inversion if $p\equiv3\bmod4$. Thus, the condition $\phi^\xi\circ\alpha=\phi$ is $\phi\circ\alpha=\phi$ if $\delta=1$, and $\phi\circ\alpha=\phi^{-1}$ otherwise; the claim follows.
\end{iprf}
\end{proof}  

\newpage
\section{Notation}\label{secNot}

For convenience we collect the notation used throughout this paper; here $p$ and $q$ denote primes and $e\in\{1,2\}$. In Section~\ref{secEven}, the sets of Section~\ref{secOdd} are used with $\nu_0$ in place of $\nu$, see Section~\ref{secStrat}.

\vspace*{0.5cm}
  
{\footnotesize
\raggedright
\renewcommand{\arraystretch}{1.1}
\setlength{\LTpre}{0pt}\setlength{\LTpost}{0pt}
\begin{longtable}{@{}p{3.3cm}p{11.9cm}@{}}
  \multicolumn{2}{@{}l@{}}{\bf General notation}\\*[2pt]
 $C_n$ & cyclic group of order $n$ (Section~\ref{sec_prelim})\\
$\gnu(n)$ & number of groups of order $n$, up to isomorphism (Definition~\ref{def_setup}a)\\
$\gnuFF(n)$, $\gnuFF(n,\ell)$ & the same for solvable Frattini-free groups, with socle order $\ell$ (Definition~\ref{def_setup}b,c)\\
$\mathcal{Q}(n)$ & product of all primes $p$ with $p^2\mid n$ (Definition~\ref{def_setup}d)\\
$\mathcal{S}(n)$ & orders of the possible simple direct factors (Definition~\ref{def_setup}e), including $1$\\
$\mathcal{A}(n)$ & isomorphism types of abelian groups of order $n$ (Definition~\ref{def_setup}f)\\
$\Delta_P$ & generalised Kronecker delta: $1$ if $P$ holds, and $0$ otherwise (Section~\ref{sec_prelim})\\
$G_p$, $\rk(G)$ & Sylow $p$-subgroup of $G$; minimal size of a generating set (Section~\ref{sec_prelim})\\
$v_q$ & $q$-adic valuation (Notation~\ref{notSetup}c)\\[7pt]
\multicolumn{2}{@{}l@{}}{\bf Socle, columns, and the odd case}\\*[2pt]
$\ell=p_1^{e_1}\ldots p_m^{e_m}$ & socle order; each $e_i\in\{1,2\}$ (Notation~\ref{notSetup}a)\\
$S$, $\Aut(S)$ & socle $\prod_iC_{p_i}^{e_i}$, and $\Aut(S)=\prod_i\GL_{e_i}(p_i)$ (Notation~\ref{notSetup}b)\\
$\pi_i$, $K_i$ & projection onto the $i$-th column of $\Aut(S)$, and $K_i=\pi_i(K)$ (Notation~\ref{notSetup}b)\\
$\nu$, $\nu_0$, $t$ & $\nu=n/\ell=2^t\nu_0$ with $\nu_0$ odd (Notations~\ref{notSetup}c and~\ref{notEvenN}a)\\
$\mathcal{K}$ & the solvable $K\leq\Aut(S)$ of order $\nu$ (Notation~\ref{notSetup}d)\\
$C(p,b)$ & subgroup of order $b$ of $\GL_1(p)$ (Definition~\ref{defCan}a)\\
$D(p)$, $J$, $M(p)$ & diagonal, swap, and monomial matrices in $\GL_2(p)$ (Definition~\ref{defCan}b)\\
$\Sigma(p)$, $\Sigma(p,b)$ & standard Singer cycle, and its subgroup of order $b$ (Definition~\ref{defCan}c)\\
$\mathcal{U}(p,e)$, $\mathcal{U}(p,e,L)$ & canonical representatives (Definition~\ref{def_Upc}); those that are quotients of $L$ \eqref{eq_defUpel}\\
$\Theta(U)$, $\theta_U$ & $N_{\GL_e(p)}(U)/C_{\GL_e(p)}(U)$, of order at most $2$ (Lemma~\ref{lemT})\\
$\mathcal{U}=(U_1,\ldots,U_m)$ & projection tuple; note the overload with $\mathcal{U}(p,e)$ (Section~\ref{secReduce})\\
$\Theta(\mathcal{U})$ & $\Theta(U_1)\times\ldots\times\Theta(U_m)$, an elementary abelian $2$-group \eqref{def_ThetaU}\\
$\mathcal{K}(\mathcal{U})$, $\mathcal{K}(\mathcal{U},L)$ & the $K\in\mathcal{K}$ with $\pi_i(K)=U_i$; those with $K\cong L$ (Section~\ref{secReduce}, Corollary~\ref{corRed})\\
$\mathcal{K}_q(\mathcal{U},L)$ & subdirect products of the $(U_i)_q$ isomorphic to $L_q$ \eqref{eq_defKq}\\
$j_i$, $d_i$, $d_i^\pm$ & $v_q(|X_i|)$, $\rk(X_i)$, $\rk(X_i^\pm)$ for $X_i=(U_i)_q$ (Section~\ref{secStep1}, Lemma~\ref{lemFix})\\
$g$, $c$, $r$, $z$ & auxiliary counting functions on $\{0,1,2\}$ (Section~\ref{secStep1})\\
$\Lambda^\pm$, $\Lambda^{(\chi)}$ & counts attached to the eigenspace ranks (Lemmas~\ref{lemFix} and~\ref{lemFixHS})\\[7pt]
\multicolumn{2}{@{}l@{}}{\bf The even case and the technical arguments}\\*[2pt]
$N_i$, $Z_i$, $\epsilon_{U}$ & normaliser and centraliser of $U_i$; the map $N_i\to\Theta(U_i)$ (Notation~\ref{notEvenN}c)\\
$N(\mathcal{U})$, $C(\mathcal{U})$, $\epsilon$ & $\prod_iN_i$, $\prod_iZ_i$, and $\epsilon\colon N(\mathcal{U})\to\Theta(\mathcal{U})$ (Notation~\ref{notEvenN}d)\\
$E$, $O$ & Sylow $2$-subgroup and normal Hall $2'$-subgroup of $K=E\ltimes O$ (Lemma~\ref{lem_kHall})\\
$\mathcal{P}$ & pairs $(E,O)$ with $|E|=2^t$ and $|O|=\nu_0$ (before Proposition~\ref{propPairs})\\
$\mathcal{E}_t(\mathcal{U})$ & $C(\mathcal{U})$-classes of subgroups of order $2^t$ in $N(\mathcal{U})$ \eqref{eq_mcE}\\
$\mathcal{M}(\mathcal{U},L)$ & the pairs in $\mathcal{P}$ compatible with $\mathcal{U}$ and $L$ (before Proposition~\ref{propRedS})\\
column type ${\rm(0)}$--${\rm(5)}$ & one of six cases for $U$, $C$, and $N$ (Definition~\ref{defTypes})\\
$\mathcal{K}(\mathcal{U},L)^H$ & the $H$-invariant elements of $\mathcal{K}(\mathcal{U},L)$ (after Definition~\ref{defTypes}; Lemma~\ref{lemFixHS})\\
$\mathcal{I}(U)$ & $C$-classes of elements of $N$ of order dividing $2$ (Section~\ref{secT1})\\
$A(U;\xi,\psi)$ & counting function for the case $t=1$ (Section~\ref{secT1}, Table~\ref{tab_Afirst})\\
$\mathcal{H}(E)$ & tuples of $Z_i$-classes of maps $E\to N_i$ with trivial common kernel (before \eqref{eq_newpsi})\\
$A(\mathcal{U};E,\xi,\alpha,\psi)$ & counting function for the case $t=2$ \eqref{eq_newA}\\
$T(U;P,\xi,\alpha,\psi)$ & columnwise version of $A(\mathcal{U};E,\xi,\alpha,\psi)$ \eqref{eq_defT}\\
$\chi$, $W_\chi$, $\mathcal{Y}$, $d_i^{(\chi)}$ & data for the statement and proof of Lemma~\ref{lemFixHS}\\
$\mu$ & M\"obius function of the subgroup lattice of $E$ (Lemma~\ref{lemHcountS})\\
$\varphi_c$, $\varphi_c^-$, $\rho$ & auxiliary homomorphism and representation counts (Definition~\ref{def_table2stuff}, Table~\ref{tabPhiS})\\
$a$, $\bar a$, $\delta$ & $v_2(p-1)$, $v_2(p+1)$, and $\Delta_{4\mid(p-1)}$ (Definition~\ref{def_table2stuff})
\end{longtable}
}

\newpage
 
{\bf Data availability}\\
The only data connected to this publication is a GAP implementation that has resulted from the research developed in this publication. This implementation is made available via a public GitHub repository, see the citation in the text. There is no further data connected to this publication.

\medskip

    {\bf Conflict of interest}\\
    The authors state  that there is no conflict of interest.
    
 \medskip

\enlargethispage{1cm}
\bibliographystyle{abbrv}

\end{document}